\documentclass[12pt,a4paper]{amsart}
\usepackage{mathrsfs}

\usepackage{amsfonts}
\usepackage{amsmath}
\usepackage{amssymb}
\usepackage{enumerate}
\usepackage{hyperref}
\usepackage{latexsym}
\usepackage{xcolor}
\usepackage[shortlabels]{enumitem}
\setlist[enumerate,1]{label={\upshape(\roman*)}}
\usepackage[numbers,sort&compress]{natbib}
\usepackage{tikz}

 \makeatletter
    
    \newcommand{\Rmnum}[1]
    {\expandafter\@slowromancap\romannumeral #1@}
    \makeatother

\newtheorem{thm}{Theorem}[section]

\newtheorem{prop}[thm]{Proposition}
\newtheorem{lemma}[thm]{Lemma}
\newcounter{foo}[subsection]
\newcounter{fooo}[section]
\newcounter{foooo}[section]
\newcounter{fooooo}[section]
\newtheorem{step}[fooo]{Step}
\newtheorem{stepp}[foo]{Step}
\newtheorem{steppp}[foooo]{Step}
\newtheorem{stepppp}[fooooo]{Step}

\newtheorem{cor}[thm]{Corollary}

\newtheorem{example}[thm]{Example}
\newtheorem{defin}[thm]{Definition}

\theoremstyle{definition}
\newtheorem{remark}[thm]{Remark}

\title[
]{Weakly distance-regular digraphs whose underlying graphs are distance-regular, \Rmnum{2}}

\date{}

\thanks{*Corresponding author}

\author[
]{Qing Zeng}
\address{{Department of Mathematical Sciences, Tsinghua University, Beijing 100084, China, and Laboratory of Mathematics and Complex Systems (MOE),~School of Mathematical Sciences\\Beijing Normal University\\Beijing 100875\\China}}
\email{qingz@mail.bnu.edu.cn}

\author[
]{Yuefeng Yang*}
\address{School of Science\\China University of Geosciences\\Beijing 100083\\China}
\email{yangyf@cugb.edu.cn}

\author[
]{Kaishun Wang}
\address{Laboratory of Mathematics and Complex Systems (MOE),~School of Mathematical Sciences\\Beijing Normal University\\Beijing 100875\\China}
\email{wangks@bnu.edu.cn}

\begin{document}

\begin{abstract}

Weakly distance-regular digraphs are a natural directed version of distance-regular graphs. 
In \cite{YZW23}, we classified all commutative weakly distance-regular digraphs whose underlying graphs are Hamming graphs, folded $n$-cubes, or Doob graphs.
In this paper, we  classify all commutative weakly distance-regular digraphs whose underlying graphs are Johnson graphs or folded Johnson graphs.
\end{abstract}



\maketitle
\section{Introduction}

A \emph{digraph} $\Gamma$ is a pair $(V(\Gamma),A(\Gamma))$, where $V(\Gamma)$ is a finite nonempty set of vertices and $A(\Gamma)$ is a set of ordered pairs ({\em arcs}) $(x,y)$ with distinct vertices $x$ and $y$. For any arc $(x,y)\in A(\Gamma)$, if $A(\Gamma)$ also contains an arc $(y,x)$, then $\{(x,y),(y,x)\}$ can be viewed as an {\em edge}.
We say that $\Gamma$ is an \emph{undirected graph} or a {\em graph} if $A(\Gamma)$ is a symmetric relation. For a digraph $\Gamma$, we form the {\em underlying graph} of $\Gamma$ with the same vertex set, and there is an edge between vertices $x$ and $y$ whenever $(x,y)\in A(\Gamma)$ or $(y,x)\in A(\Gamma)$. 

A \emph{path} of length $r$ from a vertex $x$ to a vertex $y$ in $\Gamma$ is a finite sequence of vertices $(x=w_{0},w_{1},\ldots,w_{r}=y)$ such that $(w_{t-1}, w_{t})\in A(\Gamma)$ for $1\leq t\leq r$.   A digraph (resp. graph) is said to be \emph{strongly connected} (resp. \emph{connected}) if, for any two vertices $x$ and $y$, there is a path from $x$ to $y$. A path $(w_{0},w_{1},\ldots,w_{r})$ is called a \emph{circuit} of length $r+1$ if $(w_{r},w_0)\in A(\Gamma)$. The \emph{girth} of $\Gamma$ is the length of a shortest circuit in $\Gamma$.

The length of a shortest path from $x$ to $y$ is called the \emph{distance} from $x$ to $y$ in $\Gamma$, denoted by $\partial_\Gamma(x,y)$. The maximum value of the distance function in $\Gamma$ is called the \emph{diameter} of $\Gamma$. For $0\leq i\leq d$,  we define $\Gamma_{i}$ to be the set of ordered pairs $(x,y)$ and $\Gamma_{i}(x)$ to be the set of vertices $y$ with $\partial_{\Gamma}(x,y)=i$, where $d$ is the diameter of $\Gamma$. Let $\tilde{\partial}_{\Gamma}(x,y):=(\partial_{\Gamma}(x,y),\partial_{\Gamma}(y,x))$ be the \emph{two-way distance} from $x$ to $y$, and $\tilde{\partial}(\Gamma):=\{\tilde{\partial}_{\Gamma}(x,y)\mid x,y\in V(\Gamma)\}$ the \emph{two-way distance set} of $\Gamma$. For any $\tilde{i}\in\tilde{\partial}(\Gamma)$, we define $\Gamma_{\tilde{i}}$ to be the set of ordered pairs $(x,y)$ with $\tilde{\partial}_{\Gamma}(x,y)=\tilde{i}$. 


In \cite{KSW03}, Wang and Suzuki proposed a natural directed version of a distance-regular graph without bounded diameter, i.e., a weakly distance-regular digraph. A strongly connected digraph $\Gamma$ is said to be \emph{weakly distance-regular}  if the configuration $\mathfrak{X}(\Gamma)=(V(\Gamma),\{\Gamma_{\tilde{i}}\}_{\tilde{i}\in\tilde{\partial}(\Gamma)})$ is a non-symmetric association scheme. We call $\mathfrak{X}(\Gamma)$ the \emph{attached scheme} of $\Gamma$. A weakly distance-regular digraph is \emph{commutative} if its attached scheme is commutative. 
For more information, see \cite{MW24,YL24,YF22,KSW03,HS04,KSW04,YYF16,YYF18,YYF20,YYF22,ZYW22,YZW23}. 


Wang and Suzuki \cite{KSW03} proposed a question when an orientation of a distance-regular graph defines a weakly distance-regular digraph. In \cite{YZW23},  we initiated this project, and classified all commutative weakly distance-regular digraphs whose underlying graphs  are Hamming graphs, folded $n$-cubes, or Doob graphs. We also proved that the  commutative weakly distance-regular digraphs whose underlying graphs are complete graphs have diameter $2$ and girth $g\leq3$ (See \cite[Proposition 3.5]{YZW23}). In this paper, we continue this project, and classify all commutative weakly distance-regular digraphs whose underlying graphs are Johnson graphs or folded Johnson graphs of diameter at least two.




\begin{thm}\label{main}
Let $\Gamma$ be a commutative weakly distance-regular digraph. 
Then $\Gamma$ has a Johnson graph or a folded Johnson graph of diameter at least two as its underlying graph if and only if $\Gamma$ is isomorphic to ${\rm Cay}(\mathbb{Z}_6,\{1,2\})$ or ${\rm Cay}(\mathbb{Z}_6,\{1,4\})$.
\end{thm}

\begin{remark}
According to Theorem \ref{main},  no commutative weakly distance-regular digraph has a  folded Johnson graph of diameter at least two as its underlying graph.
\end{remark}


The remainder of this paper is organized as follows. In Section 2, we provide the required notation, concepts, and preliminary results for association schemes and distance-regular graphs. Let $\Gamma$  be a commutative weakly distance-regular digraph with a distance-regular graph $\Sigma$ of diameter at least two as its underlying graph. In Section 3, we discuss the induced subdigraph of $\Gamma$ on $\{x,z\}\cup(\Sigma_1(x)\cap\Sigma_1(z))$ with $(x,z)\in\Gamma_2$ under the certain property on $\Sigma$, and give some results which are used frequently in the sequel. Then we set $\Sigma$ to be a Johnson graph or a folded Johnson graph. In Section 4, applying the results in Section 3, we  give all the possibilities of $T$, that is $T=\{3\}$,  $T=\{3,4\}$ or $T=\{2,3,4\}$, where $T=\{q+1\mid (1,q)\in\tilde{\partial}(\Gamma)\}$. In Section 5, we show that $\Sigma$ is $J(m,2)$ if $T=\{3\}$. In Section 6, applying the results in Section 5, we prove that $\Gamma$ is ${\rm Cay}(\mathbb{Z}_6,\{1,4\})$ when $T=\{3\}$; and $\Gamma$ is ${\rm Cay}(\mathbb{Z}_6,\{1,2\})$ when $T=\{3,4\}$ or $T=\{2,3,4\}$. Therefore, we complete the proof of Theorem \ref{main}. 

\section{Preliminaries}

In this section, we present some notation, concepts, and basic results for association schemes and distance-regular graphs that we shall use in this paper.

\subsection{Association schemes}

A \emph{$d$-class association scheme} $\mathfrak{X}$ is a pair $(X,\{R_{i}\}_{i=0}^{d})$, where $X$ is a finite set, and each $R_{i}$ is a
nonempty subset of $X\times X$ satisfying the following axioms (see \cite{EB84,PHZ96,PHZ05} for a background of the theory of association schemes):
\begin{itemize}
\item [{\rm(i)}] $R_{0}=\{(x,x)\mid x\in X\}$ is the diagonal relation;

\item [{\rm(ii)}] $X\times X=R_{0}\cup R_{1}\cup\cdots\cup R_{d}$, $R_{i}\cap R_{j}=\emptyset~(i\neq j)$;

\item [{\rm(iii)}] for each $i$, $R_{i}^{\top}=R_{i^{*}}$ for some $0\leq i^{*}\leq d$, where $R_{i}^{\top}=\{(y,x)\mid(x,y)\in R_{i}\}$;

\item [{\rm(iv)}] for all $i,j,l$, the cardinality of the set $$P_{i,j}(x,y):=R_{i}(x)\cap R_{j^{*}}(y)$$ is constant whenever $(x,y)\in R_{l}$, where $R(u)=\{v\mid (u,v)\in R\}$ for $R\subseteq X\times X$ and $u\in X$. This constant is denoted by $p_{i,j}^{l}$.
\end{itemize}
A $d$-class association scheme is also called an association scheme with $d$ classes (or even simply a scheme). The integers $p_{i,j}^{l}$ are called the \emph{intersection numbers} of $\mathfrak{X}$. We say that $\mathfrak{X}$ is \emph{commutative} if $p_{i,j}^{l}=p_{j,i}^{l}$ for all $0\leq i,j,l\leq d$. The subsets $R_{i}$ are called the \emph{relations} of $\mathfrak{X}$. For each $i$, the integer $k_{i}=p_{i,i^{*}}^{0}$ is called the \emph{valency} of $R_{i}$.   A relation $R_{i}$ is called \emph{symmetric} if $i=i^{*}$, and \emph{non-symmetric} otherwise. An association scheme is called \emph{symmetric} if all relations are symmetric, and \emph{non-symmetric} otherwise. We say that $\mathfrak{X}$ is \emph{primitive} if the digraph $(X,R_i)$ is strongly connected for all $1\leq i\leq d$, and \emph{imprimitive} otherwise.


Now we list the basic properties of intersection numbers.

\begin{lemma}\label{intersection numners}
{\rm (\cite[Chapter \Rmnum{2}, Proposition 2.2]{EB84})} Let $(X,\{R_{i}\}_{i=0}^{d})$ be an association scheme. For all $h,i,j,l,m\in\{0,1,\ldots,d\},$ the following hold{\rm:}
\begin{enumerate}
\item\label{intersection numners 1} $k_{i}k_{j}=\sum_{h=0}^{d}p_{i,j}^{h}k_{h}$;

\item\label{intersection numners 2} $p_{i,j}^{l}k_{l}=p_{l,j^{*}}^{i}k_{i}=p_{i^{*},l}^{j}k_{j}$;


\item\label{intersection numners 4} $\sum_{r=0}^{d}p_{i,l}^{r}p_{m,r}^{j}=
    \sum_{t=0}^{d}p_{m,i}^{t}p_{t,l}^{j}$. 
\end{enumerate}
\end{lemma}

\subsection{Distance-regular graphs}

A connected graph $\Gamma$ of diameter $d$ is said to be  \emph{distance-regular}  if there are integers $b_{i-1}$, $c_i$ for $1\leq i\leq d$ such that there are exactly $c_i$ neighbours of $y$ in $\Gamma_{i-1}(x)$ and $b_i$ neighbours of $y$ in $\Gamma_{i+1}(x)$ for any $(x,y)\in\Gamma_{i}$. This definition implies the existence of numbers $a_i$ such that there are exactly $a_i$ neighbours of $y$ in $\Gamma_{i}(x)$ for $0\leq i\leq d$. They do not depend on $x$ and $y$ because $b_0$ is the valency of $\Gamma$ and $a_i$
is determined in terms of $b_i$ and $c_i$ by the equation $b_{0}=c_{i}+a_{i}+b_{i}$, where $c_0=b_d=0$. We observe $a_0=0$ and $c_1=1$. The array $\iota(\Gamma)=\{b_0,b_1,\ldots,b_{d-1}; c_1,c_2,\ldots,c_d\}$ is called the \emph{intersection array} of $\Gamma$. An important property of distance-regular graphs is that $(V(\Gamma),\{\Gamma_{i}\}_{i=0}^{d})$ is a symmetric association scheme. We call a distance-regular graph $\Gamma$ \emph{primitive} if the association scheme $(V(\Gamma),\{\Gamma_{i}\}_{i=0}^{d})$ is primitive, and \emph{imprimitive} otherwise. One can easily observe that $a_{i}=p_{1,i}^{i}$ $(0\leq i\leq d)$, $b_{i}=p_{1,i+1}^{i}$ $(0\leq i\leq d-1)$ and $c_{i}=p_{1,i-1}^{i}$ $(1\leq i\leq d)$. For more detailed results on such special graphs, we refer the reader to \cite{NB74,AEB98,DKT16}.

The Johnson graph is a classical example of distance-regular graphs. Johnson graphs have found their applications in various topics of combinatorial geometry, Ramsey theory, coding theory, and other branches of modern combinatorics. The definition of a  Johnson graph is as follows. Let $X$ be an $n$-set. The {\em Johnson graph} $J(n,e)$ of the $e$-sets in $X$ has vertex set ${X\choose e}$, the collection of $e$-subsets of $X$. Two vertices $x$ and $y$ are adjacent whenever $|x\cap y|=e-1$.
Note that $J(n,e)$ is isomorphic to $J(n,n-e)$, and $J(n,1)$ is a clique. Therefore we restrict to $n\geq2e$ and $e\geq2$.
The following theorem determines the intersection numbers of a Johnson graph.

\begin{thm}\label{johnson}
{\rm (\cite[Theorem 9.1.2]{AEB98})} Let $n$ and $e$ be two positive integers with $n\geq2e$ and $e\geq2$. Then, the Johnson graph $J(n,e)$ is  a distance-regular graph of diameter $e$, and has intersection array given by $b_i=(e-i)(n-e-i)$, $c_i=i^2$ and $a_i=i(n-2i)$ for $0\leq i\leq e$.
\end{thm}

The only imprimitive Johnson graphs are the graphs $J(2e,e)$ which are antipodal, they are double covers  of the {\em folded Johnson graphs} $\bar{J}(2e,e)$ (see \cite[Appendix A.5]{AEB98} for definitions of antipodal distance-regular graphs and double covers). The vertices of  $\bar{J}(2e,e)$ are the partitions of a $2e$-set $X$ into two $e$-sets, two partitions being adjacent whenever their common refinement is a partition of $X$ into four sets of sizes $1$, $e-1$, $1$, $e-1$. The vertex $\{x,X\setminus x\}$ in $\bar{J}(2e,e)$ can be represented as $x$, where  $x\in{X\choose e}$.
Then $x$ and $y$ are adjacent whenever $|x\cap y|\in\{1,e-1\}$. Note that  $\bar{J}(2e,e)$ is a clique if $e\leq3$. Therefore we restrict to  $e\geq4$. The following theorem determines the intersection numbers of a folded Johnson graph.

\begin{thm}\label{folded johnson}
{\rm (\cite[Proposition 9.1.5]{AEB98})} 
Let $e$  be a positive integer with  $e\geq4$.  Then the
folded Johnson graph $\bar{J}(2e,e)$ is distance-regular with diameter $d=[e/2]$,
intersection array given by $b_i=(e-i)^2$, $c_i=i^2$ and $a_i=i(2e-2i)~(0\leq i\leq e/2)$
 but $c_d=2d^2$ and $a_d=e^2-2d^2$ if $e$ is even.
\end{thm}

In the remainder of this section, we assume that  $\Sigma$ denotes a Johnson graph $J(n,e)$ with $n\geq2e$ and $e\geq2$, or a folded Johnson graph $\bar{J}(2e,e)$ with $e\geq4$. Let $m=n$ if $\Sigma$ is a Johnson graph $J(n,e)$, and $m=2e$ if $\Sigma$ is a folded Johnson graph $\bar{J}(2e,e)$. Then we may set ${X\choose e}$ to be the vertex set of $\Sigma$, where $X$ is an $m$-set. 
Let $x$ be an $e$-subset of $X$. Denote $x(\alpha,\beta)=(x\backslash\alpha)\cup\beta$ for $\alpha\subseteq x$ and $\beta\subseteq X\setminus x$. For convenience's sake, we write $x(a,b)$ instead of $x(\{a\},\{b\})$, where $a\in x$ and $b\in X\setminus x$.
For $y=x(a_1,b_1)\in\Sigma_1(x)$ with $a_1\in x$ and $b_1\in X\setminus x$, we set $Y_1(x,y)=\{x(a,b_1)\mid a\in x\setminus\{a_1\}\}$ and $Y_2(x,y)=\{x(a_1,b)\mid b\in X\setminus(x\cup\{b_1\})\}$.

Now we give some preliminary results for $\Sigma$.

\begin{lemma}\label{sigma1}
Suppose $(x,y)\in\Sigma_1$. Then $\Sigma_1(x)\cap\Sigma_1(y)=Y_1(x,y)\cup Y_2(x,y)$, $|Y_1(x,y)|=e-1$ and $|Y_2(x,y)|=m-e-1$. Moreover, the following hold:
\begin{enumerate}
\item\label{sigma11}  The distance in $\Sigma$ between two distinct vertices in  $Y_i(x,y)$ is $1$ for $i=1,2$;

\item\label{sigma12} The distance in $\Sigma$ between a vertex in $Y_1(x,y)$ and a vertex in $Y_2(x,y)$ is $2$.
\end{enumerate}
%
\end{lemma}
\begin{proof}
It is immediate from the definitions of a Johnson graph and a folded Johnson graph.
\end{proof}

\begin{lemma}\label{iwuguan}
Let $(x,y)\in\Sigma_1$. Then $Y_i(x,y)=Y_i(y,x)$ for $i\in\{1,2\}$.
\end{lemma}
\begin{proof}
Since $(x,y)\in\Sigma_1$, there exist $a_1\in x$ and $b_1\in X\setminus x$ such that $y=x(a_1,b_1)$, which implies $x=y(b_1,a_1)$. Then
\begin{align}
Y_1(y,x)=\{y(a,a_1)\mid a\in y\setminus\{b_1\}\}&=\{x(a,b_1)\mid a\in x\setminus\{a_1\}\}=Y_1(x,y),\nonumber\\
Y_2(y,x)=\{y(b_1,b)\mid b\in X\setminus(y\cup\{a_1\})\}&=\{x(a_1,b)\mid b\in X\setminus(x\cup\{b_1\})\}=Y_2(x,y).\nonumber
\end{align}

This completes the proof of the lemma.
\end{proof}

\begin{lemma}\label{xyz}
Let $(x,y)\in\Sigma_1$ and $\{i,j\}=\{1,2\}$. Suppose $z\in Y_i(x,y)$. Then $Y_i(x,y)\cup\{y\}=Y_i(x,z)\cup\{z\}$ and $Y_j(x,y)\cap Y_h(x,z)=\emptyset$ for all $h\in\{1,2\}$.
\end{lemma}
\begin{proof}
Since $(x,y)\in\Sigma_1$, there exist $a_1\in x$ and $b_1\in X\setminus x$ such that $y=x(a_1,b_1)$, which implies $$Y_1(x,y)=\{x(a,b_1)\mid a\in x\setminus\{a_1\}\}~\mbox{and}~ Y_2(x,y)=\{x(a_1,b)\mid b\in X\setminus(x\cup\{b_1\})\}.$$ Since the proofs of $z\in Y_1(x,y)$ and $z\in Y_2(x,y)$ are similar, we only need to consider the case   $z\in Y_1(x,y)$. Then there exists $a_2\in x\setminus\{a_1\}$ such that $z=x(a_2,b_1)$, and so $$Y_1(x,z)=\{x(a,b_1)\mid a\in x\setminus\{a_2\}\}~\mbox{and}~ Y_2(x,z)=\{x(a_2,b)\mid b\in X\setminus(x\cup\{b_1\})\}.$$ Then $Y_1(x,y)\cup\{y\}=Y_1(x,z)\cup\{z\}$ and $Y_2(x,y)\cap Y_h(x,z)=\emptyset$ for all $h\in\{1,2\}$. 
\end{proof}

\section{General results}




In the remainder of this paper, we always assume that $\Gamma$ is a commutative weakly distance-regular digraph whose underlying graph $\Sigma$ is  distance-regular.

To distinguish the notation between $\Gamma$ and $\Sigma$, let $p_{\tilde{i},\tilde{j}}^{\tilde{h}},P_{\tilde{i},\tilde{j}}(x,y)$ denote the symbols belonging to $\Gamma$ with $\tilde{h},\tilde{i},\tilde{j}\in\tilde{\partial}(\Gamma)$ and $x,y\in V(\Gamma)$, and $a_i,c_i$ the symbols belonging to $\Sigma$ with $0\leq i\leq d$, where $d$ is the diameter of $\Sigma$.

An arc $(x,y)$ of $\Gamma$ is of {\em type} $(1,q)$ if $\partial_{\Gamma}(y,x)=q$. 
A path (circuit) is called \emph{pure} if it consists of one type of arcs, and \emph{mixed} otherwise. 
We say that $(w_{0},w_{1},\ldots,w_{r})$ is a \emph{non-path} if neither $(w_{0},w_{1},\ldots,w_{r})$ nor $(w_{r},w_{r-1},\ldots,w_{0})$ is a path in $\Gamma$.

Let $(x,z)\in\Sigma_1\cup\Sigma_2$ and $C(x,z)=\Sigma_1(x)\cap\Sigma_1(z)$. If $y\in C(x,z)$, then there are six cases:
\begin{itemize}
\item [{\rm C1)}] $y\in P_{(1,p),(1,p)}(x,z)$, i.e., $(x,y,z)$ is a  pure path;
\item [{\rm C2)}] $y\in P_{(p,1),(p,1)}(x,z)$, i.e., $(z,y,x)$ is a  pure path;
\item [{\rm C3)}] $y\in P_{(1,p),(1,q)}(x,z)$, i.e., $(x,y,z)$ is a  mixed path;
\item [{\rm C4)}] $y\in P_{(q,1),(p,1)}(x,z)$, i.e., $(z,y,x)$ is a  mixed path;
\item [{\rm C5)}] $y\in P_{(1,r),(s,1)}(x,z)$, i.e., $(x,y,z)$ is a non-path with $(x,y)\in A(\Gamma)$;
\item [{\rm C6)}] $y\in P_{(s,1),(1,r)}(x,z)$, i.e., $(x,y,z)$ is a non-path with $(y,x)\in A(\Gamma)$.
\end{itemize}
Here, $p,q \geq 1$ with $p\neq q$ and $r,s\geq 2$.

Set $(x,z)\in\Gamma_{\tilde{h}}$ with $\tilde{h}\in\tilde{\partial}(\Gamma)$. By Cases C1--C6, we have
\begin{align}\label{a1c2}\tag{3.1}
\sum_{1\in\{p_1,p_2\},1\in\{q_1,q_2\}}p^{\tilde{h}}_{(p_1,p_2),(q_1,q_2)}=
\begin{cases}
a_1,~(x,z)\in\Sigma_1,\\
c_2,~(x,z)\in\Sigma_2.
\end{cases}
\end{align}

In the remainder of this section, we always assume that $\Sigma$ has the following property.
\begin{itemize}
\item For any  $(x,z)\in\Sigma_2$, the induced subgraph on $\{x,z\}\cup C(x,z)$ is isomorphic to $J(4,2)$.
\end{itemize}

Next we discuss the induced subdigraph of $\Gamma$ on $\{x,z\}\cup C(x,z)$ with $(x,z)\in\Sigma_2$, and give some results which are used frequently in the sequel.
 By the property of $\Sigma$, we have $c_2=4$.

The following lemma is immediate from the property of $\Sigma$.

\begin{lemma}\label{sigma2}
Let $(x,z)\in\Sigma_2$ and $y_1\in C(x,z)$. Then there exists a unique vertex $y_2\in C(x,z)$ such that  $(y_1,y_2)\in \Sigma_2$.  Moreover, if $C(x,z)=\{y_1, y_2, y_3, y_4\}$, then $(y_3,y_4)\in\Sigma_2$,  $C(y_1,y_2)=\{x,z,y_3,y_4\}$ and $C(y_3,y_4)=\{x,z,y_1,y_2\}$.
\end{lemma}

For $(a,b)\in\tilde{\partial}(\Gamma)$, we write $\Gamma_{a,b}$ (resp. $k_{a,b}$) instead of $\Gamma_{(a,b)}$  (resp. $k_{(a,b)}$).



%
%
%


\begin{lemma}\label{fenlei}
Suppose $(x,z)\in\Gamma_{2,i}$ for some $i\geq3$ and $ C(x,z) = \{y_1, y_2, y_3, y_4\}$. Then exactly one of the following holds:
\begin{enumerate}
\item\label{fenlei2} four pure paths from $x$ to $z$, i.e.,
$y_j \in P_{(1,p_j), (1, p_j)}(x,z) \text{ with } p_j \geq 2$ for $1\leq j\leq 4$;

\item\label{fenlei3} four mixed paths from $x$ to $z$, i.e., by reordering vertices $y_1, y_2, y_3, y_4$,
$y_1 \in P_{(1,p), (1, q)}(x,z)$,
$y_2 \in P_{(1,q), (1, p)}(x,z)$, $y_3 \in P_{(1,r), (1, s)}(x,z)$, $y_4 \in P_{(1,s), (1, r)}(x,z)$ with $p, q, r,s \geq 1$, $p\neq q$, and $r\neq s$;
\item\label{fenlei4} two mixed paths and two non-paths from $x$ to $z$, i.e., by reordering vertices $y_1, y_2, y_3, y_4$,
$y_1 \in P_{(1,p), (1, q)}(x,z)$,
$y_2 \in P_{(1,q), (1, p)}(x,z)$, $y_3 \in P_{(1,r), (s,1)}(x,z)$,  $y_4 \in P_{(s,1), (1,r)}(x,z)$ with $r, s \geq 2$, $p, q\geq1$,  and $p\neq q$;
\item\label{fenlei5} two pure paths and two mixed paths from $x$ to $z$, i.e., by reordering vertices $y_1, y_2, y_3, y_4$,
$y_1 \in P_{(1,p), (1, p)}(x,z)$, $y_2 \in P_{(1,q), (1, q)}(x,z)$, $y_3 \in P_{(1,r), (1,s)}(x,z)$,  $y_4 \in P_{(1,s), (1,r)}(x,z)$ with $p, q \geq 2$, $r, s\geq 1$ and $r\neq s$;
\item\label{fenlei6} two pure paths and two non-paths from $x$ to $z$, i.e., by reordering vertices $y_1, y_2, y_3, y_4$,
$y_1 \in P_{(1,p), (1, p)}(x,z)$, $y_2 \in P_{(1,q), (1, q)}(x,z)$, $y_3 \in P_{(1,r), (s,1)}(x,z)$, $y_4 \in P_{(s,1), (1,r)}(x,z)$ with $p, q,r,s \geq 2$.
\end{enumerate}
Moreover, in the case \ref{fenlei5}, $(y_1, y_2), (y_3, y_4)\in \Sigma_2$, and in the case \ref{fenlei6}, if $p \geq 3$ or $q \geq 3$, then $(y_1, y_2), (y_3, y_4)\in \Sigma_2$.
\end{lemma}
\begin{proof}
By the commutativity of $\Gamma$, we get $|P_{(1,p),(1,q)}(x,z)|=|P_{(1,q),(1,p)}(x,z)|$ and $|P_{(1,r),(s,1)}(x,z)|=|P_{(s,1),(1,r)}(x,z)|$ for $r, s \geq 2$ and $p, q\geq1$ with $p\neq q$.  Since $\partial_{\Gamma}(z,x)=i\geq3$, one has $C(x,z)\cap(P_{(1,1),(s,1)}(x,z)\cup P_{(s,1),(1,1)}(x,z))=\emptyset$ for all $s\geq1$. By Cases C1--C6, the first statement holds. Then we only need to prove the second statement.

Suppose that \ref{fenlei5} holds. 
By Lemma \ref{sigma2}, $(y_1, y_2)\in \Sigma_2$ is equivalent to $(y_3, y_4)\in \Sigma_2$. Hence, by way of contradiction, we may assume that $(y_1,y_2)\in \Sigma_1$. Then $(y_1,y_3)\in \Sigma_2$ or $(y_1,y_4)\in \Sigma_2$. By replacing the roles of $r$ and $s$, we may assume that $(y_1,y_3)\in \Sigma_2$. Then $(y_2,y_4)\in \Sigma_2$.

Since $(y_1,y_3)\in \Sigma_2$, from Lemma \ref{sigma2}, we get $C(y_1, y_3) = \{x,z,y_2,y_4\}$. Since $x\in P_{(p,1),(1,r)}(y_1, y_3)$ and $z\in P_{(1,p),(s,1)}(y_1, y_3)$ with $r \neq s$, from the commutativity of $\Gamma$, we have
\begin{align}
y_2\in P_{(1,r),(p,1)}(y_1, y_3)  &\mbox{ and } y_4\in P_{(s,1),(1,p)}(y_1, y_3), \mbox{ or} \tag{3.2}\label{eq:2-1-1}\\
y_2\in P_{(s,1),(1,p)}(y_1, y_3)  &\mbox{ and } y_4\in P_{(1,r),(p,1)}(y_1, y_3).\tag{3.3}\label{eq:2-1-2}
\end{align}

Since
$(y_2,y_4)\in \Sigma_2$, from Lemma \ref{sigma2}, one has $C(y_2, y_4) = \{x,z,y_1,y_3\}$. Since $x\in P_{(q,1),(1,s)}(y_2, y_4)$ and $z\in P_{(1,q),(r,1)}(y_2, y_4)$, from the commutativity of $\Gamma$,  we get
\begin{align}
y_1\in P_{(1,s),(q,1)}(y_2, y_4) &\mbox{ and } y_3\in P_{(r,1),(1,q)}(y_2, y_4), \mbox{ or} \tag{3.4}\label{eq:2-2-1}\\
y_1\in P_{(r,1),(1,q)}(y_2, y_4) &\mbox{ and } y_3\in P_{(1,s),(q,1)}(y_2, y_4). \tag{3.5}\label{eq:2-2-2}
\end{align}

If \eqref{eq:2-1-1} holds, then $(y_1,y_2)\in\Gamma_{1,r}$ and $(y_2,y_3)\in\Gamma_{p,1}$ with $p\geq2$, which imply $(y_1,y_2)\notin\Gamma_{s,1}$ with $r\neq s$ and $(y_2,y_3)\notin\Gamma_{1,s}$, contrary to \eqref{eq:2-2-1} and \eqref{eq:2-2-2}. If \eqref{eq:2-1-2} holds, then $(y_1,y_2)\in\Gamma_{s,1}$ and $(y_1,y_4)\in\Gamma_{1,r}$, which imply $(y_1,y_2)\notin\Gamma_{1,r}$ with $r\neq s$ and $(y_1,y_4)\notin\Gamma_{q,1}$ with $q\geq2$, contrary to \eqref{eq:2-2-2} and \eqref{eq:2-2-1}. Thus, $(y_1,y_2)$, $(y_3,y_4)\in\Sigma_2$.

%
%
%
%

Suppose that \ref{fenlei6} holds with $p\geq3$ or $q\geq3$. 
By way of contradiction, we may assume $(y_1,y_2)\in\Sigma_1$. Then $(y_1,y_3)\in \Sigma_2$ or $(y_1,y_4) \in \Sigma_2$. In view of Lemma \ref{sigma2}, $(y_1,y_4) \in \Sigma_2$ is equivalent to $(y_2,y_3) \in \Sigma_2$, which implies $(y_1,y_3)\in \Sigma_2$ or $(y_2,y_3) \in \Sigma_2$.  By replacing the roles of $p$ and $q$, we may assume $(y_1,y_3)\in \Sigma_2$.

By Lemma \ref{sigma2}, we get $(y_2,y_4)\in \Sigma_2$, $C(y_1, y_3) = \{x,z,y_2,y_4\}$ and $C(y_2, y_4) =\{x,z,y_1, y_3\}$. Moreover,  we obtain
\begin{align}
x\in P_{(p,1),(1,r)}(y_1, y_3) \mbox{ and } z\in P_{(1,p),(1,s)}(y_1,y_3),\tag{3.6}\label{eq:2-3-1}\\
x\in P_{(q,1),(s,1)}(y_2,y_4) \mbox{ and } z\in P_{(1,q),(r,1)}(y_2, y_4).\tag{3.7}\label{eq:2-3-2}
\end{align}

By (\ref{eq:2-3-1}), one has
\begin{equation}\label{eq:2-3-3}
y_2\in P_{(1,r),(p,1)}(y_1, y_3) \text{ or } y_4\in P_{(1,r),(p,1)}(y_1, y_3).\nonumber
\end{equation}
If $y_2\in P_{(1,r),(p,1)}(y_1, y_3)$, then $(y_2, z, y_3)$ is a circuit in $\Gamma$ with $2\geq \partial_{\Gamma}(z,y_2) = q$ and $2\geq \partial_{\Gamma}(y_2,y_3) = p$, which is impossible. Hence, $y_4\in P_{(1,r),(p,1)}(y_1, y_3)$. Then $(y_3, y_4, z)$ is a circuit in $\Gamma$ with $2\geq \partial(y_4, y_3) = p$ and $2\geq \partial_{\Gamma}(z,y_4) = r$, and so $q \geq 3$.

By (\ref{eq:2-3-2}), one obtains
\begin{equation}\label{eq:2-3-4}
y_1\in P_{(r,1),(1,q)}(y_2, y_4) \text{ or } y_3\in P_{(r,1),(1,q)}(y_2, y_4) .\nonumber
\end{equation}
Note that $y_4\in P_{(1,r),(p,1)}(y_1, y_3)$. If $y_1\in P_{(r,1),(1,q)}(y_2, y_4)$, then $2 \geq r =\partial(y_4, y_1) =  q \geq 3$, a contradiction. If $y_3\in P_{(r,1),(1,q)}(y_2, y_4)$, then $2\geq p = \partial(y_4, y_3) = q\geq 3$, a contradiction. Thus, $(y_1,y_2)$, $(y_3,y_4)\in\Sigma_2$.

%

This completes the proof of the lemma.
\end{proof}

The commutativity of $\Gamma$ will be used frequently in the sequel, so we no longer refer to it for the sake of simplicity.

The following corollary is immediate by Lemma \ref{fenlei}.

\begin{cor}\label{fenlei*}
Let $(x,z)\in\Gamma_{2,i}$ for some $i\geq3$. Suppose $y_1,y_2\in C(x,z)$ and $(y_1,y_2)\in\Sigma_2$. If $y_1\in P_{(1,p),(1,p)}(x,z)$ for some $p\geq3$, then $y_2\in P_{(1,q),(1,q)}(x,z)$ for some $q\geq2$.
\end{cor}

\begin{lemma}\label{22}
Suppose $(2,2)\in \tilde{\partial}(\Gamma)$. Then exactly one of the following holds:
\begin{enumerate}
\item\label{221} $p^{(2,2)}_{(1,1),(1,1)} = 4$;
\item\label{222} $p^{(2,2)}_{(1,2),(1,2)} = p^{(2,2)}_{(2,1),(2,1)} = 2$;
\item\label{223} $p^{(2,2)}_{(1,p),(1,q)} = p^{(2,2)}_{(1,q),(1,p)} = p^{(2,2)}_{(p,1),(q,1)} = p^{(2,2)}_{(q,1),(p,1)} = 1$ with $(p,q)\in\{(1,2),(2,3)\}$;
\item\label{224} $p^{(2,2)}_{(1,1),(1,1)}=2$ and $p^{(2,2)}_{(1,r),(r,1)} = p^{(2,2)}_{(r,1),(1,r)} = 1$ with $r\geq2$;
\item\label{225} $p^{(2,2)}_{(1,p),(1,p)}=p^{(2,2)}_{(p,1),(p,1)}=p^{(2,2)}_{(1,r),(r,1)} = p^{(2,2)}_{(r,1),(1,r)} = 1$ with $p\in\{2,3\}$ and $r\geq2$.
\end{enumerate}
\end{lemma}
\begin{proof}
For all $\tilde{i},\tilde{j}\in\tilde{\partial}(\Gamma),$ we have
\begin{align}\label{3.5gongshi}\tag{3.9}
p^{(2,2)}_{\tilde{i},\tilde{j}}=
p^{(2,2)}_{\tilde{j}^*,\tilde{i}^*}=p^{(2,2)}_{\tilde{i}^*,\tilde{j}^*}=p^{(2,2)}_{\tilde{j},\tilde{i}}.
 \end{align}
Let $(x,z)\in\Gamma_{2,2}$ and $C(x,z) = \{y_1, y_2, y_3, y_4\}$. Then there is a path of length $2$ from $x$ to $z$ in $\Gamma$.

\textbf{Case 1.} There exists a mixed path from $x$ to $z$, i.e., $p^{(2,2)}_{(1,p),(1,q)}\neq0$ for some $p\neq q$.

Since $c_2=4$, from \eqref{a1c2} and \eqref{3.5gongshi}, we have \begin{equation}\label{eq:22-1}
p^{(2,2)}_{(1,p),(1,q)}  = p^{(2,2)}_{(q,1),(p,1)} = p^{(2,2)}_{(p,1),(q,1)} = p^{(2,2)}_{(1,q),(1,p)}= 1.\tag{3.10}
\end{equation}
Without loss of generality, we may assume that $p<q$. Since $(x,z)\in\Gamma_{2,2}$, we get $1\leq p<q\leq3$, and so $(p,q)=(1,2),(1,3)$ or $(2,3)$. By reordering $y_1, y_2, y_3, y_4$, we may assume that
$$y_1\in P_{(1,p),(1,q)}(x,z),  y_2\in P_{(q,1),(p,1)}(x,z),  y_3\in P_{(p,1),(q,1)}(x,z), y_4\in P_{(1,q),(1,p)}(x,z). $$


Suppose $(p,q)=(1,3)$. If $(y_1,y_2)\in \Sigma_1$, then $\partial_{\Gamma}(y_1,y_2)=1$ or $\partial_{\Gamma}(y_2,y_1)=1$, which implies $(y_1, y_2, x)$ or $(y_1, z, y_2)$   is a circuit in $\Gamma$, contrary to the fact that $(y_2,x),(y_1,z)\in\Gamma_{1,3}$. Hence, $(y_1,y_2)\notin \Sigma_1$, and so $(y_1,y_2)\in\Sigma_2$. By Lemma \ref{sigma2}, one has $C(y_1,y_2)=\{x,z,y_3,y_4\}$.

If $\partial_{\Gamma}(y_3,y_1)=1$, then $(y_1,z,y_3)$ is a circuit in $\Gamma$, contrary to the fact that $(y_1,z)\in\Gamma_{1,3}$. Thus,  $\partial_{\Gamma}(y_3,y_1)\neq 1$, and so $\partial_{\Gamma}(y_1,y_3)=1$. Then $(y_1,y_3,x)$ is a circuit in $\Gamma$, which implies $(y_1,y_3)\in \Gamma_{1,2}$. Since $(y_1,y_2)\in\Sigma_2$ and $(y_1, z, y_2, x)$ is a circuit in $\Gamma$, we get $(y_1,y_2)\in \Gamma_{2,2}$. Then $y_3\in P_{(1,2),(s,t)}(y_1,y_2)$ for some $1\in\{s,t\}$,  contrary to (\ref{eq:22-1}).

%

Therefore, $(p,q)=(1,2)$ or $(2,3)$, and so \ref{223} holds.

\textbf{Case 2.} Each path from $x$ to $z$ is pure.

By \eqref{3.5gongshi}, we get $p^{(2,2)}_{(1,p),(1,p)}=p^{(2,2)}_{(p,1),(p,1)}\neq0$ with $p\in\{1,2,3\}$. 

\textbf{Case 2.1.} There exists $y\in C(x,z)$ such that $(x,y,z)$  is a non-path.

Note that $y\in P_{(1,r),(l,1)}(x,z)$ or $P_{(l,1),(1,r)}(x,z)$ for some $r,l\geq2$. By \eqref{3.5gongshi}, one has $$p^{(2,2)}_{(1,r),(l,1)}=p^{(2,2)}_{(1,l),(r,1)}=p^{(2,2)}_{(r,1),(1,l)}=p^{(2,2)}_{(l,1),(1,r)}\neq0.$$
Since $c_2=4$ and $p^{(2,2)}_{(1,p),(1,p)}=p^{(2,2)}_{(p,1),(p,1)}\neq0$ with $p\in\{1,2,3\}$, from \eqref{a1c2}, we get $l=r$ and $p^{(2,2)}_{(1,r),(r,1)}=p^{(2,2)}_{(r,1),(1,r)}=1,$ and so \ref{224} or \ref{225} holds.

\textbf{Case 2.2.} For any $y\in C(x,z)$,  $(x,y,z)$ or $(z,y,x)$ is a pure path.

Since $c_2=4$ and $p^{(2,2)}_{(1,p),(1,p)}=p^{(2,2)}_{(p,1),(p,1)}\neq0$ with $p\in\{1,2,3\}$, from \eqref{a1c2}, there are six possibilities:
\begin{itemize}
\item [{\rm(a)}] $p^{(2,2)}_{(1,1),(1,1)}=4$;
\item[{\rm(b)}] $p^{(2,2)}_{(1,2),(1,2)}=p^{(2,2)}_{(2,1),(2,1)}=2$;
\item[{\rm(c)}] $p^{(2,2)}_{(1,3),(1,3)}=p^{(2,2)}_{(3,1),(3,1)}=2$;
\item [{\rm(d)}] $p^{(2,2)}_{(1,1),(1,1)}=2$ and $p^{(2,2)}_{(1,2),(1,2)}=p^{(2,2)}_{(2,1),(2,1)}=1$;
\item [{\rm(e)}] $p^{(2,2)}_{(1,1),(1,1)}=2$ and $p^{(2,2)}_{(1,3),(1,3)}=p^{(2,2)}_{(3,1),(3,1)}=1$;
\item [{\rm(f)}] $p^{(2,2)}_{(1,2),(1,2)}=p^{(2,2)}_{(2,1),(2,1)}=p^{(2,2)}_{(1,3),(1,3)}=p^{(2,2)}_{(3,1),(3,1)}=1$.
\end{itemize}

Suppose $p^{(2,2)}_{(1,3),(1,3)}\neq0$. Set $y_1\in P_{(1,3),(1,3)}(x,z)$.
By the above six possibilities, there exist  two distinct vertices $y,y'\in C(x,z)$ such that $(z,y,x)$ and $(z,y',x)$ are paths in $\Gamma$. In view of Lemma \ref{sigma2}, one has $y\in\Sigma_1(y_1)$ or $y'\in\Sigma_1(y_1)$. Without loss of generality, we may assume $y\in\Sigma_1(y_1)$. Then $\partial_{\Gamma}(y_1,y)=1$ or $\partial_{\Gamma}(y,y_1)=1$. It follows that $(y_1,y,x)$ or $(y,y_1,z)$ is a circuit in $\Gamma$, contrary to the fact that $(x,y_1),(y_1,z)\in\Gamma_{1,3}$. Thus, $p^{(2,2)}_{(1,3),(1,3)}=0$, and so  (c), (e) and (f) do not hold. Thus, (a), (b) or (d) holds.

To prove this lemma, it suffices to show that \ref{221} or \ref{222} holds. Then we only need to prove that (d) does not hold. Suppose for the contrary that (d) holds.  By reordering $y_1, y_2, y_3, y_4$, we may assume that
\begin{align}\label{gongshi22}
y_1\in P_{(1,2),(1,2)}(x,z),~ y_2\in P_{(2,1),(2,1)}(x,z),  ~y_3,y_4\in P_{(1,1),(1,1)}(x,z).\tag{3.11}
\end{align} If $y_j\in\Sigma_2(y_1)$ for some $j\in\{3,4\}$, since $(y_1,z,y_j,x)$ is a circuit in $\Gamma$, then $(y_1,y_j)\in\Gamma_{2,2}$, which implies $z\in P_{(1,2),(1,1)}(y_1,y_j)$, and so there exists $y\in C(x,z)$ such that $y\in P_{(1,2),(1,1)}(x,z)$, contrary to \eqref{gongshi22}. Hence, $y_j\notin\Sigma_2(y_1)$ for $j\in\{3,4\}$. By Lemma \ref{sigma2}, we get $y_2\in\Sigma_2(y_1)$, which implies $(y_3,y_4)\in\Sigma_2$, $C(y_1,y_2)=\{x,z,y_3,y_4\}$ and $C(y_3,y_4)=\{x,z,y_1,y_2\}$.
Since $(y_1,z,y_2,x)$ is a circuit in $\Gamma$, one has $(y_1,y_2)\in\Gamma_{2,2}$. Since $z\in P_{(1,2),(1,2)}(y_1,y_2)$ and $x\in P_{(2,1),(2,1)}(y_1,y_2)$, from (d), we obtain $y_3,y_4\in P_{(1,1),(1,1)}(y_1,y_2)$. Then $x,z,y_1,y_2\in P_{(1,1),(1,1)}(y_3,y_4)$, which implies $(y_3,y_4)\in\Gamma_{2,2}$ and $p^{(2,2)}_{(1,1),(1,1)}=4$, contrary to \eqref{gongshi22}. Thus, (d) does not hold.


This completes the proof of this lemma.
\end{proof}

\section{The types of arcs}


In the sequel,  we always assume that  $\Sigma$  is a Johnson graph $J(n,e)$ with $n\geq2e$ and $e\geq2$, or a folded Johnson graph $\bar{J}(2e,e)$ with $e\geq4$.
Let $m=n$ if $\Sigma$ is a Johnson graph $J(n,e)$, and $m=2e$ if $\Sigma$ is a folded Johnson graph $\bar{J}(2e,e)$. Set ${X\choose e}$ to be the vertex set of $\Sigma$, where $X$ is an $m$-set.
Note that $\Sigma$ satisfies the property in Section 3. Then $c_2=4$. 


Let $T$ be the set of integers $q+1$ such that $(1,q)\in\tilde{\partial}(\Gamma)$. That is, $T=\{q+1\mid (1,q)\in\tilde{\partial}(\Gamma)\}$.
The main result of this section is as follows, which determines the types of arcs.

\begin{prop}\label{arcpure}
The following hold:
\begin{enumerate}
\item \label{arcpure1} If $p^{(1,2)}_{(1,3),(1,3)}=0$, then $T=\{3\}$;
\item \label{arcpure2} If $p^{(1,2)}_{(1,3),(1,3)}\neq0$, then $T=\{3,4\}$ or $T=\{2,3,4\}$.
\end{enumerate}
\end{prop}

%
%

The following result is key in the proof of Proposition \ref{arcpure}, which will be proved in Appendix A.

\begin{lemma}\label{arc}
Let $q+1\in T$. If $p^{(1,q-1)}_{(1,q),(1,h)}=0$ for all $h\geq1$, then $q\leq3$.
\end{lemma}

%
%


To prove Proposition \ref{arcpure}, we recall the definitions of pure arcs and mixed arcs introduced in \cite{YYF20}. An arc of type $(1,q)$ is said to be \emph{pure}, if every circuit of length $q+1$ containing it is pure; otherwise, this arc is said to be \emph{mixed}. We say that $(1,q)$ is \emph{pure} if any arc of type $(1,q)$ is pure; otherwise, we say that $(1,q)$ is \emph{mixed}.  


Next we divide the proof of Proposition \ref{arcpure} into two subsections.

\subsection{Proof of Proposition \ref{arcpure} \ref{arcpure1}}


\begin{lemma}\label{131i}
If $p^{(1,2)}_{(1,3),(1,3)}=0$, then $p^{(1,2)}_{(1,3),(1,r)}=0$ for all $r\geq1$.
\end{lemma}
\begin{proof}
Assume the contrary, namely, $p^{(1,2)}_{(1,3),(1,r)}\neq0$ for some $r\geq1$. Let $(x_0,x_1,x_2)$  be a circuit with $(x_0,x_1)\in\Gamma_{1,2}$ and $(x_i,x_{i+1})\in\Gamma_{1,r_i}$ for $r_i\in\{1,2\}$ with  $i=1,2$ and $x_3=x_0$. Pick a vertex $y\in P_{(1,3),(1,r)}(x_0,x_1)$. Since $p^{(1,2)}_{(1,3),(1,3)}=0$ and $(x_0,y,x_1,x_2)$ is a circuit, one has $r\in\{1,2\}$.

We claim that $(x_2,y)\in\Gamma_{1,2}$. Suppose $(y,x_2)\in\Sigma_2$. Since $(x_0,y,x_1,x_2)$ is a circuit, we have $(y,x_2)\in\Gamma_{2,2}$. Since $x_0\in P_{(1,r_2),(1,3)}(x_2,y)$, from  Lemma \ref{22}, one gets $(x_2,x_0)\in\Gamma_{1,2}$, and so  $p^{(2,2)}_{(1,2),(1,3)}\neq0$. It follows that Lemma \ref{22} \ref{223} holds. Since $x_1\in P_{(1,r),(1,r_1)}(y,x_2)$  with $r,r_1\in\{1,2\}$, we get $p^{(2,2)}_{(1,r),(1,r_1)}\neq0$, and so $c_2\geq5$ from \eqref{a1c2} and  Lemma \ref{22} \ref{223}, contrary to the fact that $c_2=4$. Thus, $(y,x_2)\in\Sigma_1$. Since $3=\partial_{\Gamma}(y,x_0)\leq\partial_{\Gamma}(y,x_2)+1$, we have $(x_2,y)\in\Gamma_{1,t}$ with $t\geq2$. The fact $(y,x_1,x_2)$ is a circuit in $\Gamma$ implies $t=2$. Thus, our claim holds.


\textbf{Case 1.} $r=2$.

By the claim, one has $x_2\in P_{(1,r_1),(1,2)}(x_1,y)$ with $(x_1,y)\in\Gamma_{2,1}$, which implies $p^{(2,1)}_{(1,r_1),(1,2)}\neq0$. Since $x_2\in P_{(1,r_1),(1,r_2)}(x_1,x_0)$ was arbitrary, we may assume $x_2\in P_{(1,r_1),(1,2)}(x_1,x_0)$.

Since $p^{(1,2)}_{(1,2),(1,3)}=p^{(1,2)}_{(1,3),(1,2)}\neq0$, choose a vertex $z\in P_{(1,2),(1,3)}(x_2,x_0)$. If $(z,x_1)\in\Sigma_2$, then $(z,x_1)\in\Gamma_{2,2}$ since $x_0\in P_{(1,3),(1,2)}(z,x_1)$ and $x_2\in P_{(1,r_1),(1,2)}(x_1,z)$ with $r_1\in\{1,2\}$, contrary to  Lemma \ref{22}. Hence, $(z,x_1)\in\Sigma_1$. Since $3=\partial_{\Gamma}(x_0,z)\leq1+\partial_{\Gamma}(x_1,z)$, one has $(z,x_1)\in\Gamma_{1,l}$ with $l\geq2$. The fact that $(x_2,z,x_1)$ is a circuit implies $l=2$.

Since $x_1\in C(x_2,x_0)$, from Lemma \ref{sigma1}, we obtain $x_1\in Y_{i}(x_2,x_0)$ for some $i\in\{1,2\}$. Since $z,y\in C(x_2,x_0)$ and $z,y\in\Sigma_1(x_1)$, from Lemma \ref{sigma1} \ref{sigma11}, we get $z,y\in Y_{i}(x_2,x_0)$, and so
$(z,y)\in\Sigma_1$. The fact $3=\partial_{\Gamma}(x_0,z)\leq1+\partial_{\Gamma}(y,z)$ implies $(z,y)\in\Gamma_{1,t}$ for some $t\geq2$. Since $p^{(1,2)}_{(1,3),(1,3)}=0$ and $x_0\in P_{(1,3),(1,3)}(z,y)$, we have $t\neq2$.  The fact that $(z,y,x_1,x_2)$ is a circuit implies $t=3$, and so $z\in P_{(1,2),(1,3)}(x_2,y)$. Since $z\in P_{(1,2),(1,3)}(x_2,x_0)$ was arbitrary, we get $P_{(1,2),(1,3)}(x_2,y)\supseteq\{x_0\}\cup P_{(1,2),(1,3)}(x_2,x_0)$, a contradiction. 

\textbf{Case 2.} $r=1$.

By the claim, one has $x_1\in P_{(1,1),(1,r_1)}(y,x_2)$ with $(y,x_2)\in\Gamma_{2,1}$, which implies $p^{(2,1)}_{(1,1),(1,r_1)}\neq0$. Since $(x_1,x_0)\in\Gamma_{2,1}$ and $(x_1,x_2,x_0)$ is a path, we may assume $r_1=1$.
Since $p^{(1,2)}_{(1,3),(1,1)}\neq0$, from Lemma \ref{intersection numners} \ref{intersection numners 2}, we get $p^{(1,1)}_{(2,1),(1,3)}=p^{(1,1)}_{(3,1),(1,2)}=k_{1,2}p^{(1,2)}_{(1,3),(1,1)}/k_{1,1}\neq0$, which implies that there exists $z\in P_{(2,1),(1,3)}(x_1,x_2)$.

Suppose $(z,y)\in\Sigma_2$. Since $x_2\in P_{(1,3),(1,2)}(z,y)$ and $x_1\in P_{(1,2),(1,1)}(z,y)$, from Lemma \ref{22}, we get $(z,y)\notin\Gamma_{2,2}$, and so $(z,y)\in\Gamma_{2,3}$. Since $c_2=4$, by \eqref{a1c2}, one has
\begin{align}\label{gongshi2-1}
p^{(2,3)}_{(1,3),(1,2)}=p^{(2,3)}_{(1,2),(1,3)}=p^{(2,3)}_{(1,2),(1,1)}=p^{(2,3)}_{(1,1),(1,2)}=1.\tag{4.1}
\end{align}
Since $x_0\in C(x_1,x_2)$, from Lemma \ref{sigma1}, we obtain $x_0\in Y_{i}(x_1,x_2)$ for some $i\in\{1,2\}$. Since $y,z\in C(x_1,x_2)$ with $y\in\Sigma_1(x_0)$ and $z\in\Sigma_2(y)$, from Lemma \ref{sigma1} \ref{sigma11} and \ref{sigma12}, one has $y\in Y_i(x_1,x_2)$ and $z\in Y_j(x_1,x_2)$ with $j\neq i$, and so $(z,x_0)\in\Sigma_2$.
Since $x_2\in P_{(1,3),(1,r_2)}(z,x_0)$ with $r_2\in\{1,2\}$ and $x_1\in P_{(1,2),(2,1)}(z,x_0)$, from Lemma \ref{22}, we get $(z,x_0)\notin\Gamma_{2,2}$, which implies $(z,x_0)\in\Gamma_{2,3}$, and so $p^{(2,3)}_{(1,2),(2,1)}\neq0$.  \eqref{a1c2} and \eqref{gongshi2-1} imply $c_2\geq5$, a contradiction. Thus, $(z,y)\in\Sigma_1$.

Since $3=\partial_{\Gamma}(x_2,z)\leq1+\partial_{\Gamma}(y,z)$ and $\partial_{\Gamma}(y,z)\leq1+\partial_{\Gamma}(x_1,z)=3$, we obtain $(z,y)\in\Gamma_{1,t}$ for some $t\in\{2,3\}$. Since $x_2\in P_{(1,3),(1,2)}(z,y)$, from Case 1, we get $t=3$ and $y\in P_{(1,3),(1,1)}(z,x_1)$. Since $y\in P_{(1,3),(1,1)}(x_0,x_1)$ was arbitrary, we get $P_{(1,3),(1,1)}(z,x_1)\supseteq\{x_2\}\cup P_{(1,3),(1,1)}(x_0,x_1)$, a contradiction.
\end{proof}

\begin{lemma}\label{maxT=3}
If $p^{(1,2)}_{(1,3),(1,3)}=0$, then $\max T=3$.
\end{lemma}
\begin{proof}
Since $\Gamma$ is not undirected, we have $\max T\geq3$. Assume the contrary, namely,  $q$ is the minimal integer such that $q+1\in T$ with $q\geq3$. If $q\geq4$, from  Lemma \ref{arc}, then $p^{(1,q-1)}_{(1,q),(1,h)}\neq0$ for some $h\geq1$, which implies  $q\in T$, contrary to the minimality of $q$.  Then $q=3$.


Let $(x_0,x_1,x_2,x_3)$ be a circuit such that $(x_0,x_1)\in\Gamma_{1,3}$ and $(x_i,x_{i+1})\in\Gamma_{1,r_i}$ for $r_i\geq1$ with $1\leq i\leq3$ and $x_4=x_0$.   Lemma \ref{131i} implies $(x_0,x_2)\in\Gamma_{2,2}$. Suppose $r_i\neq3$ for some $1\leq i\leq3$. Then Lemma \ref{22} \ref{223} holds. It follows that $r_1=2$ and $x_1\in P_{(1,3),(1,2)}(x_0,x_2)$. Therefore,  there exists $x'_3\in P_{(1,2),(1,3)}(x_2,x_0)$. Note that $(x'_3,x_0,x_1,x_2)$ is a circuit with $(x'_3,x_0),(x_0,x_1)\in\Gamma_{1,3}$. Since $p^{(1,2)}_{(1,3),(1,3)}=0$,  we obtain $(x'_3,x_1)\in\Gamma_{2,2}$, which implies $p^{(2,2)}_{(1,3),(1,3)}\neq0$. \eqref{a1c2} and Lemma \ref{22} \ref{223} imply $c_2\geq5$,  contrary to the fact that $c_2=4$. Thus, $r_i=3$ for all $1\leq i\leq3$. It follows that Lemma \ref{22} \ref{225} holds. Then we may set
 \begin{align}\tag{4.2}\label{5.1gongshi1}
z_1\in P_{(1,r),(r,1)}(x_0,x_2)~\mbox{and}~z_2\in P_{(r,1),(1,r)}(x_0,x_2)~\mbox{with}~r\geq2.
\end{align}

Since $(x_0,x_2)\in\Gamma_{2,2}$, we get $(x_0,x_2)\in\Sigma_2$. The fact $c_2=4$ implies $C(x_0,x_2)=\{x_1,x_3,z_1,z_2\}$. Since $(x_1,x_2,x_3,x_0)$ is a circuit consisting of arcs of type $(1,3)$, one has $(x_1,x_3)\in\Gamma_{2,2}$, and so $(x_1,x_3)\in\Sigma_2$. By Lemma \ref{sigma2}, we obtain $C(x_1,x_3)=\{x_0,x_2,z_1,z_2\}$. Since $x_2\in P_{(1,3),(1,3)}(x_1,x_3)$ and $x_0\in P_{(3,1),(3,1)}(x_1,x_3)$, from Lemma \ref{22} \ref{225}, we get $z_1\in P_{(1,r),(r,1)}(x_1,x_3)$ or $P_{(r,1),(1,r)}(x_1,x_3)$. Since $3=\partial_{\Gamma}(x_2,x_1)\leq1+\partial_{\Gamma}(z_1,x_1)$, we get \begin{align}\tag{4.3}\label{5.1gongshi2}z_1\in P_{(1,r),(r,1)}(x_1,x_3)~\mbox{and}~z_2\in P_{(r,1),(1,r)}(x_1,x_3).\end{align}
By \eqref{5.1gongshi1} and \eqref{5.1gongshi2}, we have $x_1\in P_{(1,3),(1,r)}(x_0,z_1)$ with $(x_0,z_1)\in\Gamma_{1,r}$, which implies $p^{(1,r)}_{(1,3),(1,r)}\neq0$. It follows from Lemma \ref{131i} that $r\neq2$, and so $r\geq3$.

%

Since $x_0\in{X\choose e},$ $(x_0,x_2)\in\Sigma_2$ and $x_1\in C(x_0,x_2)$, there exist distinct elements $a_1,a_2\in x_0$ and $b_1,b_2\in X\setminus x_0$ such that $x_1=x_0(a_1,b_1)~\mbox{and}~x_2= x_0(\{a_1,a_2\},\{b_1,b_2\}).$ The fact $x_3\in C(x_0,x_2)\cap\Sigma_2(x_1)$ implies $x_3=x_0(a_2,b_2)$. Since $z_1,z_2\in C(x_0,x_2)$, without loss of generality, we may assume $z_1=x_0(a_2,b_1)~\mbox{and}~z_2=x_0(a_1,b_2).$

By \eqref{5.1gongshi2}, we get $x_1\in P_{(1,3),(1,r)}(x_0,z_1)$ and $x_3\in P_{(3,1),(1,r)}(x_0,z_1)$. Then  there exists
\begin{align}\tag{4.4}\label{5.1gongshi3}w\in P_{(1,r),(3,1)}(x_0,z_1)~\mbox{with}~w\notin\{x_1,x_3\}.\end{align} By Lemma \ref{sigma1}, we have $w=x_0(a_3,b_1)$ for some $a_3\in x_0\setminus\{a_1,a_2\}$, or $w=x_0(a_2,b_3)$ for some $b_3\in X\setminus(x_0\cup\{b_1,b_2\})$.

%
%

\textbf{Case 1.} $w=x_0(a_3,b_1)$ for some $a_3\in x_0\setminus\{a_1,a_2\}$.

By \eqref{5.1gongshi2} and \eqref{5.1gongshi3}, we have $3=\partial_{\Gamma}(w,z_1)\leq\partial_{\Gamma}(w,x_1)+1$, which implies
$(x_1,w)\in\Gamma_{1,l}~\mbox{with}~l\geq2.$
Let $w_1=x_0(\{a_1,a_3\},\{b_1,b_2\})$.  Since $z_2=x_0(a_1,b_2)$, one gets $(z_2,w)\in\Sigma_2$ and $x_0,x_1,w_1\in C(z_2,w)$ with $(x_0,w_1)\in\Sigma_2$. By \eqref{5.1gongshi1}, \eqref{5.1gongshi2} and \eqref{5.1gongshi3}, we get $x_0\in P_{(1,r),(1,r)}(z_2,w)$ and $x_1\in P_{(1,r),(1,l)}(z_2,w)$ with $r\geq3$. Lemma \ref{22} \ref{225} implies $(z_2,w)\notin\Gamma_{2,2}$. It follows from Corollary \ref{fenlei*}  that $w_1\in P_{(1,s),(1,s)}(z_2,w)~\mbox{with}~s\geq2.$

Since $x_2=x_0(\{a_1,a_2\},\{b_1,b_2\})$, we get $(x_2,w)\in\Sigma_2$ and $ C(x_2,w)=\{x_1,z_1,w_1,w_2\}$ with $w_2=x_0(\{a_2,a_3\},\{b_1,b_2\})$. By \eqref{5.1gongshi1} and \eqref{5.1gongshi3}, one has $x_1\in P_{(3,1),(1,l)}(x_2,w)$ and $z_1\in P_{(1,r),(1,3)}(x_2,w)$. The fact $w_1\in P_{(1,s),(1,s)}(z_2,w)$ implies \begin{align}\tag{4.5}\label{5.1gongshi5}w_2\in P_{(1,l),(3,1)}(x_2,w).\end{align}


By \eqref{5.1gongshi3} and \eqref{5.1gongshi5}, we have $3=\partial_{\Gamma}(w,z_1)\leq1+\partial_{\Gamma}(w_2,z_1),$ which implies $(z_1,w_2)\in\Gamma_{1,t}$ for some $t\geq2$.  Since $w_2=x_0(\{a_2,a_3\},\{b_1,b_2\})$, one gets $(x_0,w_2)\in\Sigma_2$ and $C(x_0,w_2)=\{z_1,w,w',x_3\}$ with $w'=x_0(a_3,b_2)$. By \eqref{5.1gongshi1}, \eqref{5.1gongshi3} and \eqref{5.1gongshi5}, we get \begin{align}\tag{4.6}\label{5.1gongshi4}z_1\in P_{(1,r),(1,t)}(x_0,w_2)~\mbox{and}~w\in P_{(1,r),(1,3)}(x_0,w_2).\end{align} 
By Lemma \ref{22} \ref{225}, we have $(x_0,w_2)\notin\Gamma_{2,2}$. Hence, $x_3\in P_{(3,1),(1,s')}(x_0,w_2)$ for some $s'\geq2$, and so $w'\in P_{(1,s'),(3,1)}(x_0,w_2)$. If $r\neq t$ or $r\neq3$, from \eqref{5.1gongshi4}, then there exists $u\in P_{(1,t),(1,r)}(x_0,w_2)$ or $P_{(1,3),(1,r)}(x_0,w_2)$ with $u'\notin\{z_1,w,w',x_3\}$, a contradiction. Thus, $r=t=3$. By Lemma \ref{fenlei} \ref{fenlei6}, one gets $(z_1,w)\in\Sigma_2$, contrary to the fact that $(z_1,w)\in\Sigma_1$.

\textbf{Case 2.} $w=x_0(a_2,b_3)$ for some $b_3\in X\setminus(x_0\cup\{b_1,b_2\})$.

By \eqref{5.1gongshi2} and \eqref{5.1gongshi3}, we have $3=\partial_{\Gamma}(w,z_1)\leq\partial_{\Gamma}(w,x_3)+1$, which implies $(x_3,w)\in\Gamma_{1,l}$ for some $l\geq2$. Since $z_2=x_0(a_1,b_2)$, one gets $(z_2,w)\in\Sigma_2$ and $ C(z_2,w)=\{x_0,x_3,w',w_1\}$ with $w'=x_0(a_1,b_3)$ and $w_1=x_0(\{a_1,a_2\},\{b_2,b_3\})$.  By \eqref{5.1gongshi1}, \eqref{5.1gongshi2} and \eqref{5.1gongshi3}, we get $x_0\in P_{(1,r),(1,r)}(z_2,w)$ and $x_3\in P_{(1,r),(1,l)}(z_2,w)$. Lemma \ref{22} \ref{225} implies $(z_2,w)\notin\Gamma_{2,2}$. Since $(x_0,w_1)\in\Sigma_2$, from Lemma \ref{fenlei} \ref{fenlei2} and \ref{fenlei5},
\begin{align}\tag{4.7}\label{5.1gongshi6}
w_1\in P_{(1,s),(1,s)}(z_2,w)~\mbox{with}~s\geq2,~\mbox{and}~(z_2,w',w)~\mbox{is a path in}~ \Gamma.\end{align}

Since $x_1=x_0(a_1,b_1)$, we obtain $(x_1,w)\in\Sigma_2$ and $C(x_1,w)=\{x_0,z_1,w_2,w'\}$ with $w_2=x_0(\{a_1,a_2\},\{b_1,b_3\})$.  By \eqref{5.1gongshi2} and \eqref{5.1gongshi3}, we get $x_0\in P_{(3,1),(1,r)}(x_1,w)$ and $z_1\in P_{(1,r),(1,3)}(x_1,w)$. \eqref{5.1gongshi6} implies $w_2\in P_{(1,r),(3,1)}(x_1,w)$.

Since $x_2=x_0(\{a_1,a_2\},\{b_1,b_2\})$, we get $(x_2,w)\in\Sigma_2$ and $ C(x_2,w)=\{x_3,z_1,w_1,w_2\}$. By \eqref{5.1gongshi1} and \eqref{5.1gongshi3}, one has $z_1\in P_{(1,r),(1,3)}(x_2,w)$. Since $x_3\in P_{(1,3),(1,l)}(x_2,w)$, from Lemma \ref{22} \ref{225}, we have $(x_2,w)\notin\Gamma_{2,2}$.  Since $(w_2,w)\in\Gamma_{3,1}$, we have \begin{align}\tag{4.8}\label{5.1gongshi7}w_2\in P_{(1,s'),(3,1)}(x_2,w)~\mbox{with}~s'\geq2\end{align}

Since $x_3=x_0(a_2,b_2)$ and $w_2=x_0(\{a_1,a_2\},\{b_1,b_3\})$, one gets $(x_3,w_2)\in\Sigma_2$ and $ C(x_3,w_2)=\{z_1,w,x_2,w_1\}$. By \eqref{5.1gongshi7}, one has $w\in P_{(1,l),(1,3)}(x_3,w_2)$ and $x_2\in P_{(3,1),(1,s')}(x_3,w_2)$. Then $z_1$ or $w_1\in P_{(1,s'),(3,1)}(x_3,w_2)$. If $z_1\in P_{(1,s'),(3,1)}(x_3,w_2)$, from \eqref{5.1gongshi3} and \eqref{5.1gongshi7}, then $(z_1,w,w_2)$ is a circuit with $(w_2,z_1)\in\Gamma_{1,3}$, a contradiction; if $w_1\in P_{(1,s'),(3,1)}(x_3,w_2)$, from \eqref{5.1gongshi6} and \eqref{5.1gongshi7}, then $(w,w_2,w_1)$ is a circuit with $(w_2,w_1)\in\Gamma_{1,3}$, a contradiction.
\end{proof}


\begin{lemma}\label{1211i}
If $\max T=3$, then $p^{(1,2)}_{(1,1),\tilde{i}}=0$ for all $\tilde{i}\in\{(1,1),(1,2),(2,1)\}$.
\end{lemma}
\begin{proof}
The proof is rather long, and we shall prove it in Appendix B.
\end{proof}

Now we are ready to prove Proposition \ref{arcpure} \ref{arcpure1}.

\begin{proof}[Proof of Proposition \ref{arcpure} \ref{arcpure1}]
By Lemma \ref{maxT=3}, we get $\max T=3$. Then $T=\{3\}$ or  $T=\{2,3\}$.

Suppose $T=\{2,3\}$. By \eqref{a1c2}, Lemma \ref{1211i} and Lemma \ref{intersection numners} \ref{intersection numners 2}, we get
\begin{align}\tag{4.9}\label{proof in 6.1}
p^{(1,1)}_{\tilde{i},\tilde{j}}=0~\mbox{for}~\tilde{i},\tilde{j}\in\{(1,1),(1,2),(2,1)\}~{\rm with}~(\tilde{i},\tilde{j})\neq((1,1),(1,1)).
\end{align}  Let $(x,y)\in\Gamma_{1,1}$. Since $x\in{X\choose e}$, one has $y=x(a_1,b_1)$ for some $a_1\in x$ and $b_1\in X\setminus x$.
  Pick a vertex $z\in\Sigma_1(x)\setminus\{y\}$. Then $z=x(a,b)$ for some $a\in x$ and $b\in X\setminus x$ with $(a,b)\neq(a_1,b_1)$. If $a=a_1$ or $b=b_1$, by Lemma \ref{sigma1}, then $z\in C(x,y)$, which implies $z\in P_{(1,1),(1,1)}(x,y)$ from \eqref{proof in 6.1}. Now we consider the case that $a\neq a_1$ and  $b\neq b_1$. Let $w=x(a_1,b)$. Then $w\in C(x,y)$, and so $w\in P_{(1,1),(1,1)}(x,y)$ from \eqref{proof in 6.1}. Since $z\in C(x,w)$, from \eqref{proof in 6.1}, we obtain $z\in P_{(1,1),(1,1)}(x,w).$ Hence, $(x,z)\in\Gamma_{1,1}$. Since  $z\in\Sigma_1(x)\setminus\{y\}$ was arbitrary, $\Gamma$ is undirected, a contradiction. Thus, $T=\{3\}$.
\end{proof}

\subsection{Proof of Proposition \ref{arcpure} \ref{arcpure2}}


\begin{lemma}\label{1313-2*}
Suppose $p^{(1,2)}_{(1,3),(1,3)}\neq 0$. Then $(1,2)$ is pure and $p^{(2,2)}_{(1,3),(1,2)}\neq0$. Moreover, if $(w_0,w_1,w_2)$ is a circuit in $\Gamma$ with $(w_0,w_1)\in\Gamma_{1,2}$ and $v\in P_{(1,3),(1,3)}(w_0,w_1)$, then the following hold:
\begin{enumerate}
\item\label{1313-2*2} $(v,w_2)\in\Gamma_{2,2}$;

\item\label{1313-2*3} There exist vertices $v_1,v_2\in C(v,w_2)$ such that $v_1\in P_{(1,3),(1,3)}(w_2,w_0)\cap\Sigma_2(w_1)$ and $v_2\in P_{(1,3),(1,3)}(w_1,w_2)\cap\Sigma_2(w_0)$.
\end{enumerate}
\end{lemma}
\begin{proof}
Since $(v,w_1,w_2,w_0)$ is a circuit with $(v,w_1),(w_0,v)\in\Gamma_{1,3}$, we get $(v,w_2)\in\Gamma_{2,2}$. Then \ref{1313-2*2} holds. To prove that $(1,2)$ is pure, it suffices to show $(w_1,w_2),(w_2,w_0)\in\Gamma_{1,2}$.  By Lemma \ref{22}, one has $p^{(2,2)}_{(1,3),(1,1)}=0$, which implies $w_1\notin P_{(1,3),(1,1)}(v,w_2)$ and $w_0\notin P_{(1,1),(1,3)}(w_2,v)$. It follows that $(w_1,w_2),(w_2,w_0)\in\Gamma_{1,2}$.  Since $w_1\in P_{(1,3),(1,2)}(v,w_2)$, one gets $p^{(2,2)}_{(1,3),(1,2)}\neq0$. Hence, the first statement holds.

Since $w_0\in P_{(3,1),(2,1)}(v,w_2)$ and $w_1\in P_{(1,3),(1,2)}(v,w_2)$,  there exist vertices
\begin{align}\label{1313-2*-1}
v_1\in P_{(2,1),(3,1)}(v,w_2) ~\mbox{and}~ v_2\in P_{(1,2),(1,3)}(v,w_2).\tag{4.10}
 \end{align}
Note that $(v,w_2)\in\Sigma_2$. Since $c_2=4$, we get $ C(v,w_2)=\{w_0,w_1,v_1,v_2\}$. Since $(w_0,v,v_2,w_2)$ is a circuit with $(w_0,v),(v_2,w_2)\in\Gamma_{1,3}$, one gets $(w_0,v_2)\in\Gamma_{2,2}$, and so $(w_0,v_2)\in\Sigma_2$. Lemma \ref{sigma2} implies $ C(w_0,v_2)=\{w_1,w_2,v,v_1\}$ and $(w_1,v_1)\in\Sigma_2.$

Since $v\in P_{(1,3),(1,2)}(w_0,v_2)$, $w_2\in P_{(2,1),(3,1)}(w_0,v_2)$ and $(w_0,w_1)\in\Gamma_{1,2}$,  one has
\begin{align}\label{1313-2*-2}
w_1\in P_{(1,2),(1,3)}(w_0,v_2)~\mbox{and}~v_1\in P_{(3,1),(2,1)}(w_0,v_2).\tag{4.11}
\end{align}
\eqref{1313-2*-1} and \eqref{1313-2*-2} imply $v_1\in P_{(1,3),(1,3)}(w_2,w_0)$ and  $v_2\in P_{(1,3),(1,3)}(w_1,w_2)$. Thus, \ref{1313-2*3} holds.
\end{proof}

In the remainder of this subsection, we may assume that $(x_0,x_1,x_2)$  is a circuit  consisting of arcs of type $(1,2)$ and $y\in P_{(1,3),(1,3)}(x_0,x_1)$. By Lemma \ref{1313-2*} \ref{1313-2*2}, we have $(y,x_2)\in\Gamma_{2,2}$, and so $(y,x_2)\in\Sigma_2$.

\begin{lemma}\label{1313-2}
If $p^{(1,2)}_{(1,3),(1,3)}\neq0$, then $p^{(1,2)}_{(1,3),(1,3)}=p^{(1,2)}_{(2,1),(2,1)}=1$.
\end{lemma}
\begin{proof}
Since $(y,x_2)\in\Sigma_2$, from Lemma \ref{1313-2*} \ref{1313-2*3},  there exist $y_1,y_2\in C(y,x_2)$ such that $y_1\in P_{(1,3),(1,3)}(x_2,x_0)\cap\Sigma_2(x_1)$ and  $y_2\in P_{(1,3),(1,3)}(x_1,x_2)\cap\Sigma_2(x_0)$. Since $c_2=4$, we get $C(y,x_2)=\{x_0,x_1,y_1,y_2\}$. Since $(x_0,y_2)\in\Sigma_2$, from Lemma \ref{sigma2}, one has $C(x_0,y_2)=\{x_1,x_2,y,y_1\}$.

Since $x_0\in{X\choose e}$,   $(x_0,y_2)\in\Sigma_2$ and $x_1\in C(x_0,y_2)$, there exist $a_1,a_2\in x_0$ with $a_1\neq a_2$ and $b_1,b_2\in X\setminus x_0$ with $b_1\neq b_2$ such that $x_1=x_0(a_1,b_1)$ and $y_2=x_0(\{a_1,a_2\},\{b_1,b_2\})$. Since $y_1\in C(x_0,y_2)\cap\Sigma_2(x_1)$, one gets $y_1=x_0(a_2,b_2)$. Then  $x_2=x_0(a_i,b_j)$ and $y=x_0(a_j,b_i)$ for $\{i,j\}=\{1,2\}$.

\begin{step}
{\rm Show that $p^{(1,2)}_{(1,3),(1,3)}=1$.}
\end{step}

Suppose for the contrary that $p^{(1,2)}_{(1,3),(1,3)}>1$. Let $z\in P_{(1,3),(1,3)}(x_0,x_1)$ with $z\neq y$. By Lemma \ref{1313-2*} \ref{1313-2*2}, we obtain $(z,x_2)\in\Sigma_2$. Lemma \ref{1313-2*} \ref{1313-2*3}  implies that there exists $z_1\in C(z,x_2)$ such that $z_1\in P_{(1,3),(1,3)}(x_2,x_0)\cap\Sigma_2(x_1)$. Since $z\in\Sigma_1(x_0)\cap\Sigma_1(x_1)\cap\Sigma_2(x_2)$ with $z\neq y$ and $z_1\in\Sigma_1(x_0)\cap\Sigma_1(x_2)\cap\Sigma_1(z)$, from Lemma \ref{sigma1}, we get
\begin{align}
&(i,j)=(1,2),~z=x_0(a_3,b_1)~\mbox{and}~z_1=x_0(a_3,b_2),~a_3\in x_0\setminus\{a_1,a_2\},~\mbox{or}\tag{4.12}\label{1313-2gongshi1}\\
&(i,j)=(2,1),~z=x_0(a_1,b_3)~\mbox{and}~z_1=x_0(a_2,b_3),~b_3\in X\setminus(x_0\cup\{b_1,b_2\}).\tag{4.13}\label{1313-2gongshi2}
\end{align}


Let \begin{align}\tag{4.14}\label{1313-2gongshi3}
w=
\begin{cases}
x_0(\{a_2,a_3\},\{b_1,b_2\}),~(i,j)=(1,2),\\
x_0(\{a_1,a_2\},\{b_2,b_3\}),~(i,j)=(2,1).
\end{cases}
\end{align}
By \eqref{1313-2gongshi1} and \eqref{1313-2gongshi2}, we have $(x_0,w)\in\Sigma_2$ and $ C(x_0,w)=\{y,z,y_1,z_1\}$. Note that $C(x_0,w)\subseteq(\Gamma_{1,3}\cup\Gamma_{3,1})(x_0)$. If there exists $u\in C(x_0,w)$ such that $u\in P_{(1,3),\tilde{i}}(x_0,w)\cup P_{(3,1),\tilde{i}}(x_0,w)$ with $\tilde{i}\notin\{(1,3),(3,1)\}$, then there exists $v\in C(x_0,w)$ such that $v\in P_{\tilde{i},(1,3)}(x_0,w)\cup P_{\tilde{i},(3,1)}(x_0,w)$, a contradiction. Thus, $C(x_0,w)\subseteq(\Gamma_{1,3}\cup\Gamma_{3,1})(w)$. If $(x_0,y,w)$ or $(x_0,z,w)$ is a pure path, from Lemma \ref{fenlei}, then $(x_0,w)\in\Gamma_{2,2}$ since $C(x_0,w)\setminus\{y,z\}\subseteq\Gamma_{3,1}(x_0)$, which implies $p^{(2,2)}_{(1,3),(1,3)}\neq0$, contrary to Lemma \ref{22} \ref{225} and Lemma \ref{1313-2*}. Then $y,z\in P_{(1,3),(3,1)}(x_0,w)$, and so $y_1,z_1\in P_{(3,1),(1,3)}(x_0,w)$.

Since $y_2=x_0(\{a_1,a_2),(b_1,b_2)\}$, from \eqref{1313-2gongshi3}, one gets $(y_2,w)\in\Sigma_1$. If $(y_2,w)\in A(\Gamma)$, then $(x_1,y_2,w,z)$ is a circuit with $(x_1,y_2),(w,z),(z,x_1)\in\Gamma_{1,3}$, which implies  $(w,x_1)\in\Gamma_{2,2}$, and so $p^{(2,2)}_{(1,3),(1,3)}\neq0$; if $(w,y_2)\in A(\Gamma)$, then $(y_2,x_2,z_1,w)$ is a circuit with $(y_2,x_2),(x_2,z_1),(z_1,w)\in\Gamma_{1,3}$, which implies  $(y_2,z_1)\in\Gamma_{2,2}$, and so $p^{(2,2)}_{(1,3),(1,3)}\neq0$, both contrary to  Lemma \ref{22} \ref{225} and Lemma \ref{1313-2*}. 

\begin{step}
{\rm  Show that $p^{(1,2)}_{(2,1),(2,1)}=1$.}
\end{step}

Suppose for the contrary that $p^{(1,2)}_{(2,1),(2,1)}>1$. Let $z\in P_{(2,1),(2,1)}(x_2,x_0)$ with $z\neq x_1$.   Since $(x_2,x_0,z)$ is a circuit in $\Gamma$ and $y_1\in P_{(1,3),(1,3)}(x_2,x_0)$, from Lemma \ref{1313-2*} \ref{1313-2*2}, we have $(y_1,z)\in\Sigma_2$. By Lemma \ref{1313-2*} \ref{1313-2*3},  there exist $z_1,z_2\in C(y_1,z)$ such that $z_1\in P_{(1,3),(1,3)}(x_0,z)\cap\Sigma_2(x_2)$ and $z_2\in P_{(1,3),(1,3)}(z,x_2)\cap\Sigma_2(x_0)$. Note that $x_2=x_0(a_i,b_j)$ and $y_1=x_0(a_2,b_2)$. Since $z\in\Sigma_1(x_0)\cap\Sigma_1(x_2)\cap\Sigma_2(y_1)$  with $z\neq x_1$, from Lemma \ref{sigma1}, we have \begin{align}
&(i,j)=(1,2),~z=x_0(a_1,b_3),~b_3\in X\setminus(x_0\cup\{b_1,b_2\}),~\mbox{or}\tag{4.15}\label{1313-2gongshi4}\\
&(i,j)=(2,1),~z=x_0(a_3,b_1),~a_3\in x_0\setminus\{a_1,a_2\}.\tag{4.16}\label{1313-2gongshi5}
 \end{align}
Since $z_1,z_2\in C(y_1,z)$ with $z_1\in\Sigma_2(x_2)$ and $z_2\in\Sigma_2(x_0)$, we get
\begin{align}
&(i,j)=(1,2),~z_1=x_0(a_2,b_3)~\mbox{and}~z_2=x_0(\{a_1,a_2\},\{b_2,b_3\}),~\mbox{or}\nonumber\\
&(i,j)=(2,1),~z_1=x_0(a_3,b_2)~\mbox{and}~z_2=x_0(\{a_2,a_3\},\{b_1,b_2\}).\nonumber
\end{align}


Since $y=x_0(a_j,b_i)$, from \eqref{1313-2gongshi4} and \eqref{1313-2gongshi5}, one gets $(y,z)\in\Sigma_2$ and $C(y,z)=\{x_0,x_1,z_1,w\}$, where  $$w=
\begin{cases}
x_0(\{a_1,a_2\},\{b_1,b_3\}),~(i,j)=(1,2),\\
x_0(\{a_1,a_3\},\{b_1,b_2\}),~(i,j)=(2,1).
\end{cases}$$ The fact $x_0\in P_{(3,1),(1,2)}(y,z)$ implies that there exists $u\in\{x_1,z_1,w\}$ such that $u\in P_{(1,2),(3,1)}(y,z)$. Since $(y,x_1)\in\Gamma_{1,3}$ and $(z_1,z)\in\Gamma_{1,3}$, we have $u=w$.

Since $y_2=x_0(\{a_1,a_2\},\{b_1,b_2\})$, by \eqref{1313-2gongshi4} and \eqref{1313-2gongshi5}, we get $(y_2,z)\in\Sigma_2$ and $C(y_2,z)=\{x_1,x_2,z_2,w\}$. Since  $x_2\in P_{(1,3),(2,1)}(y_2,z)$, there exists  $v\in\{x_1,z_2,w\}$ such that $v\in P_{(2,1),(1,3)}(y_2,z)$, contrary to the fact that $(y_2,x_1),(w,z),(z_2,z)\in\Gamma_{3,1}$.
\end{proof}

\begin{lemma}\label{1313}
Suppose that  $p^{(1,2)}_{(1,3),(1,3)}\neq0$. If $(s,t)\neq(3,3)$, then $p^{(1,2)}_{(1,s),(1,t)}=0$.
\end{lemma}
\begin{proof}
Assume the contrary, namely, $p^{(1,2)}_{(1,s),(1,t)}\neq0$. Let $z\in P_{(1,s),(1,t)}(x_0,x_1)$. Lemma \ref{1313-2*} implies that $(1,2)$ is pure. Since $(x_0,z,x_1,x_2)$ is a circuit in $\Gamma$,  we may assume $s\in\{1,2\}$ and $t\in\{2,3\}$. Since $x_0\in P_{(1,2),(1,s)}(x_2,z)$, from Lemmas \ref{22} and  \ref{1313-2*}, one has $(x_2,z)\notin\Gamma_{2,2}$. Since $(z,x_1,x_2,x_0)$ is a circuit in $\Gamma$, we get $(z,x_2)\in\Sigma_1$,  and so $(z,x_2)\in A(\Gamma)$ or $(x_2,z)\in A(\Gamma)$.

Note that $(1,2)$ is pure. Since $(x_0,z,x_2)$ or $(z,x_1,x_2)$ is a circuit in $\Gamma$ with $(x_2,x_0),(x_1,x_2)\in\Gamma_{1,2}$,  we get $(x_0,z),(z,x_2)\in\Gamma_{1,2}$ or $(z,x_1),(x_2,z)\in\Gamma_{1,2}$. Then $z,x_1\in P_{(2,1),(2,1)}(x_2,x_0)$ or $z,x_0\in P_{(2,1),(2,1)}(x_1,x_2)$, contrary to Lemma \ref{1313-2}.
\end{proof}

Now we are ready to prove Proposition \ref{arcpure} \ref{arcpure2}.

\begin{proof}[Proof of Proposition \ref{arcpure} \ref{arcpure2}]
Since $p^{(1,2)}_{(1,3),(1,3)}\neq0$, we have $\{3,4\}\subseteq T$. Then it suffices to show that $\max T=4$.

Assume the contrary, namely,  $q$ is the minimal integer such that $q+1\in T$ with $q\geq4$.  According to Lemma \ref{arc}, we get $p^{(1,q-1)}_{(1,q),(1,h)}\neq0$ for some $h\geq1$, which implies $q\in T$. By the minimality of $q$,  we have $q=4$. By Lemma \ref{arc} again, one has $p^{(1,3)}_{(1,4),(1,s)}=p^{(1,3)}_{(1,s),(1,4)}\neq0$ for some $s\geq1$.

Let $z_1\in P_{(1,s),(1,4)}(x_0,y)$. Suppose $(z_1,x_1)\in\Sigma_1$. Since $4=\partial_{\Gamma}(y,z_1)\leq1+\partial_{\Gamma}(x_1,z_1)$, one gets $(z_1,x_1)\in\Gamma_{1,t}$ for $t\geq3$, and so $z_1\in P_{(1,s),(1,t)}(x_0,x_1)$.  It follows from Lemma \ref{1313} that $s=t=3$. Then $z_1,y\in P_{(1,3),(1,3)}(x_0,x_1)$,  contrary to Lemma \ref{1313-2}. Thus, $(z_1,x_1)\in\Sigma_2$. Since $(z_1,y,x_1,x_2,x_0)$ is a circuit with  $(z_1,y)\in\Gamma_{1,4}$, we get $(z_1,x_1)\in\Gamma_{2,3}$.  Since $y\in P_{(1,4),(1,3)}(z_1,x_1)$ and  $x_0\in P_{(s,1),(1,2)}(z_1,x_1)$, from Lemma \ref{fenlei},  one obtains
\begin{align}\label{jiaochashu1}
p^{(2,3)}_{(1,4),(1,3)}=p^{(2,3)}_{(1,3),(1,4)}=p^{(2,3)}_{(s,1),(1,2)}=p^{(2,3)}_{(1,2),(s,1)}=1.\tag{4.17}
\end{align}

Let  $z_2\in P_{(1,s),(1,4)}(y,x_1)$.   If $(x_0,z_2)\in A(\Gamma)$, then $(x_0,z_2,x_1,x_2)$ is a circuit with $(z_2,x_1)\in\Gamma_{1,4}$, a contradiction; if $(z_2,x_0)\in A(\Gamma)$, then $(z_2,x_0,y)$ is a circuit with $(x_0,y)\in\Gamma_{1,3}$, a contradiction. Thus, $(x_0,z_2)\notin\Sigma_1$, and so $(x_0,z_2)\in\Sigma_2$. Since $x_1\in P_{(1,2),(4,1)}(x_0,z_2)$, from Lemma \ref{22}, we have $(x_0,z_2)\notin\Gamma_{2,2}$. The fact that $(x_0,y,z_2,x_1,x_2)$ is a circuit in $\Gamma$ implies $(x_0,z_2)\in\Gamma_{2,3}$. Since $x_1\in P_{(1,2),(4,1)}(x_0,z_2)$ and $c_2=4$, from \eqref{a1c2} and \eqref{jiaochashu1}, one gets $s=4$. Then $(z_2,x_1,x_2,x_0,y)$ is a circuit with $(z_2,x_1),(y,z_2)\in\Gamma_{1,4}$, and so $(z_2,x_2)\in\Gamma_{2,3}$. Since $x_1\in P_{(1,4),(1,2)}(z_2,x_2)$, we obtain $p^{(2,3)}_{(1,4),(1,2)}\neq0$, contrary to \eqref{jiaochashu1}.
\end{proof}

\section{One type of arcs}

Recall that  $\Sigma$ denotes a Johnson graph $J(n,e)$ with $n\geq2e$ and $e\geq2$, or a folded Johnson graph $\bar{J}(2e,e)$ with $e\geq4$. Let $m=n$ if $\Sigma$ is a Johnson graph $J(n,e)$, and $m=2e$ if $\Sigma$ is a folded Johnson graph $\bar{J}(2e,e)$.  Set ${X\choose e}$ to be the vertex set of $\Gamma$, where $X$ is an $m$-set.

The main result of this section is as follows.

\begin{prop}\label{e=2}
If $T=\{3\}$, then $\Sigma$ is the Johnson graph $J(m,2)$.
\end{prop}

In the remainder of this section, we always assume that $T=\{3\}$. Next we divide this section into two subsections.

\subsection{Auxiliary lemmas of Proposition \ref{e=2}}

In this subsection, we prove some lemmas which are used frequently in the proof of Proposition \ref{e=2}.

\begin{lemma}\label{e=2s}
Let $|\Gamma_{\tilde{h}^*}(x)\cap Y_i(x,y)|=s_i$ for $i\in\{1,2\}$ with $(x,y)\in\Gamma_{\tilde{h}}$ and $\tilde{h}\in\{(1,2),(2,1)\}$. Then $$(e-2s_1)[4p^{(1,2)}_{(1,2),(1,2)}-(m+e-4)]=0~\mbox{and}~(m-e-2s_2)[4p^{(1,2)}_{(1,2),(1,2)}-(2m-e-4)]=0.$$
\end{lemma}
\begin{proof}
Since $|\Gamma_{\tilde{h}^*}(x)\cap Y_i(x,y)|=s_i$, we have
\begin{align}\tag{5.1}\label{gongshisi1}
|\Gamma_{\tilde{h}}(x)\cap Y_i(x,y)|=|Y_i(x,y)|-s_i.\end{align}  Lemma \ref{intersection numners} \ref{intersection numners 2} implies $$p^{(1,2)}_{(1,2),(1,2)}=p^{(1,2)}_{(1,2),(2,1)}=p^{(1,2)}_{(2,1),(1,2)}=p^{(2,1)}_{(1,2),(2,1)}=p^{(2,1)}_{(2,1),(1,2)}=p^{(2,1)}_{(2,1),(2,1)}$$ and $p^{(2,1)}_{(1,2),(1,2)}=p^{(1,2)}_{(2,1),(2,1)}$. By Theorems \ref{johnson} and \ref{folded johnson}, we have $a_1=m-2$. Since $T=\{3\}$, from \eqref{a1c2}, we get $$m-2=p^{(1,2)}_{(1,2),(1,2)}+p^{(1,2)}_{(1,2),(2,1)}+p^{(1,2)}_{(2,1),(1,2)}+p^{(1,2)}_{(2,1),(2,1)}=3p^{(1,2)}_{(1,2),(1,2)}+p^{(1,2)}_{(2,1),(2,1)}.$$

Since $y\in\Sigma_1(x)$, there exist $a\in x$ and $b\in X\setminus x$ such that  $y=x(a,b)$. Let $w\in\Gamma_{\tilde{h}}(x)$.  Then $w=x(a',b')$ for some $a'\in x$ and $b'\in X\setminus x$.   If $a'=a$ or $b'=b$ with $(a,b)\neq(a',b')$, then $w\in\Sigma_1(y)$. If $a'\neq a$ and $b'\neq b$, then $w\in\Sigma_1(z)$ for $z\in\{x(a',b),x(a,b')\}$ with $x(a',b)\in Y_1(x,y)$ and $x(a,b')\in Y_2(x,y)$. If $(a',b')=(a,b)$, then $w\in\Sigma_1(z)$ for $z\in\{x(a'',b),x(a,b'')\}$ with $x(a'',b)\in Y_1(x,y)$ and $x(a,b'')\in Y_2(x,y)$, where $a''\in x\setminus\{a\}$ and $b''\in X\setminus(x\cup\{b\})$. Since $w\in\Gamma_{\tilde{h}}(x)$ was arbitrary, we obtain
\begin{align}\tag{5.2}\label{k12--1}
\Gamma_{\tilde{h}}(x)=\cup_{z\in Y_i(x,y)\cup\{y\}}(\Gamma_{\tilde{h}}(x)\cap\Sigma_1(z))
\end{align}
for $i\in\{1,2\}$.

By Lemma \ref{xyz},  we have $y\in Y_i(x,z)$ for all $z\in Y_i(x,y)$, which implies $Y_j(x,z)\cap Y_i(x,y)=\emptyset$, where $\{i,j\}=\{1,2\}$. By Lemma \ref{xyz} again, we get $Y_j(x,y)\cap(Y_1(x,z)\cup Y_2(x,z))=\emptyset$ and $Y_i(x,z)\setminus Y_i(x,y)=\{y\}$. It follows  that
\begin{align}
(Y_1(x,z)\cup Y_2(x,z))\setminus(Y_1(x,y)\cup Y_2(x,y))&=(Y_1(x,z)\cup Y_2(x,z))\setminus Y_i(x,y)\nonumber\\
&=Y_j(x,z)\cup(Y_i(x,z)\setminus Y_i(x,y))\nonumber\\
&=Y_j(x,z)\cup\{y\}\nonumber
\end{align}
for all $z\in Y_i(x,y)$. By Lemma \ref{sigma1}, we get
\begin{equation}\tag{5.3}\label{k12--0}
\begin{aligned}
&(\Gamma_{\tilde{h}}(x)\cap\Sigma_1(z))\setminus(\Gamma_{\tilde{h}}(x)\cap\Sigma_1(y))\\
=&(\Gamma_{\tilde{h}}(x)\cap(Y_1(x,z)\cup Y_2(x,z)))\setminus(\Gamma_{\tilde{h}}(x)\cap(Y_1(x,y)\cup Y_2(x,y)))\\
=&\Gamma_{\tilde{h}}(x)\cap[(Y_1(x,z)\cup Y_2(x,z))\setminus(Y_1(x,y)\cup Y_2(x,y))]\\
=&\Gamma_{\tilde{h}}(x)\cap(Y_j(x,z)\cup\{y\})
\end{aligned}
\end{equation}
for all $z\in Y_i(x,y)$. Since $y\in\Gamma_{\tilde{h}}(x)$, by \eqref{k12--1} and \eqref{k12--0}, we obtain
\begin{equation}\tag{5.4}\label{k12-1}
\begin{aligned}
\Gamma_{\tilde{h}}(x)&=\cup_{z\in Y_{i}(x,y)}[(\Gamma_{\tilde{h}}(x)\cap\Sigma_1(z))\setminus(\Gamma_{\tilde{h}}(x)\cap\Sigma_1(y))]\cup(\Gamma_{\tilde{h}}(x)\cap\Sigma_1(y))\\
&=\cup_{z\in Y_{i}(x,y)}(\Gamma_{\tilde{h}}(x)\cap Y_j(x,z))\cup(\Gamma_{\tilde{h}}(x)\cap\Sigma_1(y))\cup\{y\}.
\end{aligned}
\end{equation}

By \eqref{k12--0}, we have $(\Gamma_{\tilde{h}}(x)\cap Y_j(x,z))\cap(\Gamma_{\tilde{h}}(x)\cap\Sigma_1(y))=\emptyset$ for all $z\in Y_i(x,y)$.
For all $z\in Y_i(x,y)$, from Lemma \ref{xyz}, we obtain $Y_i(x,z)=Y_i(x,y)\cup\{y\}\setminus\{z\}$, which implies $z'\in Y_i(x,z)$ with $z'\in Y_i(x,y)\setminus\{z\}$, and so $Y_j(x,z)\cap Y_j(x,z')=\emptyset$.  Since $y\in\Gamma_{\tilde{h}}(x)$ and $y\in Y_i(x,z)$ for all $z\in Y_i(x,y)$, from \eqref{k12-1}, we get
\begin{align}\tag{5.5}\label{inter*}
k_{\tilde{h}}=&\sum_{z\in Y_{i}(x,y)}|\Gamma_{\tilde{h}}(x)\cap Y_j(x,z)|+|\Gamma_{\tilde{h}}(x)\cap\Sigma_1(y)|+1.
\end{align}

Since $y\in\Gamma_{\tilde{h}}(x)$ and $Y_i(x,z)=Y_i(x,y)\cup\{y\}\setminus\{z\}$ for all $z\in Y_i(x,y)$, from Lemma \ref{sigma1}, one has
\begin{align}\label{k12-3}
|\Gamma_{\tilde{h}}(x)\cap Y_j(x,z)|&=|(\Gamma_{\tilde{h}}(x)\cap\Sigma_1(z))|-|(\Gamma_{\tilde{h}}(x)\cap Y_i(x,z))|\nonumber\\
&=|(\Gamma_{\tilde{h}}(x)\cap\Sigma_1(z))|-|(\Gamma_{\tilde{h}}(x)\cap(Y_i(x,y)\cup\{y\}\setminus\{z\}))|\tag{5.6}\\
&=|(\Gamma_{\tilde{h}}(x)\cap\Sigma_1(z))|-(|\Gamma_{\tilde{h}}(x)\cap Y_i(x,y)|-|\Gamma_{\tilde{h}}(x)\cap\{z\}|+1).\nonumber
\end{align}

Since $T=\{3\}$, we get $\Gamma_{\tilde{h}}(x)\cap\Sigma_1(z)=P_{\tilde{h},(1,2)}(x,z)\cup P_{\tilde{h},(2,1)}(x,z)$  for all $z\in Y_i(x,y)$.
If $z\in\Gamma_{\tilde{h}}(x)\cap Y_i(x,y),$ by \eqref{gongshisi1} and \eqref{k12-3}, then
\begin{equation}\tag{5.7}
\begin{aligned}\label{gamma12x-1}
|\Gamma_{\tilde{h}}(x)\cap Y_j(x,z)|&=|P_{\tilde{h},(1,2)}(x,z)\cup P_{\tilde{h},(2,1)}(x,z)|-(|Y_i(x,y)|-s_i)\\
&=p^{\tilde{h}}_{\tilde{h},(1,2)}+p^{\tilde{h}}_{\tilde{h},(2,1)}-(|Y_i(x,y)|-s_i)\\
&=2p^{(1,2)}_{(1,2),(1,2)}-(|Y_i(x,y)|-s_i).
\end{aligned}
\end{equation}
Note that $m-2=3p^{(1,2)}_{(1,2),(1,2)}+p^{(1,2)}_{(2,1),(2,1)}$. If $z\in\Gamma_{\tilde{h}^*}(x)\cap Y_i(x,y),$ by \eqref{gongshisi1} and \eqref{k12-3}, then
\begin{equation}\tag{5.8}
\begin{aligned}\label{gamma12x-2}
|\Gamma_{\tilde{h}}(x)\cap Y_j(x,z)|&=|P_{\tilde{h},(1,2)}(x,z)\cup P_{\tilde{h},(2,1)}(x,z)|-(|Y_i(x,y)|-s_i+1)\\
&=p^{\tilde{h}^*}_{\tilde{h},(1,2)}+p^{\tilde{h}^*}_{\tilde{h},(2,1)}-(|Y_i(x,y)|-s_i+1)\\
&=(m-2)-2p^{(1,2)}_{(1,2),(1,2)}-(|Y_i(x,y)|-s_i+1).
\end{aligned}
\end{equation}

By Lemma \ref{sigma1}, we get    $|\Gamma_{\tilde{h}}(x)\cap\Sigma_1(y)|=|P_{\tilde{h},(1,2)}(x,y)\cup P_{\tilde{h},(2,1)}(x,y)|=p^{\tilde{h}}_{\tilde{h},(1,2)}+p^{\tilde{h}}_{\tilde{h},(2,1)}=2p^{(1,2)}_{(1,2),(1,2)}$. Since $T=\{3\}$, one has $z\in\Gamma_{\tilde{h}}(x)$ or $z\in\Gamma_{\tilde{h}^*}(x)$ for all $z\in Y_i(x,y)$. By \eqref{inter*}, \eqref{gamma12x-1} and \eqref{gamma12x-2}, we have
\begin{align}\label{zq}
k_{\tilde{h}}=&|\Gamma_{\tilde{h}}(x)\cap Y_{i}(x,y)|(2p_{(1,2),(1,2)}^{(1,2)}-|Y_i(x,y)|+s_i)\nonumber\\&+|\Gamma_{\tilde{h}^*}(x)\cap Y_{i}(x,y)|(m-3-2p_{(1,2),(1,2)}^{(1,2)}-|Y_i(x,y)|+s_i)+2p_{(1,2),(1,2)}^{(1,2)}+1.\tag{5.9}
\end{align}
By setting $i=1$ and $i=2$ in \eqref{zq} respectively, from Lemma \ref{sigma1} and \eqref{gongshisi1},  we get
\begin{align}
k_{\tilde{h}}=&2(e-2s_1)p^{(1,2)}_{(1,2),(1,2)}+s_1m-e(e-s_1)+2(e-2s_1),\tag{5.10}\label{k12-4}\\
k_{\tilde{h}}=&2(m-e-2s_2)p^{(1,2)}_{(1,2),(1,2)}-(m-2)(m-e-2s_2)+e(m-e-s_2).\tag{5.11}\label{k12-5}
\end{align}


In view of Theorems \ref{johnson} and \ref{folded johnson}, we have $b_0=e(m-e)$, and so $k_{\tilde{h}}=e(m-e)/2$.  Substituting $k_{\tilde{h}}$ into \eqref{k12-4} and \eqref{k12-5}, the desired results hold. 
\end{proof}

\begin{lemma}\label{inter}
$p^{(1,2)}_{(1,2),(1,2)}=m/4-1$ and $p^{(1,2)}_{(2,1),(2,1)}=m/4+1$.  Moreover, if $(x,y)\in\Sigma_1$ and $\tilde{h}\in\{(1,2),(2,1)\}$, then
\begin{align}
|\Gamma_{\tilde{h}}(x)\cap(Y_i(x,y)\cup\{y\})|=(|Y_i(x,y)|+1)/2=\left\{
\begin{array}{ll}
e/2, & \textrm{if}\ i=1,\\
(m-e)/2, & \textrm{if}\ i=2.
\end{array} \right.\nonumber
\end{align}
\end{lemma}
\begin{proof}
By Theorems \ref{johnson} and \ref{folded johnson}, we  have $a_1=m-2$. Since $T=\{3\}$, from \eqref{a1c2} and Lemma \ref{intersection numners} \ref{intersection numners 2}, we get $$m-2=p^{(1,2)}_{(1,2),(1,2)}+p^{(1,2)}_{(1,2),(2,1)}+p^{(1,2)}_{(2,1),(1,2)}+p^{(1,2)}_{(2,1),(2,1)}=3p^{(1,2)}_{(1,2),(1,2)}+p^{(1,2)}_{(2,1),(2,1)}.$$
Let $(x,y)\in\Gamma_{\tilde{h}}$ with $\tilde{h}\in\{(1,2),(2,1)\}$.  Set
$|\Gamma_{\tilde{h}^*}(x)\cap Y_i(x,y)|=s_i$ for  $i\in\{1,2\}$.  Then
$|\Gamma_{\tilde{h}}(x)\cap Y_i(x,y)|=|Y_i(x,y)|-s_i.$

\begin{step}\label{7.3-2}
{\rm Show that $s_2=(|Y_2(x,y)|+1)/2$.}
\end{step}

Suppose not. Lemma \ref{sigma1} implies $s_2\neq (m-e)/2$. By Lemma \ref{e=2s}, we have  $p^{(1,2)}_{(1,2),(1,2)}=(2m-e)/4-1$. Since $m-2=3p^{(1,2)}_{(1,2),(1,2)}+p^{(1,2)}_{(2,1),(2,1)}$, we get $p^{(1,2)}_{(2,1),(2,1)}=(3e-2m)/4+1$. Since $T=\{3\}$, we have $p^{(1,2)}_{(2,1),(2,1)}\geq1$, which implies $m\leq 3e/2$, contrary to the fact that $m\geq 2e$.

\begin{step}\label{7.3-3}
{\rm Show that $s_1=(|Y_1(x,y)|+1)/2$.}
\end{step}

Suppose not. Lemma \ref{sigma1} implies $s_1\neq e/2$. By Lemma \ref{intersection numners} \ref{intersection numners 2} and Lemma \ref{e=2s}, we have  $p^{(1,2)}_{(1,2),(1,2)}=p^{(1,2)}_{(1,2),(2,1)}=p^{(1,2)}_{(2,1),(1,2)}=p^{(2,1)}_{(1,2),(2,1)}=p^{(2,1)}_{(2,1),(1,2)}=p^{(2,1)}_{(2,1),(2,1)}=(m+e)/4-1$. Since $m-2=3p^{(1,2)}_{(1,2),(1,2)}+p^{(1,2)}_{(2,1),(2,1)}$, from Lemma \ref{intersection numners} \ref{intersection numners 2}, one gets $p^{(1,2)}_{(2,1),(2,1)}=p^{(2,1)}_{(1,2),(1,2)}=(m-3e)/4+1$. Since $T=\{3\}$, we have $p^{(1,2)}_{(2,1),(2,1)}\geq1$, which implies  $m\geq 3e$.

We claim that $Y_1(u_1,u_2)\subseteq P_{(1,2),(1,2)}(u_1,u_2)$ for any $(u_1,u_2)\in\Gamma_{1,2}$. Let $\{i,j\}=\{1,2\}$. By Lemma \ref{sigma1}, we have \begin{align}|\Gamma_{j,i}(u_i)\cap(Y_1(u_i,u_j)\cup Y_2(u_i,u_j))|&=|P_{(j,i),(1,2)}(u_i,u_j)|+|P_{(j,i),(2,1)}(u_i,u_j)|\nonumber\\
&=p_{(j,i),(1,2)}^{(i,j)}+p_{(j,i),(2,1)}^{(i,j)}=(m-e)/2\geq e\nonumber\end{align}
since $m\geq3e$. Lemma \ref{sigma1} implies $|Y_1(u_i,u_j)|=e-1$. It follows that there exists $z_i\in\Gamma_{j,i}(u_i)\cap Y_2(u_i,u_j)$.

Let $t_i=|\Gamma_{i,j}(u_i)\cap Y_2(u_i,u_j)|$. Since $(u_i,u_j)\in\Gamma_{i,j}$ and $z_i\in\Gamma_{j,i}(u_i)\cap Y_2(u_i,u_j)$, from Lemmas  \ref{sigma1} and \ref{xyz}, we have \begin{align}1+t_i&=|\Gamma_{i,j}(u_i)\cap(Y_2(u_i,u_j)\cup\{u_j\})|=|\Gamma_{i,j}(u_i)\cap(Y_2(u_i,z_i)\cup\{z_i\})|\nonumber\\
&\leq |P_{(i,j),(1,2)}(u_i,z_i)\cup P_{(i,j),(2,1)}(u_i,z_i)|=p_{(i,j),(1,2)}^{(j,i)}+p_{(i,j),(2,1)}^{(j,i)}=(m-e)/2.\nonumber\end{align} Then $t_i\leq(m-e)/2-1$.

Let $t'_i=|\Gamma_{i,j}(u_i)\cap Y_1(u_i,u_j)|$. By Lemma \ref{sigma1}, we have  \begin{align}|\Gamma_{i,j}(u_i)\cap(Y_1(u_i,u_j)\cup Y_2(u_i,u_j))|&=|P_{(i,j),(1,2)}(u_i,u_j)|+|P_{(i,j),(2,1)}(u_i,u_j)|\nonumber\\
&=p^{(i,j)}_{(i,j),(1,2)}+p^{(i,j)}_{(i,j),(2,1)}=(m+e)/2-2.\nonumber\end{align} Then $t'_i=(m+e)/2-2-t_i\geq e-1$. The fact $|Y_1(u_i,u_j)|=e-1$  implies $t'_i=e-1$, and so $Y_1(u_i,u_j)\subseteq\Gamma_{i,j}(u_i)$. Since $\{i,j\}=\{1,2\}$, from Lemma \ref{iwuguan}, we get $Y_1(u_1,u_2)\subseteq\Gamma_{1,2}(u_1)\cap\Gamma_{2,1}(u_2)$. Thus, our claim holds.

Let $(x',y')\in\Gamma_{1,2}$ and $z\in Y_1(x',y')$. By the claim, we have $z\in P_{(1,2),(1,2)}(x',y')$, which implies $Y_1(x',z)\subseteq P_{(1,2),(1,2)}(x',z)$. In view of Lemma \ref{xyz}, we get $y'\in Y_1(x',z)\subseteq P_{(1,2),(1,2)}(x',z)$, contrary to the fact $(z,y')\in\Gamma_{1,2}$.

\begin{step}
{\rm We prove this lemma.}
\end{step}

By Lemma \ref{sigma1}, we get \begin{align}2p^{(1,2)}_{(1,2),(1,2)}&=p^{\tilde{h}}_{\tilde{h},(1,2)}+p^{\tilde{h}}_{\tilde{h},(2,1)}=|P_{\tilde{h},(1,2)}(x,y)|+|P_{\tilde{h},(2,1)}(x,y)|\nonumber\\
&=|\Gamma_{\tilde{h}}(x)\cap(Y_1(x,y)\cup Y_2(x,y))|=|\Gamma_{\tilde{h}}(x)\cap Y_1(x,y)|+|\Gamma_{\tilde{h}}(x)\cap Y_2(x,y)|\nonumber\\
&=(e-s_1-1)+(m-e-s_2-1)=m-s_1-s_2-2.\nonumber\end{align} By Lemma \ref{sigma1} and Steps \ref{7.3-2}, \ref{7.3-3},  we get $2p^{(1,2)}_{(1,2),(1,2)}=m/2-2$, and so $p^{(1,2)}_{(1,2),(1,2)}=m/4-1$. The fact $m-2=3p^{(1,2)}_{(1,2),(1,2)}+p^{(1,2)}_{(2,1),(2,1)}$ implies $p^{(1,2)}_{(2,1),(2,1)}=m/4+1$. Then the first statement holds.

Since $y\in\Gamma_{\tilde{h}}(x)$, from Steps \ref{7.3-2} and \ref{7.3-3}, one obtains $|\Gamma_{\tilde{h}}(x)\cap(Y_i(x,y)\cup\{y\})|=|Y_i(x,y)|-s_i+1=(|Y_i(x,y)|+1)/2$ for $i\in\{1,2\}$.  The second statement follows from Lemma \ref{sigma1}.
\end{proof}

\begin{lemma}\label{yi}
Let $(x,y)\in\Gamma_{1,2}$. Then $|P_{(1,2),(2,1)}(x,y)\cap Y_i(x,y)|=|P_{(2,1),(1,2)}(x,y)\cap Y_i(x,y)|$  and $|P_{(2,1),(2,1)}(x,y)\cap Y_i(x,y)|=|P_{(1,2),(1,2)}(x,y)\cap Y_i(x,y)|+1>0$ for $i\in\{1,2\}$.
\end{lemma}
\begin{proof}
Let $\{h,l\}=\{1,2\}$. Since $y\in\Gamma_{1,2}(x)$, we get $|\Gamma_{h,l}(x)\cap Y_i(x,y)|+(l-1)=|\Gamma_{h,l}(x)\cap(Y_i(x,y)\cup\{y\})|=(|Y_i(x,y)|+1)/2$ from Lemma \ref{inter}. Set $|P_{(h,l),(l,h)}(x,y)\cap Y_i(x,y)|=s_h.$  Since $T=\{3\}$, we have $|P_{(h,l),(h,l)}(x,y)\cap Y_i(x,y)|+|P_{(h,l),(l,h)}(x,y)\cap Y_i(x,y)|=|\Gamma_{h,l}(x)\cap Y_i(x,y)|$, which implies
\begin{align}\tag{5.12}\label{yi1}
|P_{(h,l),(h,l)}(x,y)\cap Y_i(x,y)|=(|Y_i(x,y)|+1)/2-s_h-(l-1).
\end{align}

Since $x\in\Gamma_{2,1}(y)$, one gets $|\Gamma_{1,2}(y)\cap Y_i(y,x)|=|\Gamma_{1,2}(y)\cap(Y_i(y,x)\cap\{x\})|=(|Y_i(y,x)|+1)/2$ from Lemma \ref{inter}. Note that $|P_{(h,l),(l,h)}(x,y)\cap Y_i(x,y)|=s_h$. In view of Lemma \ref{iwuguan}, one has
\begin{align}
(|Y_i(x,y)|+1)/2&=(|Y_i(y,x)|+1)/2\nonumber\\
                &=|\Gamma_{1,2}(y)\cap Y_i(y,x)|\nonumber\\
                &=|P_{(1,2),(1,2)}(y,x)\cap Y_i(y,x)|+|P_{(1,2),(2,1)}(y,x)\cap Y_i(y,x)|\nonumber\\
                &=|P_{(2,1),(2,1)}(x,y)\cap Y_i(x,y)|+|P_{(1,2),(2,1)}(x,y)\cap Y_i(x,y)|\nonumber\\
                &=(|Y_i(x,y)|+1)/2-s_2+s_1\nonumber
\end{align}
by setting $(h,l)=(2,1)$ in \eqref{yi1}. It follows that $s_1=s_2$. Then $|P_{(1,2),(2,1)}(x,y)\cap Y_i(x,y)|=|P_{(2,1),(1,2)}(x,y)\cap Y_i(x,y)|$. \eqref{yi1} implies $|P_{(2,1),(2,1)}(x,y)\cap Y_i(x,y)|=|P_{(1,2),(1,2)}(x,y)\cap Y_i(x,y)|+1>0$.
%
\end{proof}

\begin{lemma}\label{k33}
$p^{(3,3)}_{(1,2),(2,1)}=2$.
\end{lemma}
\begin{proof}
Let $M=\{(1,2),(2,1),(2,2),(2,3),(2,4)\}$ and $N=(M\cup\{(0,0),(3,2),(3,3)\})\setminus\{(2,4)\}$. Since $T=\{3\}$, from Lemma \ref{intersection numners} \ref{intersection numners 1}, we obtain
$(k_{1,2})^2=\sum_{\tilde{i}\in\tilde{\partial}(\Gamma)}p^{\tilde{i}}_{(1,2),(1,2)}k_{\tilde{i}}=
\sum_{\tilde{i}\in M}p^{\tilde{i}}_{(1,2),(1,2)}k_{\tilde{i}}$
and
$k_{1,2}k_{2,1}=\sum_{\tilde{i}\in\tilde{\partial}(\Gamma)}p^{\tilde{i}}_{(1,2),(2,1)}k_{\tilde{i}}=
\sum_{\tilde{i}\in N}p^{\tilde{i}}_{(1,2),(2,1)}k_{\tilde{i}}.$ The fact $\sum_{\tilde{i}\in M}p^{\tilde{i}}_{(1,2),(1,2)}k_{\tilde{i}}=\sum_{\tilde{i}\in N}p^{\tilde{i}}_{(1,2),(2,1)}k_{\tilde{i}}$ implies $(p^{(1,2)}_{(1,2),(1,2)}+p^{(2,1)}_{(1,2),(1,2)}-p^{(1,2)}_{(1,2),(2,1)}-p^{(2,1)}_{(1,2),(2,1)})k_{1,2}-p^{(0,0)}_{(1,2),(2,1)}k_{0,0}=
(p^{(2,2)}_{(1,2),(2,1)}-p^{(2,2)}_{(1,2),(1,2)})k_{2,2}+p^{(3,3)}_{(1,2),(2,1)}k_{3,3}+(p^{(2,3)}_{(1,2),(2,1)}+p^{(3,2)}_{(1,2),(2,1)}-p^{(2,3)}_{(1,2),(1,2)})k_{2,3}
-p^{(2,4)}_{(1,2),(1,2)}k_{2,4}$.
By Lemma \ref{intersection numners} \ref{intersection numners 2} and Lemma \ref{inter}, we get
\begin{equation}\tag{5.13}
\begin{aligned}\label{k331}
k_{1,2}=&(p^{(2,2)}_{(1,2),(2,1)}-p^{(2,2)}_{(1,2),(1,2)})k_{2,2}+(2p^{(2,3)}_{(1,2),(2,1)}-p^{(2,3)}_{(1,2),(1,2)})k_{2,3}\\
&+p^{(3,3)}_{(1,2),(2,1)}k_{3,3}-p^{(2,4)}_{(1,2),(1,2)}k_{2,4}.
\end{aligned}
\end{equation}

In view of Lemma \ref{22} \ref{222} and \ref{225}, one has $(p^{(2,2)}_{(1,2),(2,1)}-p^{(2,2)}_{(1,2),(1,2)})k_{2,2}\leq0$. Note that $c_2=4$.
If $k_{2,3}\neq0$, from Lemma \ref{fenlei} \ref{fenlei2} and \ref{fenlei6}, then $p^{(2,3)}_{(1,2),(1,2)}=4$ and $p^{(2,3)}_{(1,2),(2,1)}=p^{(2,3)}_{(2,1),(1,2)}=0$, or $p^{(2,3)}_{(1,2),(1,2)}=2$ and $p^{(2,3)}_{(1,2),(2,1)}=p^{(2,3)}_{(2,1),(1,2)}=1$,  which imply $(2p^{(2,3)}_{(1,2),(2,1)}-p^{(2,3)}_{(1,2),(1,2)})k_{2,3}\leq0$.  Since $k_{1,2}>0$, from \eqref{k331}, we obtain $p^{(3,3)}_{(1,2),(2,1)}\neq0$. Since $c_2=4$ and $p^{(3,3)}_{(1,2),(1,2)}=p^{(3,3)}_{(2,1),(2,1)}=0$, we get  $p^{(3,3)}_{(1,2),(2,1)}=p^{(3,3)}_{(2,1),(1,2)}=2$
\end{proof}

\subsection{Proof of Proposition \ref{e=2}}

In this subsection, we give a proof of Proposition \ref{e=2}.
Assume the contrary, namely, $\Sigma=\bar{J}(2e,e)$ with $e\geq4$, or $\Sigma=J(n,e)$ with  $e>2$. Then $m\geq2e>4$.  By Lemma \ref{intersection numners} \ref{intersection numners 2} and  Lemma \ref{inter}, we get $p^{(1,2)}_{(1,2),(1,2)}=p^{(1,2)}_{(2,1),(1,2)}=p^{(2,1)}_{(2,1),(1,2)}=m/4-1>0.$

\begin{stepp}\label{22cunzai}
{\rm Show that $p_{(1,2),(1,2)}^{(2,2)}\neq0$.}
\end{stepp}

Let $(x,y)\in\Gamma_{1,2}$. Since $p^{(1,2)}_{(1,2),(1,2)}>0$, there exists $z\in P_{(1,2),(1,2)}(x,y)$. Then $z\in Y_i(x,y)$ for some $i\in\{1,2\}$. By Lemma \ref{yi},  there exists $w\in P_{(2,1),(2,1)}(x,y)\cap Y_j(x,y)$ with $j\in\{1,2\}\setminus\{i\}$. Lemma \ref{sigma1} \ref{sigma12} implies $(z,w)\in\Sigma_2$. Since $y\in P_{(1,2),(1,2)}(z,w)$ and $x\in P_{(1,2),(1,2)}(w,z)$, one has $(z,w)\in\Gamma_{2,2}$, and so $p_{(1,2),(1,2)}^{(2,2)}\neq0$.

%

\begin{stepp}\label{eleq6}
{\rm Let $(x,x')\in\Gamma_{2,2}$ and $y\in P_{(1,2),(1,2)}(x',x)$.  Show that $|\Gamma_{1,2}(y)\cap Y_1(x,y)\cap \Sigma_2(x')|=1$ and $m = 2e$ when $\Gamma_{1,2}(y)\cap Y_1(x,y)\cap \Sigma_2(x')\neq\emptyset$.}
\end{stepp}
Since $(x,x')\in\Sigma_2$, there exist distinct elements $a_1,a_2\in x$ and $b_1,b_2\in X\setminus x$ such that $x'=x(\{a_1,a_2\},\{b_1,b_2\})$. Since $y\in C(x,x')$, we may assume $y=x(a_1,b_1)$.

We claim that $m=2e$, $z\in\Gamma_{2,1}(x)$ and $|P_{(2,1),(1,2)}(x,z)\cap Y_2(x,z)\cap\Sigma_3(x')|=m/4-2$ for all $z\in\Gamma_{1,2}(y)\cap Y_1(x,y)$ with $z\neq x(a_2,b_1)$.


Let $z=x(a,b_1)$ with $a\in x\setminus\{a_1,a_2\}$. By Step \ref{22cunzai}, there exists $y'\in P_{(1,2),(1,2)}(x,x')$. If there exists a vertex $w:=x(a,b)\in P_{(2,1),(2,1)}(x,z)$ for some $b\in X\setminus(x\cup\{b_1,b_2\})$, then $(x',y,z,w)$ and $(w,x,y',x')$ are paths in $\Gamma$, which imply $(x',w)\in\Gamma_{3,3}\cap\Sigma_3$, and so $p^{(3,3)}_{(1,2),(2,1)}=0$, contrary to Lemma \ref{k33}. Thus,
\begin{align}\tag{5.14}\label{7.1-1}
\Gamma_{2,1}(x)\cap\{x(a,b)\mid b\in X\setminus(x\cup\{b_1,b_2\})\}\subseteq P_{(2,1),(1,2)}(x,z).
\end{align} 

By Lemma  \ref{inter}, we have
\begin{align}\tag{5.15}\label{7.1*}
|\Gamma_{2,1}(x)\cap(Y_2(x,z)\cup\{z\})|=|\Gamma_{2,1}(x)\cap\{x(a,c)\mid c\in X\setminus x\}|=(m-e)/2,
 \end{align}
 which implies
\begin{align}\tag{5.16}\label{7.1-2}
|\Gamma_{2,1}(x)\cap\{x(a,b)\mid b\in X\setminus(x\cup\{b_1,b_2\})\}|\geq (m-e)/2-2.
\end{align}
 Since $y\in P_{(2,1),(1,2)}(x,z)$ and $p_{(2,1),(1,2)}^{(1,2)}=p_{(2,1),(1,2)}^{(2,1)}=m/4-1$, from \eqref{7.1-1} and \eqref{7.1-2}, one gets
 \begin{equation}\tag{5.17}
 \begin{aligned}
 m/4-1=&|P_{(2,1),(1,2)}(x,z)|\\
 \geq&1+|\Gamma_{2,1}(x)\cap\{x(a,b)\mid b\in X\setminus(x\cup\{b_1,b_2\})\}|\\
 \geq&(m-e)/2-1,
 \end{aligned}
 \end{equation}
and so $2e\leq m$.
The fact $m\geq2e$ implies $m=2e$ and $|\Gamma_{2,1}(x)\cap\{x(a,b)\mid b\in X\setminus(x\cup\{b_1,b_2\})\}|=(m-e)/2-2=m/4-2$. By \eqref{7.1-1}, we get
\begin{align}m/4-2=&|\{x(a,b)\in P_{(2,1),(1,2)}(x,z)\mid b\in X\setminus(x\cup\{b_1,b_2\})\}|\nonumber\\=&|P_{(2,1),(1,2)}(x,z)\cap Y_2(x,z)\cap\Sigma_3(x')|.\nonumber\end{align}
 Since $|\Gamma_{2,1}(x)\cap\{x(a,c)\mid c\in X\setminus x\}|=m/4$ from \eqref{7.1*}, we have $x(a,b_1),x(a,b_2)\in\Gamma_{2,1}(x)$, and so $z\in\Gamma_{2,1}(x)$. Thus, the claim holds.


By the claim, it suffices to show that $|\Gamma_{1,2}(y)\cap Y_1(x,y)\cap\Sigma_2(x')|=1$. Assume the contrary, namely,
there exist distinct vertices $u_1,u_2\in\Gamma_{1,2}(y)\cap Y_1(x,y)\cap\Sigma_2(x')$. Then $u_i=x(c_i,b_1)$  for some $c_i\in x\setminus\{a_1,a_2\}$ with $i\in\{1,2\}$. The claim implies $u_i\in\Gamma_{2,1}(x)$ and
\begin{align}\tag{5.18}\label{7.1-3}
|\{x(c_i,b)\in P_{(2,1),(1,2)}(x,u_i)\mid b\in X\setminus(x\cup\{b_1,b_2\})\}|=m/4-2.
\end{align}
Note that $(u_1,u_2)\in\Sigma_1$. If $(u_1,u_2)\in\Gamma_{1,2}$, then $u_1,y\in P_{(2,1),(1,2)}(x,u_2)$, which implies $|P_{(2,1),(1,2)}(x,u_2)|\geq m/4$ from \eqref{7.1-3}; if $(u_1,u_2)\in\Gamma_{2,1}$, then $u_2,y\in P_{(2,1),(1,2)}(x,u_1)$, which implies $|P_{(2,1),(1,2)}(x,u_1)|\geq m/4$ from \eqref{7.1-3}, both contrary to the fact that $p^{(2,1)}_{(2,1),(1,2)}=m/4-1$. This completes the proof of the step. 

\begin{stepp}\label{eneq4}
{\rm Show that $P_{(1,2),(1,2)}(x,y)\subseteq Y_1(x,y)$ and $P_{(2,1),(1,2)}(x,y)\cap Y_1(x,y)=\emptyset$ for $(x,y)\in\Gamma_{1,2}$ when $e\neq4$ or $m\neq8$.}
\end{stepp}

Suppose for the contrary that there exists $z\in P_{(1,2),(1,2)}(x,y)$ such that $z\notin Y_1(x,y)$. It follows from Lemma \ref{sigma1} that $z\in Y_2(x,y)$. 
By Lemma \ref{yi}, there exists $w\in P_{(2,1),(2,1)}(x,y)\cap Y_1(x,y)$. 
Lemma \ref{sigma1} \ref{sigma12} implies $(w,z)\in\Sigma_2$. Since $x\in P_{(1,2),(1,2)}(w,z)$ and $y\in P_{(1,2),(1,2)}(z,w)$, one gets $(w,z)\in\Gamma_{2,2}$, which implies $|\Gamma_{1,2}(y)\cap Y_1(w,y)\cap\Sigma_2(z)|\leq1$ from Step \ref{eleq6}.

Since $w\in Y_1(x,y)$, from Lemma \ref{iwuguan}, we have $w\in Y_1(y,x)$. By Lemmas \ref{iwuguan} and \ref{xyz}, we get $Y_1(y,w)=Y_1(y,x)\cup\{x\}\setminus\{w\}=Y_1(x,y)\cup\{x\}\setminus\{w\}$. Since $z\in Y_2(x,y)$,  from Lemma  \ref{sigma1} \ref{sigma12}, one has $$Y_1(y,w)\cap\Sigma_2(z)=(Y_1(x,y)\cup\{x\}\setminus\{w\})\cap\Sigma_2(z)=Y_1(x,y)\setminus\{w\}=Y_1(y,w)\setminus\{x\}.$$
Since $x\in\Gamma_{2,1}(y)$ and $w\in\Gamma_{1,2}(y)$,  from Lemma  \ref{inter}, we get \begin{align}e/2&=|\Gamma_{1,2}(y)\cap(Y_1(y,w)\cup\{w\})|=1+|\Gamma_{1,2}(y)\cap Y_1(y,w)|\nonumber\\
&=1+|\Gamma_{1,2}(y)\cap Y_1(y,w)\cap\Sigma_2(z)|.\nonumber\end{align} Since $e>2$, from Lemma \ref{iwuguan}, one gets  $|\Gamma_{1,2}(y)\cap Y_1(w,y)\cap\Sigma_2(z)|=e/2-1>0$. Since $|\Gamma_{1,2}(y)\cap Y_1(w,y)\cap\Sigma_2(z)|\leq1$, one gets $|\Gamma_{1,2}(y)\cap Y_1(w,y)\cap\Sigma_2(z)|=1$, and so $e=4$. By Step \ref{eleq6}, we obtain $m=2e=8$, a contradiction. Thus, $P_{(1,2),(1,2)}(x,y)\subseteq Y_1(x,y)$.

Since $p^{(1,2)}_{(1,2),(1,2)}=m/4-1$, one gets $|P_{(1,2),(1,2)}(x,y)\cap Y_1(x,y)|=m/4-1$, which implies $|P_{(2,1),(2,1)}(x,y)\cap Y_1(x,y)|=m/4$ from Lemma \ref{yi}. By Lemma \ref{sigma1}, one has
\begin{align}m/2-1&=m/4-1+m/4\nonumber\\
&=|(P_{(1,2),(1,2)}(x,y)\cap Y_1(x,y))\cup(P_{(2,1),(2,1)}(x,y)\cap Y_1(x,y))|\leq|Y_1(x,y)|\nonumber\\
&=e-1,\nonumber
\end{align}
which implies $m=2e$. Then $(P_{(1,2),(1,2)}(x,y)\cap Y_1(x,y))\cup(P_{(2,1),(2,1)}(x,y)\cap Y_1(x,y))=Y_1(x,y)$, and so $P_{(2,1),(1,2)}(x,y)\cap Y_1(x,y)=\emptyset$.

\begin{stepp}\label{f=4}
{\rm Show that $e=4$ and $m=8$. Moreover, if $(x,y)\in\Gamma_{1,2}$ and $\{i,j\}=\{1,2\}$, then $P_{(1,2),(1,2)}(x,y)\subseteq Y_i(x,y)$ and $P_{(1,2),(2,1)}(x,y)\cup  P_{(2,1),(1,2)}(x,y)\subseteq Y_j(x,y)$. } 
\end{stepp}

Suppose $e\neq4$ or $m\neq8$. Let $(x,y)\in\Gamma_{1,2}$. Since $p^{(1,2)}_{(1,2),(1,2)}>0$, from Step \ref{eneq4}, there exists $z\in P_{(1,2),(1,2)}(x,y)\cap Y_1(x,y)$. Then $x\in P_{(2,1),(1,2)}(z,y)$ with $(z,y)\in\Gamma_{1,2}$, and so $x\in Y_2(z,y)$ from Step \ref{eneq4}. Since $z\in Y_1(x,y)$, from Lemma \ref{iwuguan}, we have $z\in Y_1(y,x)$. By Lemma \ref{xyz}, we get $x\in Y_1(y,z)$, which implies $x\in Y_1(z,y)$ from Lemma \ref{iwuguan},  a contradiction. Thus, $e=4$ and $m=8$.

By Lemma \ref{sigma1}, we obtain $|Y_h(x,y)|=3$ for all $h\in\{1,2\}$. Since $p^{(1,2)}_{(1,2),(1,2)}=m/4-1=1$, we have $\emptyset\neq P_{(1,2),(1,2)}(x,y)\subseteq Y_i(x,y)$ for some $i\in\{1,2\}$. By Lemma \ref{yi},    we get $P_{(2,1),(2,1)}(x,y)\cap Y_i(x,y)\neq\emptyset$, and $P_{(2,1),(1,2)}(x,y)\cap Y_i(x,y)\neq\emptyset$ if and only if $P_{(1,2),(2,1)}(x,y)\cap Y_i(x,y)\neq\emptyset$. Since $|Y_i(x,y)|=3$, one gets $P_{(1,2),(2,1)}(x,y)\cap Y_i(x,y)=P_{(2,1),(1,2)}(x,y)\cap Y_i(x,y)=\emptyset$, and so $P_{(1,2),(2,1)}(x,y)\cup  P_{(2,1),(1,2)}(x,y)\subseteq Y_j(x,y)$ for $j\in\{1,2\}\setminus\{i\}$.

\begin{stepp}
{\rm We reach a contradiction.}
\end{stepp}

Let $(x,y)\in\Gamma_{1,2}$. Then there exist $a_1\in x$ and $b_1\in X\setminus x$ such that $y=x(a_1,b_1)$. Since $p^{(1,2)}_{(1,2),(1,2)}>0$, there exists $y_1\in P_{(1,2),(1,2)}(x,y).$ Let $y_1\in Y_i(x,y)$ for some $i\in\{1,2\}$.  By Lemma \ref{yi}, there exists  $y_2\in P_{(2,1),(2,1)}(x,y)\cap Y_j(x,y)$ for $j\in\{1,2\}\setminus\{i\}$.
 Then there exist $a_2\in x\setminus\{a_1\}$ and $b_2\in X\setminus(x\cup\{b_1\})$ such that $\{y_1,y_2\}=\{x(a_1,b_2),x(a_2,b_1)\}$. It follows that $(y_1,y_2)\in\Sigma_2$ and $C(y_1,y_2)=\{x,y,z_1,z_2\}$ with $z_1=x(a_2,b_2)$ and $z_2=x(\{a_1,a_2\},\{b_1,b_2\})$.

Since $y,z_1\in C(x,y_1)$ and $(y,z_1)\in\Sigma_2$, from Lemma \ref{sigma1} \ref{sigma12}, we have $y\in Y_s(x,y_1)$ and $z_1\in Y_t(x,y_1)$ with $\{s,t\}=\{1,2\}$. Since $y\in P_{(1,2),(2,1)}(x,y_1)$ with $(x,y_1)\in\Gamma_{1,2}$, from Step \ref{f=4}, one has $z_1\in P_{(1,2),(1,2)}(x,y_1)$ or $P_{(2,1),(2,1)}(x,y_1)$.

\textbf{Case 1.} $z_1\in P_{(1,2),(1,2)}(x,y_1)$.

Since $\{y_1,y_2\}=\{x(a_1,b_2),x(a_2,b_1)\}$ and $z_1=x(a_2,b_2)$, we have $y_1,y_2\in C(x,z_1)$ and $(y_1,y_2)\in\Sigma_2$. Lemma \ref{sigma1} \ref{sigma12} implies $y_1\in Y_s(x,z_1)$ and $y_2\in Y_t(x,z_1)$ with $\{s,t\}=\{1,2\}$. Since  $y_1\in P_{(1,2),(2,1)}(x,z_1)$ with $(x,z_1)\in\Gamma_{1,2}$, from Step \ref{f=4}, we get $y_2\in P_{(1,2),(1,2)}(x,z_1)$ or $y_2\in P_{(2,1),(2,1)}(x,z_1)$. The fact $(x,y_2)\in\Gamma_{2,1}$ implies $y_2\in P_{(2,1),(2,1)}(x,z_1)$. Note that $ C(y_1,y_2)=\{x,y,z_1,z_2\}$. Since $x\in P_{(2,1),(2,1)}(y_1,y_2),$ $y\in P_{(1,2),(1,2)}(y_1,y_2)$ and $z_1\in P_{(2,1),(1,2)}(y_1,y_2)$, we get $z_2\in P_{(1,2),(2,1)}(y_1,y_2)$.

Since $z_2=x(\{a_1,a_2\},\{b_1,b_2\})$, we get $(x,z_2)\in\Sigma_2$,  which implies $ C(x,z_2)=\{y,y_1,y_2,z_1\}$ from Lemma \ref{sigma2}. Since $y_1\in P_{(1,2),(1,2)}(x,z_2)$, $y_2\in P_{(2,1),(1,2)}(x,z_2)$ and $y,z_1\in\Gamma_{1,2}(x)$, one gets $(x,z_2)\notin\Gamma_{2,2}$. By Lemma \ref{fenlei} \ref{fenlei6}, one has $y\in P_{(1,2),(1,2)}(x,z_2)$  or  $z_1\in P_{(1,2),(1,2)}(x,z_2)$. It follows that $z_2\in P_{(1,2),(2,1)}(y_1,y)$ and $x\in P_{(2,1),(1,2)}(y_1,y)$ with $(y_1,y)\in\Gamma_{1,2}$, or $x\in P_{(2,1),(1,2)}(z_1,y_1)$ and $z_2\in P_{(1,2),(2,1)}(z_1,y_1)$ with $(z_1,y_1)\in\Gamma_{1,2}$. By Step \ref{f=4} and Lemma \ref{sigma1} \ref{sigma11},  we get $(z_2,x)\in\Sigma_1$, a contradiction.


\textbf{Case 2.} $z_1\in P_{(2,1),(2,1)}(x,y_1)$.

Since $y\in P_{(1,2),(1,2)}(y_1,y_2)$ and $x\in P_{(2,1),(2,1)}(y_1,y_2)$, we have $(y_1,y_2)\in\Gamma_{2,2}$. Note that $ C(y_1,y_2)=\{x,y,z_1,z_2\}$. Since $y=x(a_1,b_1)$ and $z_1=x(a_2,b_2)$, one gets $(y,z_1)\in\Sigma_2$. It follows from Lemma \ref{sigma2} that $C(y,z_1)=\{x,y_1,y_2,z_2\}$.

Suppose $(z_1,y_2)\in\Gamma_{1,2}$. Then $z_1,y\in P_{(1,2),(1,2)}(y_1,y_2)$. By Lemma \ref{22} \ref{222}, we get $p^{(2,2)}_{(1,2),(1,2)}=p^{(2,2)}_{(2,1),(2,1)}=2$. 
Then $z_2\in P_{(2,1),(2,1)}(y_1,y_2)$.
Note that $(y,z_1)\in\Sigma_2$ and $C(y,z_1)=\{x,y_1,y_2,z_2\}$. Since $x\in P_{(2,1),(2,1)}(y,z_1)$, $y_1\in P_{(2,1),(1,2)}(y,z_1)$ and $y_2\in P_{(1,2),(2,1)}(y,z_1)$,  one has $(z_1,y)\in\Gamma_{2,3}$. Since $T=\{3\}$, from Lemma \ref{fenlei} \ref{fenlei6}, we get $z_2\in P_{(2,1),(2,1)}(y,z_1)$. Since $x,z_2\in P_{(2,1),(1,2)}(y_1,y)$ with $(y_1,y)\in\Gamma_{1,2}$, we have $p^{(1,2)}_{(2,1),(1,2)}=m/4-1>1$, which implies  $m>8$, contrary to Step \ref{f=4}. Thus, $(z_1,y_2)\in\Gamma_{2,1}$.

Note that $ C(y_1,y_2)=\{x,y,z_1,z_2\}$. Since $y\in P_{(1,2),(1,2)}(y_1,y_2)$, $x\in P_{(2,1),(2,1)}(y_1,y_2)$ and $z_1\in P_{(1,2),(2,1)}(y_1,y_2)$, we have
\begin{align}\tag{5.19}\label{e=2step51}z_2\in P_{(2,1),(1,2)}(y_1,y_2).\end{align}
Note that $(y,z_1)\in\Sigma_2$ and $C(y,z_1)=\{x,y_1,y_2,z_2\}$. Since $x\in P_{(2,1),(2,1)}(y,z_1)$, $y_2\in P_{(1,2),(1,2)}(y,z_1)$ and $y_1\in P_{(2,1),(1,2)}(y,z_1)$, we have
 \begin{align}\tag{5.20}\label{e=2step52}z_2\in P_{(1,2),(2,1)}(y,z_1).\end{align}
 Note that $z_1\in P_{(1,2),(1,2)}(y_2,x)$ and $y\in P_{(2,1),(2,1)}(y_2,x)$ with $(y_2,x)\in\Gamma_{1,2}$. Since $p^{(1,2)}_{(1,2),(2,1)}>0$, from Step \ref{f=4} and Lemma \ref{sigma1} \ref{sigma12}, there exists
 \begin{align}\tag{5.21}\label{e=2step53}w\in P_{(1,2),(2,1)}(y_2,x)\end{align}
 with $(w,z_1)\in\Sigma_2$ and $w\neq y$. Since $y_2\in\{x(a_1,b_2),x(a_2,b_1)\}$, $z_1=x(a_2,b_2)$ and $y=x(a_1,b_1)$, we have
 $$w=
\begin{cases}
x(a_1,b_3),~\mbox{if}~y_2=x(a_1,b_2),\\
x(a_3,b_1),~\mbox{if}~y_2=x(a_2,b_1),
\end{cases}$$
where $b_3\in X\setminus(x\cup\{b_1,b_2\})$ and $a_3\in x\setminus\{a_1,a_2\}$.

By Lemma  \ref{inter} and Step \ref{f=4}, one gets $|\Gamma_{2,1}(x)\cap(Y_i(x,w)\cup\{w\})|=2$ for $i\in\{1,2\}$. Since $w\in\Gamma_{1,2}(x)$,  we have $|\Gamma_{2,1}(x)\cap Y_i(x,w)|=2$, which implies that there exists
 $$w'=
\begin{cases}
x(a,b_3),~\mbox{if}~y_2=x(a_1,b_2),\\
x(a_3,b),~\mbox{if}~y_2=x(a_2,b_1)
\end{cases}$$
such that $w'\in\Gamma_{2,1}(x)$, where $a\in x\setminus\{a_1,a_2\}$ and $b\in X\setminus(x\cup\{b_1,b_2\})$. Then $(w,w')\in\Sigma_1$. If $(w,w')\in\Gamma_{1,2}$, from \eqref{e=2step51}, \eqref{e=2step52} and \eqref{e=2step53}, then $(z_2,y_2,w,w')$ and $(w',x,y,z_2)$ are paths in $\Gamma$, which imply $(z_2,w')\in\Gamma_{3,3}\cap\Sigma_3$, and so $p^{(3,3)}_{(1,2),(2,1)}=0$, contrary to Lemma \ref{k33}. Hence, $(w,w')\in\Gamma_{2,1}$.   \eqref{e=2step53} implies that $y_2,w'\in P_{(2,1),(1,2)}(x,w)$ with $(x,w)\in\Gamma_{1,2}$. Since $p^{(1,2)}_{(2,1),(1,2)}=m/4-1$, we get $m>8$, contrary to Step \ref{f=4}.

\section{Proofs of Theorem \ref{main}}

In this section, we shall prove our main result. Before that, we need some auxiliary lemmas.

\begin{lemma}\label{1331}
If $p^{(1,2)}_{(1,3),(1,3)}\neq0$, then  $p^{(1,2)}_{(1,3),(3,1)}=p^{(1,2)}_{(1,3),(2,1)}=0$.
\end{lemma}
\begin{proof}
According to Lemma \ref{1313-2*}, there exists a circuit $(x_0,x_1,x_2)$    consisting of arcs of type $(1,2)$. Let $y\in P_{(1,3),(1,3)}(x_0,x_1)$. In view of Lemma \ref{1313-2*} \ref{1313-2*2}, we have $(y,x_2)\in\Sigma_2$.

Suppose  $p^{(1,2)}_{(1,3),(3,1)}\neq0$.  By Lemma \ref{intersection numners} \ref{intersection numners 2}, we get $p^{(1,3)}_{(1,2),(1,3)}=p^{(1,3)}_{(1,3),(1,2)}\neq0$. Let $z_1\in  P_{(1,2),(1,3)}(x_0,y)$ and $z_2\in  P_{(1,3),(1,2)}(x_0,y)$.  Since $3=\partial_{\Gamma}(x_1,y)\leq\partial_{\Gamma}(x_1,z_i)+1$ for $i\in\{1,2\}$,  one has $\partial_{\Gamma}(x_1,z_i)=r_i$ for some $r_i\geq2$. If $(z_1,x_1)\in\Gamma_{1,r_1}$, then $z_1\in P_{(1,2),(1,r_1)}(x_0,x_1)$, which implies $p^{(1,2)}_{(1,2),(1,2)}\neq0$ or $p^{(1,2)}_{(1,2),(1,3)}\neq0$ from Proposition \ref{arcpure} \ref{arcpure2}, contrary to Lemma \ref{1313}. If $(z_2,x_1)\in\Gamma_{1,r_2}$, then $z_2\in P_{(1,3),(1,r_2)}(x_0,x_1)$, which implies $p^{(1,2)}_{(1,3),(1,2)}\neq0$, or $z_2,y\in P_{(1,3),(1,3)}(x_0,x_1)$ from Proposition \ref{arcpure} \ref{arcpure2}, contrary to  Lemma \ref{1313} or \ref{1313-2}. Thus, $(z_i,x_1)\in\Sigma_2$ for all $i\in\{1,2\}$.

Since $y\in P_{(1,3),(1,3)}(z_1,x_1)$ and $x_0\in P_{(2,1),(1,2)}(z_1,x_1)$, from Lemma \ref{22} \ref{223} and Lemma \ref{1313-2*}, one obtains $(z_1,x_1)\in\Gamma_{2,3}$. Observe that $y\in P_{(1,2),(1,3)}(z_2,x_1)$ and $x_0\in P_{(3,1),(1,2)}(z_2,x_1)$. It follows from Lemma \ref{22} that  $(z_2,x_1)\in\Gamma_{2,3}$.
Then $p^{(2,3)}_{(1,3),(1,3)}p^{(2,3)}_{(2,1),(1,2)}p^{(2,3)}_{(1,2),(2,1)}p^{(2,3)}_{(3,1),(1,2)}p^{(2,3)}_{(1,2),(3,1)}\neq0$, and so $c_2\geq5$ from \eqref{a1c2}, a contradiction. Therefore, $p^{(1,2)}_{(1,3),(3,1)}=0$.  Lemma \ref{intersection numners} \ref{intersection numners 2} implies $p^{(1,3)}_{(1,3),(1,2)}=p^{(1,3)}_{(1,2),(1,3)}=0$.

Suppose  $p^{(1,2)}_{(1,3),(2,1)}\neq0$. Then there exists $w\in P_{(1,3),(2,1)}(x_0,x_1)$.  Since $3=\partial_{\Gamma}(x_1,y)\leq1+\partial_{\Gamma}(w,y)$, we get $\partial_{\Gamma}(w,y)=s$ for some $s\geq2$. Lemma \ref{1313} implies $p^{(1,2)}_{(1,3),(1,2)}=0$. Since $p^{(1,3)}_{(1,3),(1,2)}=0$ and $x_1\in P_{(1,3),(1,2)}(y,w)$, from Proposition \ref{arcpure} \ref{arcpure2}, we have $(y,w)\notin\Gamma_{1,s}$.
 %
Thus,  $(y,w)\in\Sigma_2$.

Since  $y,w,x_2\in C(x_0,x_1)$ and $w,x_2\in\Sigma_2(y)$, from Lemma \ref{sigma1} \ref{sigma12}, one gets $y\in Y_i(x_0,x_1)$, and $w,x_2\in Y_j(x_0,x_1)$ with $\{i,j\}=\{1,2\}$. Lemma \ref{sigma1} \ref{sigma11} implies $(w,x_2)\in\Sigma_1$.
Since $3=\partial_{\Gamma}(w,x_0)\leq\partial_{\Gamma}(w,x_2)+1$, from Proposition \ref{arcpure} \ref{arcpure2}, we obtain $(x_2,w)\in\Gamma_{1,t}$ for some $t\in\{2,3\}$. The fact $x_0\in P_{(1,2),(1,3)}(x_2,w)$ implies $p^{(1,2)}_{(1,2),(1,3)}\neq0$ or $p^{(1,3)}_{(1,2),(1,3)}\neq0$. Lemma \ref{1313} implies $p^{(1,3)}_{(1,2),(1,3)}\neq0$, a contradiction. 
Therefore, $p^{(1,2)}_{(1,3),(2,1)}=0$.
\end{proof}

\begin{lemma}\label{y1}
Let $T=\{3\}$. If $(x,y)\in\Gamma_{1,2}$, then $Y_1(x,y)\subseteq P_{(2,1),(2,1)}(x,y)$.
\end{lemma}
\begin{proof}
In view of Proposition \ref{e=2}, we have $\Sigma=J(m,2)$, which implies $|Y_1(x,y)|=1$ from Lemma \ref{sigma1}. By Lemma \ref{yi}, one gets $|P_{(2,1),(2,1)}(x,y)\cap Y_1(x,y)|>0$, and so $Y_1(x,y)\subseteq P_{(2,1),(2,1)}(x,y)$.
\end{proof}

\begin{lemma}\label{k23}
If $T=\{3\},$ then  $(2,3),(2,4)\notin\tilde{\partial}(\Gamma)$.
\end{lemma}
\begin{proof}

Assume the contrary, namely, $(2,s)\in\tilde{\partial}(\Gamma)$ for some $s\in\{3,4\}$. Since $T=\{3\},$ we have $p^{(2,s)}_{(1,2),(1,2)}\neq0$. Let $(x,z)\in\Gamma_{2,s}$ and $y\in P_{(1,2),(1,2)}(x,z)$.
Since $(x,z)\in\Sigma_2$, there exist distinct elements $a_1,a_2\in x$ and $b_1,b_2\in X\setminus x$ such that $z=x(\{a_1,a_2\},\{b_1,b_2\})$. Since $y\in C(x,z)$, without loss of generality, we may assume $y=x(a_1,b_1)$.
Let $y_1=x(a_2,b_1)$. Then $y_1\in Y_1(x,y)$. Lemma \ref{y1} implies $y_1\in P_{(2,1),(2,1)}(x,y)$.

Since $\partial_{\Gamma}(z,x)>2$ and $T=\{3\}$, one gets $P_{(2,1),(2,1)}(x,z)=\emptyset$. The fact $y_1\in C(x,z)$ implies $y_1\in P_{(2,1),(1,2)}(x,z)$. By Lemma \ref{fenlei}, there exist $w\in P_{(1,2),(1,2)}(x,z)$ and $w_1\in P_{(1,2),(2,1)}(x,z)$ such that $C(x,z)=\{y,y_1,w,w_1\}$. Then $w=x(a_i,b_2)$ and $w_1=x(a_j,b_2)$ for $\{i,j\}=\{1,2\}$. Since $w_1\in Y_1(x,w)$, from Lemma \ref{y1}, we have $w_1\in P_{(2,1),(2,1)}(x,w)$, contrary to the fact that $(x,w_1)\in\Gamma_{1,2}$.
\end{proof}

\begin{lemma}\label{k12}
Suppose $T=\{3\}.$ If $m>4$, then $p^{(2,2)}_{(1,2),(1,2)}=p^{(2,2)}_{(1,2),(2,1)}=1$, $k_{1,2}=m-2$,  $k_{2,2}=(m-2)(m-4)/2$ and $k_{3,3}=(m-2)/2$. 
%
\end{lemma}
\begin{proof}
In view of Proposition \ref{e=2}, we have $\Sigma=J(m,2)$. According to Lemma \ref{intersection numners} \ref{intersection numners 2} and Lemma \ref{inter}, we get $p^{(1,2)}_{(1,2),(1,2)}=p^{(1,2)}_{(1,2),(2,1)}=p^{(2,1)}_{(1,2),(2,1)}=m/4-1$ and $p^{(2,1)}_{(1,2),(1,2)}=p^{(1,2)}_{(2,1),(2,1)}=m/4+1.$ By    Lemma \ref{k33}, we have $p^{(3,3)}_{(1,2),(2,1)}=2$. 

According to Theorems \ref{johnson} and \ref{folded johnson}, one gets $k_{1,2}=b_0/2=e(m-e)/2=m-2$. By Lemma \ref{intersection numners} \ref{intersection numners 1} and  Lemma \ref{k23}, one has $(k_{1,2})^2=
p^{(1,2)}_{(1,2),(1,2)}k_{1,2}+p^{(2,1)}_{(1,2),(1,2)}k_{2,1}+p^{(2,2)}_{(1,2),(1,2)}k_{2,2}.$ Substituting $k_{1,2}$, $p^{(1,2)}_{(1,2),(1,2)}$ and $p^{(2,1)}_{(1,2),(1,2)}$ into the above equation,  one obtains
$p^{(2,2)}_{(1,2),(1,2)}k_{2,2}=(m-2)(m-4)/2$.

Since $m>4$, one has $(2,2)\in\tilde{\partial}(\Gamma)$, which implies $p^{(2,2)}_{(1,2),(1,2)}=p^{(2,2)}_{(2,1),(2,1)}\neq0$. Let $(x,z)\in\Gamma_{2,2}$ and $y\in P_{(1,2),(1,2)}(x,z)$. Since $(x,z)\in\Sigma_2$, there exist distinct elements $a_1,a_2\in x$ and $b_1,b_2\in X\setminus x$ such that $z=x(\{a_1,a_2\},\{b_1,b_2\})$. Since $y\in C(x,z)$, without loss of generality, we may assume $y=x(a_1,b_1)$.
Then \begin{align}\tag{6.1}\label{k12gongshi1}C(x,z)=\{y,y_1,w,w_1\}~\mbox{with}~y_1=x(a_2,b_1),~w=x(a_1,b_2)~\mbox{and}~w_1=x(a_2,b_2).\end{align} Since $y_1\in Y_1(x,y)$, from Lemma \ref{y1}, we get $y_1\in P_{(2,1),(2,1)}(x,y)$.


Suppose $p^{(2,2)}_{(1,2),(1,2)}>1$. Since $c_2=4$, from \eqref{a1c2} and  Lemma \ref{22} \ref{222}, one has $p^{(2,2)}_{(1,2),(1,2)}=p^{(2,2)}_{(2,1),(2,1)}=2$ and $p^{(2,2)}_{(1,2),(2,1)}=p^{(2,2)}_{(2,1),(1,2)}=0$. Since  $(x,y_1)\in\Gamma_{2,1}$, we get $y_1\in P_{(2,1),(2,1)}(x,z)$.  Since $z=y(a_2,b_2)$ and $w=y(b_1,b_2)$, we have $w\in Y_1(y,z)$, which implies $w\in P_{(2,1),(2,1)}(y,z)$ from Lemma \ref{y1}. Since $y_1=z(b_2,a_1)$ and $w_1=z(b_1,a_1)$,  we get $w_1\in Y_1(z,y_1)$, which implies  $w_1\in P_{(2,1),(2,1)}(z,y_1)$ from Lemma \ref{y1}. Since $(y,w_1)\in\Sigma_2$, from \eqref{k12gongshi1} and Lemma \ref{sigma2}, we get $C(y,w_1)=\{x,y_1,w,z\}$. The fact $z\in P_{(1,2),(2,1)}(y,w_1)$ implies $(y,w_1)\notin\Gamma_{2,2}$. Since $y_1\in P_{(1,2),(1,2)}(y,w_1)$, one gets $(y,w_1)\in\Gamma_{2,3}$, contrary to  Lemma \ref{k23}. Thus, $p^{(2,2)}_{(1,2),(1,2)}=1$.

Since $p^{(2,2)}_{(1,2),(1,2)}k_{2,2}=(m-2)(m-4)/2$,  we have $k_{2,2}=(m-2)(m-4)/2$. Since $p^{(2,2)}_{(1,2),(1,2)}=1$, from  Lemma \ref{22} \ref{225}, we obtain $p^{(2,2)}_{(1,2),(2,1)}=1$.    Note that $p^{(0,0)}_{(1,2),(2,1)}=k_{1,2}=m-2$. By Lemma \ref{intersection numners} \ref{intersection numners 1} and Lemma \ref{k23}, we get $$(k_{1,2})^2=k_{1,2}+p^{(1,2)}_{(1,2),(2,1)}k_{1,2}+p^{(2,1)}_{(1,2),(2,1)}k_{2,1}+p^{(2,2)}_{(1,2),(2,1)}k_{2,2}+
p^{(3,3)}_{(1,2),(2,1)}k_{3,3}.$$ Substituting $k_{1,2}$, $p^{(1,2)}_{(1,2),(2,1)}$, $p^{(2,1)}_{(1,2),(2,1)}$, $p^{(2,2)}_{(1,2),(2,1)}$, $p^{(3,3)}_{(1,2),(2,1)}$ and $k_{2,2}$ into the above equation, we have $k_{3,3}=(m-2)/2$.  This completes the proof of the lemma.
\end{proof}

%

Now we are ready to prove Theorem \ref{main}.

\begin{proof}[Proof of Theorem \ref{main}]
According to Proposition \ref{arcpure}, we have $T=\{3\}$, or $T=\{3,4\}$ or $\{2,3,4\}$. Next we divide the proof into two cases.

\textbf{Case 1.} $T=\{3,4\}$ or $\{2,3,4\}$.

In view of Proposition \ref{arcpure}, one gets $p^{(1,2)}_{(1,3),(1,3)}\neq0$. Let $(x,y)\in\Gamma_{1,2}$. We claim that $$z\in P_{(2,1),(2,1)}(x,y)\cup P_{(1,3),(1,3)}(x,y)~\mbox{ for all}~ z\in C(x,y).$$ Suppose for the contrary that $z\notin P_{(2,1),(2,1)}(x,y)\cup P_{(1,3),(1,3)}(x,y)$ for some $z\in C(x,y)$.  Lemma \ref{1313-2*} implies that $(1,2)$ is pure, and so $P_{(s,1),(t,1)}(x,y)=\emptyset$ for $(s,t)\neq(2,2)$. By \eqref{a1c2} and Lemma \ref{1313}, we get $z\in P_{(1,s),(t,1)}(x,y)$ or $z\in P_{(t,1),(1,s)}(x,y)$ for some $s,t\geq2$. Since $T=\{3,4\}$ or $\{2,3,4\}$, we may assume $z\in P_{(1,s),(t,1)}(x,y)$ for some $s,t\in\{2,3\}$. Lemma \ref{1331} implies  $s=2$. Then $p^{(1,2)}_{(1,2),(t,1)}\neq0$. According to Lemma \ref{intersection numners} \ref{intersection numners 2}, one gets $p^{(1,2)}_{(1,2),(1,t)}\neq0$, contrary to Lemma \ref{1313}. Thus, our claim holds.

By Lemma \ref{1313-2}, we obtain $p^{(1,2)}_{(2,1),(2,1)}=p^{(1,2)}_{(1,3),(1,3)}=1$. The claim implies $a_1=2$. It follows from Theorems \ref{johnson} and \ref{folded johnson} that $m=4$, and so $\Sigma=J(4,2)$.  By \cite{H}, $\Gamma$ is isomorphic to  ${\rm Cay}(\mathbb{Z}_6,\{1,2\})$.

\textbf{Case 2.} $T=\{3\}$.

In view of Proposition \ref{e=2}, one has $\Sigma=J(m,2)$.  If $m=4$, from \cite{H}, then $\Gamma$ is isomorphic to  ${\rm Cay}(\mathbb{Z}_6,\{1,4\})$. Then we only need to prove that the  $m=4$.
Suppose for the contrary that $m>4$. According to Lemma \ref{intersection numners} \ref{intersection numners 2} and Lemma \ref{inter}, we get $$p^{(1,2)}_{(1,2),(1,2)}=p^{(2,1)}_{(2,1),(2,1)}=p^{(1,2)}_{(1,2),(2,1)}=p^{(2,1)}_{(1,2),(2,1)}=
p^{(1,2)}_{(2,1),(1,2)}=p^{(1,2)}_{(2,1),(1,2)}=m/4-1$$ and $p^{(1,2)}_{(2,1),(2,1)}=p^{(2,1)}_{(1,2),(1,2)}=m/4+1.$ By    Lemma \ref{k33}, we have $p^{(3,3)}_{(1,2),(2,1)}=p^{(3,3)}_{(2,1),(1,2)}=2$. Since $c_2=4$, we get  $p^{(3,3)}_{(1,2),(1,2)}=p^{(3,3)}_{(2,1),(2,1)}=0$.

According to Lemma \ref{k12}, one obtains $k_{1,2}=m-2$, $k_{2,2}=(m-2)(m-4)/2$ and $k_{3,3}=(m-2)/2$. In view of Lemma \ref{intersection numners} \ref{intersection numners 2}, one gets $p^{\tilde{l}^*}_{\tilde{i}^*,(r,r)}=p^{\tilde{l}}_{\tilde{i},(r,r)}=k_{r,r}p^{(r,r)}_{\tilde{i}^*,\tilde{l}}/k_{\tilde{l}}$ for all $\tilde{l},\tilde{i},(r,r)\in\tilde{\partial}(\Gamma)$.
By  Lemma \ref{k12}  and Lemma \ref{22} \ref{225},  we have $p^{(2,2)}_{(1,2),(1,2)}=p^{(2,2)}_{(2,1),(2,1)}=p^{(2,2)}_{(1,2),(2,1)}=p^{(2,2)}_{(2,1),(1,2)}=1$.
Then $$p^{(1,2)}_{(1,2),(2,2)}=p^{(2,1)}_{(2,1),(2,2)}=p^{(2,1)}_{(1,2),(2,2)}=p^{(1,2)}_{(2,1),(2,2)}=m/2-2.$$ Since $p^{(3,3)}_{(1,2),(2,1)}=p^{(3,3)}_{(2,1),(1,2)}=2$ and $p^{(3,3)}_{(1,2),(1,2)}=p^{(3,3)}_{(2,1),(2,1)}=0$, one has \begin{align}p^{(1,2)}_{(1,2),(3,3)}=p^{(2,1)}_{(2,1),(3,3)}=p^{(1,2)}_{(3,3),(1,2)}=p^{(2,1)}_{(3,3),(2,1)}=1, \nonumber\\ p^{(2,1)}_{(1,2),(3,3)}=p^{(1,2)}_{(2,1),(3,3)}=p^{(2,1)}_{(3,3),(2,3)}=p^{(1,2)}_{(3,3),(2,1)}=0.\nonumber
\end{align}

In view of Lemma \ref{intersection numners} \ref{intersection numners 2}, one obtains $p^{(r,r)}_{\tilde{i},(s,s)}=p^{(r,r)}_{\tilde{i}^*,(s,s)}$ for all $\tilde{i},(r,r),(s,s)\in\tilde{\partial}(\Gamma)$. Then we may assume
\begin{align}
&p^{(2,2)}_{(1,2),(2,2)}=p^{(2,2)}_{(2,1),(2,2)}=a,~ p^{(3,3)}_{(1,2),(2,2)}=p^{(3,3)}_{(2,1),(2,2)}=b,\nonumber\\ &p^{(3,3)}_{(1,2),(3,3)}=p^{(3,3)}_{(2,1),(3,3)}=p^{(3,3)}_{(3,3),(1,2)}=p^{(3,3)}_{(3,3),(2,1)}=c,\nonumber\\  &p^{(2,2)}_{(1,2),(3,3)}=p^{(2,2)}_{(2,1),(3,3)}=p^{(2,2)}_{(3,3),(1,2)}=p^{(2,2)}_{(3,3),(2,1)}=d.\nonumber
\end{align} By Lemma \ref{intersection numners} \ref{intersection numners 2} again, we get
\begin{align}
&p^{(1,2)}_{(3,3),(2,2)}=p^{(2,1)}_{(3,3),(2,2)}=k_{2,2}p^{(2,2)}_{(1,2),(3,3)}/k_{1,2}=(m/2-2)d,\nonumber\\ &p^{(1,2)}_{(3,3),(3,3)}=p^{(2,1)}_{(3,3),(3,3)}=k_{3,3}p^{(3,3)}_{(1,2),(3,3)}/k_{1,2}=c/2.\nonumber
\end{align}
 Note that $p^{\tilde{l}}_{\tilde{i},(0,0)}=\delta_{\tilde{i},\tilde{l}}$ and $p^{(0,0)}_{\tilde{i},\tilde{l}^*}=k_{\tilde{i}}\delta_{\tilde{i},\tilde{l}}$, where $\delta$ is the Kronecker's delta.

Let $B_{12} = (p^{\tilde{j}}_{(1,2),\tilde{i}})_{\tilde{i},\tilde{j}}$, $B_{21} = (p^{\tilde{j}}_{(2,1),\tilde{i}})_{\tilde{i},\tilde{j}}$, and $B_{33} = (p^{\tilde{j}}_{(3,3),\tilde{i}})_{\tilde{i},\tilde{j}}$ be  matrices with $\tilde{i}, \tilde{j}\in\tilde{\partial}(\Gamma)$. Since $\Sigma=J(m,2)$, from Theorem \ref{johnson}, the diameter of $\Sigma$ is $2$. The fact $T=\{3\}$ implies $$\{(0,0),(1,2),(2,1)\}\cup\{\tilde{i},\tilde{i}^*\mid p_{(1,2),(1,2)}^{\tilde{i}}\neq0~{\rm or}~p_{(1,2),(2,1)}^{\tilde{i}}\neq0\}=\tilde{\partial}(\Gamma).$$
It follows from Lemma \ref{k23} that $\tilde{\partial}(\Gamma)=\{(0,0),(1,2),(2,1),(2,2),(3,3)\}$. Then
~\\
\begin{align*}
B_{12} & = \left(\begin{array}{ccccc}
0 & 1 & 0 & 0 & 0 \\
0 & m/4 - 1 & m/4 + 1 & 1 & 0 \\
m- 2 & m/4 - 1 & m/4 - 1 & 1 & 2 \\
0 & m/2 - 2 & m/2 - 2 & a & b \\
0 & 1 & 0 & d & c
\end{array}\right),\\
~\\
B_{21} & = \left(\begin{array}{ccccc}
0 & 0 & 1 & 0 & 0 \\
m - 2 & m/4 - 1 & m/4 - 1 & 1 & 2 \\
0 & m/4 + 1 & m/4 - 1 & 1 & 0 \\
0 & m/2- 2 & m/2 - 2 & a & b \\
0 & 0 & 1 & d & c
\end{array}\right), \\
~\\
B_{33} & = \left(\begin{array}{ccccc}
0 & 0 & 0 & 0 & 1 \\
0 & 1 & 0 & d & c \\
0 & 0 & 1 & d & c \\
0 & (m/2 - 2)d & (m/2 - 2)d & g & e \\
m/2 - 1 & c/2 & c/2 & h & f
\end{array}\right).
\end{align*}
Hence,
$$B_{12}B_{21} - B_{21}B_{12} = \left(\begin{array}{ccccc}
0 & 0 & 0 & 0 & 0 \\
0 & 0 & 0 & -2d + 2 & -2c \\
0 & 0 & 0 & 2d-2 & 2c \\
0 & -b + m - 4 & b - m + 4 & 0 & 0 \\
0 & -c & c & 0 & 0
\end{array}\right).$$
~\\

In view of Lemma \ref{intersection numners} \ref{intersection numners 4}, we get
\begin{align}
(B_{\tilde{l}}B_{\tilde{m}})_{\tilde{i},\tilde{j}}=\sum_{\tilde{r}\in\tilde{\partial}(\Gamma)}p^{\tilde{r}}_{\tilde{l},\tilde{i}}p^{\tilde{j}}_{\tilde{m},\tilde{r}}=
\sum_{\tilde{r}\in\tilde{\partial}(\Gamma)}p^{\tilde{r}}_{\tilde{i},\tilde{l}}p^{\tilde{j}}_{\tilde{m},\tilde{r}}=
\sum_{\tilde{t}\in\tilde{\partial}(\Gamma)}p^{\tilde{t}}_{\tilde{m},\tilde{i}}p^{\tilde{j}}_{\tilde{t},\tilde{l}}=
\sum_{\tilde{t}\in\tilde{\partial}(\Gamma)}p^{\tilde{t}}_{\tilde{m},\tilde{i}}p^{\tilde{j}}_{\tilde{l},\tilde{t}}=
(B_{\tilde{m}}B_{\tilde{l}})_{\tilde{i},\tilde{j}}\nonumber
 \end{align}
 for all $\tilde{m},\tilde{l},\tilde{i},\tilde{j}\in\tilde{\partial}(\Gamma).$
 Then
 $B_{\tilde{m}}B_{\tilde{l}} = B_{\tilde{l}}B_{\tilde{m}}$
 for all $\tilde{m},\tilde{l}\in\tilde{\partial}(\Gamma)$. It follows that
$c = 0$, $d = 1$ and $b = m-4$.
By substituting these values in $B_{12}$ and $B_{33}$, we have
\begin{align*}
(B_{12}B_{33} - B_{33}B_{12})_{(2,1),(3,3)} & = e+2f\\
(B_{12}B_{33} - B_{33}B_{12})_{(1,2),(3,3)} & = e-m+4.
\end{align*}
Since $B_{12}B_{33}=B_{33}B_{12}$ and both $e$ and $f$ are nonnegative integers, we have $e = f = 0$, which implies $m = 4$, a contradiction.

This completes the proof of the Theorem.
\end{proof}

\section*{Acknowledgements}

We are deeply grateful to Professor Hiroshi Suzuki for his valuable revision suggestions on this paper. Y.~Yang is supported by NSFC (12101575),  K. Wang is supported by the National Key R$\&$D Program of China (No.~2020YFA0712900) and NSFC (12071039, 12131011).

\appendix

\section{Proof of Lemma \ref{arc}}

To prove the lemma, we need some auxiliary results.

\begin{lemma}\label{arckey}
Let $x_1\in P_{(1,q), (1,r)}(x_0,x_2)$ with $r\geq 1$ and $\partial_{\Gamma}(x_2,x_0) = q-1$. Suppose $q \geq 4$ and $p^{(1,q-1)}_{(1,q),(1,h)} = 0$ for all $h\geq 1$. Then $(x_0,x_2)\in\Sigma_2$. Moreover, if $y_1\in C(x_0,x_2)\cap\Sigma_2(x_1)$, then one of the following holds:
\begin{enumerate}
\item\label{arckey1} $r = q$ and $(x_0, y_1, x_2)$ is a path in $\Gamma$;
\item\label{arckey2} $r\neq q$ and $y_1\in P_{(1,r),(1,q)}(x_0,x_2)$.
\end{enumerate}
\end{lemma}
\begin{proof}
 Since $p^{(1,q-1)}_{(1,q),(1,h)}=0$ for all $h\geq1$, we have $(x_0,x_2)\in\Gamma_{2,q-1}$.  The fact $q\geq4$ implies $\partial_{\Gamma}(x_2,x_0)>2$, and so $(x_0,x_2)\in\Sigma_2$. Hence, the first statement holds.

If $r=q$,   from  Lemma \ref{fenlei}, then \ref{arckey1} holds. Now we consider the case $r\neq q$. Suppose for the contrary that $y_1\notin P_{(1,r),(1,q)}(x_0,x_2)$.  We derive a contradiction in several steps.

\begin{step}\label{key1}
{\rm Show that there exists $z_1\in P_{(1,r),(1,q)}(x_0,x_2)$ such that $(x_1,z_1)\in \Gamma_{q',1}$ for some $q'\geq 2$.}
\end{step}

Since  $y_1\notin P_{(1,r),(1,q)}(x_0,x_2)$ with $q\neq r$, there exists $z_1\in P_{(1,r),(1,q)}(x_0,x_2)$.  Since $x_1,y_1,z_1\in C(x_0,x_2)$ and $(x_1,y_1)\in\Sigma_2$, from Lemma \ref{sigma2}, we get $z_1\in\Sigma_1(x_1)$.

Assume $(x_1,z_1)\notin\Gamma_{q',1}$ for all $q'\geq2$. Then $(x_1,z_1)\in\Gamma_{1,s}$ for some $s\geq1$. Since $\partial_{\Gamma}(x_2,x_0)=q-1$ and $r\neq q$, one gets $r<q$. Since $(x_1,z_1)\in\Gamma_{1,s}$, we have $q=\partial_{\Gamma}(x_1,x_0)\leq 1+\partial_{\Gamma}(z_1,x_0)=1+r$, which implies $r=q-1$. The fact $x_1\in P_{(1,q),(1,s)}(x_0,z_1)$ with $(x_0,z_1)\in\Gamma_{1,r}$ implies $p^{(1,q-1)}_{(1,q),(1,s)}\neq0$, contrary to the fact that $p^{(1,q-1)}_{(1,q),(1,h)}=0$ for all $h$. Thus, $(x_1,z_1)\in \Gamma_{q',1}$ for some $q'\geq 2$.

\begin{step}\label{key2}
{\rm Show that $y_1\in P_{(s,1),(1,t)}(x_0, x_2)$ for some $s,t\geq 2$.}
\end{step}

Assume the contrary, namely, $y_1\notin P_{(s,1),(1,t)}(x_0, x_2)$ for all $s,t\geq 2$.  By Step  \ref{key1}, there exists   $z_1\in P_{(1,r),(1,q)}(x_0,x_2)$ such that  $(x_1,z_1)\in\Gamma_{q',1}$ with $q'\geq2$. Since $\partial_{\Gamma}(x_2,x_0)>2$, we get $y_1\notin P_{(s,1),(t,1)}(x_0,x_2)$ for all $s,t\geq1$.
Then $y_1\in P_{(1,s),(t,1)}(x_0,x_2)$ for some $s,t\geq2$, or $y_1\in P_{(1,s),(1,t)}(x_0,x_2)$ for some $(s,t)\notin\{(r,q),(1,1)\}$.

\textbf{Case 1.} $y_1\in P_{(1,s),(t,1)}(x_0,x_2)$ for some $s,t\geq2$.


Since $(x_0,x_2)\in\Sigma_2$ and $x_1,y_1,z_1\in C(x_0,x_2)$ with $(x_1,y_1)\in\Sigma_2$, from Lemma \ref{sigma2},   we obtain  $x_0,x_2,z_1\in C(x_1,y_1)$. The fact that $4\leq q=\partial_{\Gamma}(x_2,z_1)\leq1+\partial_{\Gamma}(y_1,z_1)$ 
implies  $(z_1,y_1)\in\Gamma_{1,l}$ for $l\geq3$.  Since $x_0\in P_{(q,1),(1,s)}(x_1,y_1)$ and $z_1\in P_{(q',1),(1,l)}(x_1,y_1)$, one has $x_2\in P_{(1,s),(q,1)}(x_1,y_1)$ or $x_2\in P_{(1,l),(q',1)}(x_1,y_1)$, contrary to the fact that $(x_2,y_1)\in\Gamma_{1,t}$ with $t\geq2.$

\textbf{Case 2.} $y_1\in P_{(1,s),(1,t)}(x_0,x_2)$ for some $(s,t)\notin\{(r,q),(1,1)\}$.

Since $(x_0,x_2)\in\Gamma_{2,q-1}$ with $q\geq4$, we get $(x_0,x_2)\notin\Gamma_{2,2}.$ If $s=t$, then Lemma \ref{fenlei} \ref{fenlei5} holds, which implies  $(x_1,z_1)\in\Sigma_2$, contrary to the fact that $(x_1,z_1)\in\Gamma_{q',1}$. Thus, $s\neq t$. Note that $P_{(1,s),(1,t)}(x_0,x_2)=P_{(1,t),(1,s)}(x_0,x_2)$. Since $y_1\in P_{(1,s),(1,t)}(x_0,x_2)$ for  $(s,t)\neq(r,q)$,   there exists $z_2\in P_{(1,t),(1,s)}(x_0,x_2)$ such that $z_2\notin\{x_1,y_1,z_1\}$. The fact $c_2=4$ implies $C(x_0,x_2)=\{x_1,y_1,z_1,z_2\}$.

Since $(x_1,y_1)\in\Sigma_2$, from Lemma \ref{sigma2}, one has $ C(x_1,y_1)=\{x_0,x_2,z_1,z_2\}$. Note that $x_0\in P_{(q,1),(1,s)}(x_1,y_1)$ and $x_2\in P_{(1,r),(t,1)}(x_1,y_1)$ with $(s,t)\neq(r,q)$. Since $(x_1,z_1)\in\Gamma_{q',1}$ for $q'\geq2$,  we get
$z_2\in P_{(1,s),(q,1)}(x_1,y_1)~\mbox{and}~z_1\in P_{(t,1),(1,r)}(x_1,y_1).$
Then $x_2\in P_{(1,q),(s,1)}(z_1,z_2)$,
\begin{align}\tag{A.1}
x_0\in P_{(r,1),(1,t)}(z_1,z_2), ~x_1\in P_{(1,t),(1,s)}(z_1,z_2)
~\mbox{and} ~y_1\in P_{(1,r),(1,q)}(z_1,z_2).\label{4.5gongshi1}
\end{align}
Since $(x_1,y_1)\in\Sigma_2$, from Lemma \ref{sigma2}, we obtain $(z_1,z_2)\in\Sigma_2$ and $ C(z_1,z_2)=\{x_0,x_1,x_2,y_1\}$. Since $x_2\in P_{(1,q),(s,1)}(z_1,z_2)$ and $q\geq4$,  there exists $u\in\{x_0,x_1,y_1\}$ such that $u\in P_{(s,1),(1,q)}(z_1,z_2)$. By \eqref{4.5gongshi1}, we have $(s,t)=(r,q)$, $t=s=q=1$, or $s=r=1$, which implies $s=r$, and so $t\neq q$.

Since $y_1\in P_{(1,s),(1,t)}(x_0,x_2)$ and  $(x_0,x_2)\in\Gamma_{2,q-1}$, we have $t<q$. Since $x_2\in P_{(1,q),(s,1)}(z_1,z_2)$, from \eqref{4.5gongshi1}, one has $q=\partial_{\Gamma}(z_2,y_1)\leq\partial_{\Gamma}(z_2,x_2)+\partial_{\Gamma}(x_2,y_1)=1+t$, 
which implies $t=q-1$. The fact $z_2\in P_{(1,q),(1,s)}(y_1,x_2)$ implies $p^{(1,q-1)}_{(1,q),(1,s)}\neq0$, a contradiction.

\begin{step}\label{key3}
{\rm Show that $r=s=2$ and there exists $z_2\in P_{(1,t), (2,1)}(x_0, x_2)$ such that $C(x_0,x_2)=\{x_1,y_1,z_1,z_2\}$,
$y_1 \in P_{(1,2),(1,q)}(z_1,z_2) \text{ and } x_1\in P_{(1,t),(2,1)}(z_1,z_2).$}
\end{step}

By Step \ref{key1}, there exists $z_1\in P_{(1,r),(1,q)}(x_0,x_2)$ such that $(x_1,z_1)\in \Gamma_{q',1}$ for some $q'\geq 2$.  Step \ref{key2} implies $y_1\in P_{(s,1),(1,t)}(x_0,x_2)$ for some $s,t\geq2$. Since $c_2=4$ and  $x_1\in P_{(1,q),(1,r)}(x_0,x_2)$, there exists $z_2\in P_{(1,t),(s,1)}(x_0,x_2)$ such that $C(x_0,x_2)=\{x_1,y_1,z_1,z_2\}$. 

Since $(x_1,y_1)\in\Sigma_2$, from Lemma \ref{sigma2}, one obtains $(z_1,z_2)\in\Sigma_2$, $C(x_1,y_1)=\{x_0,x_2,z_1,z_2\}$ and $C(z_1,z_2)=\{x_0,x_1,x_2,y_1\}$, depicted in Figure 1 (a). (The black arc, orange arc, red arc, and  green arc represent the arcs of type $(1,q)$,  $(1,r)$,  $(1,t)$, and $(1,s)$, respectively.)

 \begin{center}
\begin{tikzpicture}[scale = 0.65]
\begin{scope}[every node/.style={circle,draw,scale = 0.75}]
\node (x1) at (0,3) {$x_1$};
\node (x0) at (3,6) {$x_0$};
\node (z1) at (2,2) {$z_1$};
\node (x2) at (3,0) {$x_2$};
\node (z2) at (4,4) {$z_2$};
\node (y1) at (6,3) {$y_1$};
\end{scope}
\begin{scope}[every node/.style={font=\tiny,scale=0.9},>= latex, thin]
\draw [->] (x0) to node[above, rotate=45] {$(1,q)$} (x1);
\draw [densely dashed]  (z1) -- (x1);
\draw [->,draw=orange] (x1) to node[below, rotate=-45] {$(1,r)$} (x2);
\draw [densely dashed] (z2) -- (x1);
\draw [->,draw=green] (y1) to node[above, rotate=-45] {$(1,s)$} (x0);
\draw [densely dashed] (z1) -- (y1);
\draw [->,draw=red] (y1) to node[below, rotate=45] {$(1,t)$} (x2);
\draw [densely dashed] (y1) -- (z2);
\draw [->,draw=orange] (x0) to node[above, rotate=70] {$(1,r)$} (z1);
\draw [->] (z1) to node[above, rotate=-70] {$(1,q)$} (x2);
\draw [->,draw=green] (x2) to node[below, rotate=75] {$(1,s)$}(z2);
\draw [->,draw=red] (x0) to node[below, rotate=-60] {$(1,t)$} (z2);
\end{scope}
\end{tikzpicture}
\quad
\begin{tikzpicture}[scale = 0.65]
\begin{scope}[every node/.style={circle,draw,scale = 0.75}]
\node (x1) at (0,3) {$x_1$};
\node (x0) at (3,6) {$x_0$};
\node (z1) at (2,2) {$z_1$};
\node (x2) at (3,0) {$x_2$};
\node (z2) at (4,4) {$z_2$};
\node (y1) at (6,3) {$y_1$};
\end{scope}
\begin{scope}[every node/.style={font=\tiny,scale=0.9},>= latex, thin]
\draw [->] (x0) to node[above, rotate=45] {$(1,q)$} (x1);
\draw [->,draw=red] (z1) to  (x1);
\draw [->,draw=orange] (x1) to node[below, rotate=-45] {$(1,r)$} (x2);
\draw [->,draw=orange] (z2) to  (x1);
\draw [->,draw=green] (y1) to node[above, rotate=-45] {$(1,s)$} (x0);
\draw [->,draw=green] (z1) to  (y1);
\draw [->,draw=red] (y1) to node[below, rotate=45] {$(1,t)$} (x2);
\draw [->] (y1) to  (z2);
\draw [->,draw=orange] (x0) to node[above, rotate=70] {$(1,r)$} (z1);
\draw [->] (z1) to node[above, rotate=-70] {$(1,q)$} (x2);
\draw [->,draw=green] (x2) to node[below, rotate=75] {$(1,s)$}(z2);
\draw [->,draw=red] (x0) to node[below, rotate=-60] {$(1,t)$} (z2);
\end{scope}
\end{tikzpicture}
\quad
\begin{tikzpicture}[scale = 0.65]
\begin{scope}[every node/.style={circle,draw,scale = 0.75}]
\node (x1) at (0,3) {$x_1$};
\node (x0) at (3,6) {$x_0$};
\node (z1) at (2,2) {$z_1$};
\node (x2) at (3,0) {$x_2$};
\node (z2) at (4,4) {$z_2$};
\node (y1) at (6,3) {$y_1$};
\end{scope}
\begin{scope}[every node/.style={font=\tiny,scale=0.9},>= latex, thin]
\draw [->] (x0) to node[above, rotate=45] {$(1,q)$} (x1);
\draw [->,draw=red] (z1) to (x1);
\draw [->,draw=blue] (x1) to node[below, rotate=-45] {$(1,2)$} (x2);
\draw [->,draw=blue] (z2) to  (x1);
\draw [->,draw=blue] (y1) to node[above, rotate=-45] {$(1,2)$} (x0);
\draw [->,draw=blue] (z1) to  (y1);
\draw [->,draw=red] (y1) to node[below, rotate=45] {$(1,t)$} (x2);
\draw [->] (y1) to  (z2);
\draw [->,draw=blue] (x0) to node[above, rotate=70] {$(1,2)$} (z1);
\draw [->] (z1) to node[above, rotate=-70] {$(1,q)$} (x2);
\draw [->,draw=blue] (x2) to node[below, rotate=75] {$(1,2)$}(z2);
\draw [->,draw=red] (x0) to node[below, rotate=-60] {$(1,t)$} (z2);
\end{scope}
\end{tikzpicture}

(a) \quad\quad\qquad\qquad\qquad\qquad (b) \quad\quad\qquad\qquad\qquad\qquad (c)

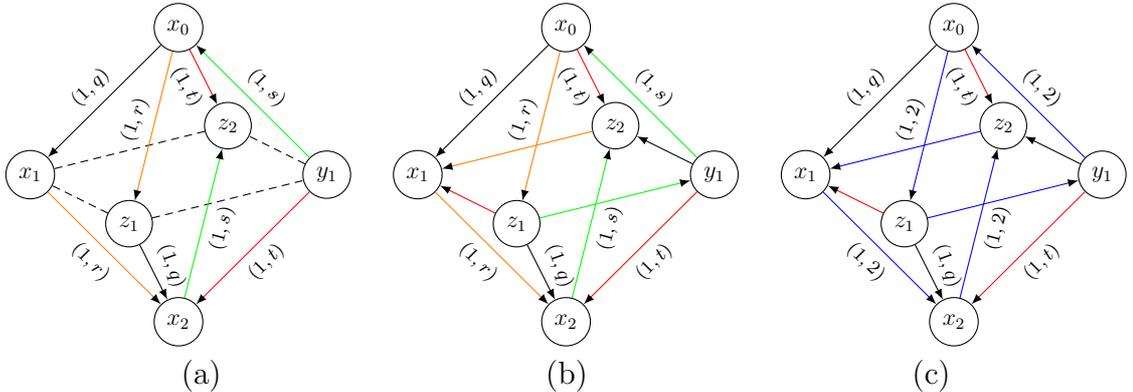
Figure 1: The induced subdigraph of $\Gamma$ on $C(x_0,x_2)\cup\{x_0,x_2\}$.
\end{center}

Note that 
$x_2\in P_{(1,q),(1,s)}(z_1,z_2)~\mbox{and}~x_0\in P_{(r,1),(1,t)}(z_1,z_2).$
If $s=q\geq4$,  by Corollary \ref{fenlei*}, then $(x_0,x_2)\notin\Sigma_2$, a contradiction. Thus, $s\neq q$. Then
\begin{align} x_1\in P_{(1,s),(1,q)}(z_1,z_2)~&\mbox{and}~  y_1\in P_{(1,t),(r,1)}(z_1,z_2),~\mbox{or}\tag{A.2}\label{s11tgongshi2}\\
y_1\in P_{(1,s),(1,q)}(z_1,z_2)~&\mbox{and}~ x_1\in P_{(1,t),(r,1)}(z_1,z_2).\tag{A.3}\label{s11tgongshi1}
 \end{align}
 Suppose that \eqref{s11tgongshi2} holds. Since $(z_1,y_1,x_0)$ is a circuit in $\Gamma$ with $(z_1,y_1)\in\Gamma_{1,t}$ and $(y_1,x_0)\in\Gamma_{1,s}$, we have $s=t=2$, which implies $q=\partial_{\Gamma}(x_1,x_0)\leq1+\partial_{\Gamma}(z_2,x_0)=1+t=3$, a contradiction. Thus, \eqref{s11tgongshi1} holds, depicted in Figure 1 (b).

 Note that $ C(x_1,y_1)=\{x_0,x_2,z_1,z_2\}$ and $x_0\in P_{(1,s),(1,q)}(y_1,x_1)$, \begin{align}\label{4.6gongshi1}x_2\in P_{(1,t),(r,1)}(y_1,x_1),~z_1\in P_{(s,1),(1,t)}(y_1,x_1),~ z_2\in P_{(1,q),(1,r)}(y_1,x_1).\tag{A.4}\end{align}
Since  $s\neq q$, from \eqref{4.6gongshi1}, there exists $v\in\{x_2,z_1,z_2\}$ such that $v\in P_{(1,q),(1,s)}(y_1,x_1)$, which implies $s=r=1$, $q=s=1$ or $r=s$. The fact $s\geq2$ implies $r=s$.  Since $(x_1,x_2,z_2)$ is a circuit in $\Gamma$ with $(x_1,x_2)\in\Gamma_{1,r}$, one has $r=s=2$, depicted in Figure 1 (c). (The blue arc represents the arc of type $(1,2)$.) By \eqref{s11tgongshi1}, this completes the proof of this step.

%

\begin{step}\label{key4}
{\rm We reach a contradiction.}
\end{step}

By Steps \ref{key1}, \ref{key2} and \ref{key3}, we have $C(x_0,x_2)=\{x_1,y_1,z_1,z_2\}$, and
\begin{align}
x_1\in P_{(1,q),(1,2)}(x_0,x_2),~&y_1\in P_{(2,1),(1,t)}(x_0,x_2),~z_1\in P_{(1,2),(1,q)}(x_0,x_2),\tag{A.5}\label{key41}\\
z_2\in P_{(1,t),(2,1)}(x_0,x_2),~
&y_1\in P_{(1,2),(1,q)}(z_1,z_2),~x_1\in P_{(1,t),(2,1)}(z_1,z_2).\tag{A.6}\label{key42}
\end{align}

Since  $x_1\in P_{(1,q),(t,1)}(x_0,z_1)$ and $y_1\in P_{(2,1),(2,1)}(x_0,z_1)$, there exists a vertex $u\in P_{(t,1),(1,q)}(x_0,z_1)$ such that $u\notin\{x_1,y_1\}$. Note that $x_0\in{X\choose e}$. Since $(x_0,x_2)\in\Sigma_2$ and $x_1\in  C(x_0,x_2)$,  there exist $a_1,a_2\in x_0$ and  $b_1,b_2\in X\setminus x_0$ such that $$x_1=x_0(a_1,b_1)~\mbox{and}~x_2=x_0(\{a_1,a_2\},\{b_1,b_2\}).$$ The fact $y_1\in  C(x_0,x_2)$ with $(x_1,y_1)\in\Sigma_2$   implies $y_1=x_0(a_2,b_2)$. Since $ C(x_0,x_2)=\{x_1,y_1,z_1,z_2\}$, we may set $$z_1=x_0(a_2,b_1)~\mbox{and}~z_2=x_0(a_1,b_2).$$ If $x_0(\{a,a_1\},\{b,b_2\})\in P_{(t,1)(1,q)}(x_2,z_2)$ for  $x_0(a,b)\in P_{(t,1),(1,q)}(x_0,z_1)$, from \eqref{key41} and \eqref{key42}. then
\begin{align}
P_{(t,1),(1,q)}(x_2,z_2)\supseteq \{x_0(\{a,a_1\},\{b,b_2\})\mid x_0(a,b)\in P_{(t,1),(1,q)}(x_0,z_1)\}\cup\{y_1\},\nonumber
\end{align}
a contradiction. It follows that there exist $a\in x_0$ and $b\in X\setminus x_0$ such that \begin{align}u:&=x_0(a,b)\in P_{(t,1),(1,q)}(x_0,z_1),\tag{A.7}\label{qneqr1gongshi81}\\ u_1:&=x_0(\{a,a_1\},\{b,b_2\})\notin P_{(t,1)(1,q)}(x_2,z_2).\tag{A.8}\label{qneqr1gongshi82}\end{align} Since $z_1=x_0(a_2,b_1)$ and $x_1,y_1\notin P_{(t,1),(1,q)}(x_0,z_1)$, from \eqref{qneqr1gongshi81},  we have $a\in x_0\setminus\{a_1,a_2\}$ and $b=b_1$, or $a=a_2$ and $b\in X\setminus(x_0\cup\{b_1,b_2\})$.

\textbf{Case 1.} $a\in x_0\setminus\{a_1,a_2\}$ and $b=b_1$.

By \eqref{qneqr1gongshi81}, one has $(u,z_1)\in\Gamma_{1,q}$ and $(u,x_1)\in\Sigma_1$ since $x_1=x_0(a_1,b_1)$. In view of \eqref{key42}, we get $q=\partial_{\Gamma}(z_1,u)\leq\partial_{\Gamma}(z_1,x_1)+\partial_{\Gamma}(x_1,u)=1+\partial_{\Gamma}(x_1,u)$, which implies $(u,x_1)\in\Gamma_{1,l}$ with $l\geq q-1\geq3$. 

Since $x_2=x_0(\{a_1,a_2\},\{b_1,b_2\})$, from \eqref{qneqr1gongshi81}, we get $(u,x_2)\in\Sigma_2$. It follows from \eqref{qneqr1gongshi82} that $ C(u,x_2)=\{x_1,z_1,u_1,u_2\}$ with $u_2=x_0(\{a_2,a\},\{b_1,b_2\})$. Since $z_1=x_0(a_2,b_1)$, from \eqref{qneqr1gongshi82}, we get $(z_1,u_1)\in\Sigma_2$. By \eqref{key41} and  \eqref{qneqr1gongshi81}, one gets $x_1\in P_{(1,l),(1,2)}(u,x_2)$ with $l\geq3$ and $z_1\in P_{(1,q),(1,q)}(u,x_2)$ with $q\geq4$. It follows from Lemma \ref{fenlei} \ref{fenlei5} that \begin{align}\label{qneqr1gongshi9}u_2\in P_{(1,2),(1,l)}(u,x_2)~ \mbox{and} ~u_1\in P_{(1,i),(1,i)}(u,x_2)~ \mbox{for some}~ i\geq2.\tag{A.9}\end{align}

Since $z_2=x_0(a_1,b_2)$, from \eqref{qneqr1gongshi81}, one has  $(u,z_2)\in\Sigma_2$. It follows from \eqref{qneqr1gongshi82} that $ C(u,z_2)=\{x_0,x_1,u',u_1\}$ with $u'=x_0(a,b_2)$.  \eqref{key42} implies $x_1\in P_{(1,l),(2,1)}(u,z_2)$ with $l\geq3$. It follows from Lemma \ref{22} that $(u,z_2)\notin\Gamma_{2,2}$. By \eqref{key42} and \eqref{qneqr1gongshi81}, we get $x_0\in P_{(1,t),(1,t)}(u,z_2)$. \eqref{qneqr1gongshi82} implies $(x_0,u_1)\in\Sigma_2$. From \eqref{qneqr1gongshi9}, one has $(u,u_1)\in\Gamma_{1,i}$, which implies \begin{align}u'\in P_{(2,1),(1,l)}(u,z_2)~\mbox{and}~u_1\in P_{(1,i),(1,i)}(u,z_2)\tag{A.10}\label{qneqr1gongshi10}\end{align}  from Lemma \ref{fenlei} \ref{fenlei6}.
Then $u_1\in P_{(i,1),(1,i)}(x_2,z_2)$. \eqref{qneqr1gongshi82} implies $i\neq t$ or $i\neq q$.

Since $y_1=x_0(a_2,b_2),$ from \eqref{qneqr1gongshi82}, we get $(y_1,u_1)\in\Sigma_2$ and $ C(y_1,u_1)=\{z_2,x_2,u',u_2\}$. \eqref{key41}, \eqref{key42}, \eqref{qneqr1gongshi9} and \eqref{qneqr1gongshi10} imply $x_2\in P_{(1,i),(t,1)}(u_1,y_1)$ and $z_2\in P_{(1,i),(q,1)}(u_1,y_1)$. Then \begin{align}\tag{A.11}\label{qneqr1gongshi11}(u_2,y_1),(u',y_1)\in\Gamma_{1,i}\end{align}  

Since $y_1=x_0(a_2,b_2)$, from \eqref{qneqr1gongshi81}, we have $(u,y_1)\in\Sigma_2$ and $ C(u,y_1)=\{z_1,x_0,u',u_2\}$. By   \eqref{key41}, \eqref{key42} and \eqref{qneqr1gongshi81}, we get $z_1\in P_{(1,q),(1,2)}(u,y_1)$ and $x_0\in P_{(1,t),(2,1)}(u,y_1)$. Then \begin{align}u'\in P_{(1,2),(1,q)}(u,y_1)&~\mbox{and}~u_2\in P_{(2,1),(1,t)}(u,y_1),~\mbox{or}\nonumber\\u_2\in P_{(1,2),(1,q)}(u,y_1)&~\mbox{and}~u'\in P_{(2,1),(1,t)}(u,y_1).\nonumber\end{align} It follows from \eqref{qneqr1gongshi11} that $q=t=i$, contrary to the fact that $i\neq t$ or $i\neq q$. 

\textbf{Case 2.} $a=a_2$ and $b\in X\setminus(x_0\cup\{b_1,b_2\})$.

By  \eqref{qneqr1gongshi81}, one gets $(u,z_1)\in\Gamma_{1,q}$ and $(u,y_1)\in\Sigma_1$ since $y_1=x_0(a_2,b_2)$.  
In view of \eqref{key42}, we have $q=\partial_{\Gamma}(z_1,u)\leq\partial_{\Gamma}(z_1,y_1)+\partial_{\Gamma}(y_1,u)=1+\partial_{\Gamma}(y_1,u)$, which implies $(u,y_1)\in\Gamma_{1,l}$ for $l\geq q-1\geq3$.



Since $x_2=x_0(\{a_1,a_2\},\{b_1,b_2\})$, from \eqref{qneqr1gongshi81}, we get $(u,x_2)\in\Sigma_2$. Since $z_1=x_0(a_2,b_1)$,  from \eqref{qneqr1gongshi82}, one has  $z_1,u_1\in C(u,x_2)$ and $(z_1,u_1)\in\Sigma_2$. By \eqref{key41} and \eqref{qneqr1gongshi81}, we have $z_1\in P_{(1,q),(1,q)}(u,x_2)$ with $q\geq4$. Corollary \ref{fenlei*} implies \begin{align}\tag{A.12}\label{qneqr1gongshi12}u_1\in P_{(1,i),(1,i)}(u,x_2) ~\mbox{with}~ i\geq2.\end{align}

Since $y_1=x_0(a_2,b_2)$,  from Lemma \ref{sigma2} and \eqref{qneqr1gongshi82}, we get $u,x_2,y_1\in C(z_1,u_1)$.  \eqref{key41}, \eqref{qneqr1gongshi81} and \eqref{qneqr1gongshi12} imply \begin{align}u\in P_{(q,1),(1,i)}(z_1,u_1)~\mbox{and}~ x_2\in P_{(1,q),(i,1)}(z_1,u_1).\nonumber\end{align} If $i\neq q$, 
then $y_1\in P_{(1,i),(q,1)}(z_1,u_1)$ since $(z_1,y_1)\in\Gamma_{1,2}$, which implies $i=2$, and so $q=\partial_{\Gamma}(y_1,u_1)\leq1+\partial_{\Gamma}(x_2,u_1)=1+i=3$,  a contradiction. Hence, $i=q$.

Since $z_2=x_0(a_1,b_2)$, from \eqref{qneqr1gongshi81}, we have $(u,z_2)\in\Sigma_2$.  Since $y_1=x_0(a_2,b_2)$, from  \eqref{qneqr1gongshi82}, one gets $y_1,x_0,u_1\in C(u,z_2)$  and   $(x_0,u_1)\in\Sigma_2$. By  \eqref{key42} and \eqref{qneqr1gongshi81}, we get $y_1\in P_{(1,l),(1,q)}(u,z_2)$ with $q\geq4$ and $x_0\in P_{(1,t),(1,t)}(u,z_2)$.  Lemma \ref{fenlei} \ref{fenlei2} and \ref{fenlei5} imply that $(u,u_1,z_2)$ is a pure path. By \eqref{qneqr1gongshi12}, one has  $u_1\in P_{(1,i),(1,i)}(u,z_2)$, and so $u_1\in P_{(i,1),(1,i)}(x_2,z_2)$. Since $q=i$, from \eqref{qneqr1gongshi82}, one obtains  $t\neq i$.

Since $y_1,x_0,u_1\in C(u,z_2)$  and   $(x_0,u_1)\in\Sigma_2$, from Lemma \ref{sigma2}, we have $z_2,y_1,u\in C(x_0,u_1)$. By \eqref{key42}, \eqref{qneqr1gongshi81} and \eqref{qneqr1gongshi12}, one obtains $z_2\in P_{(1,t),(i,1)}(x_0,u_1)$ and $u\in P_{(t,1),(1,i)}(x_0,u_1)$. The fact $t\neq i$ implies $y_1\in P_{(i,1),(1,t)}(x_0,u_1)$ or $y_1\in P_{(1,i),(t,1)}(x_0,u_1)$, contrary to the fact that $(x_0,y_1)\in\Gamma_{2,1}$ and $i=q\geq3$.
\end{proof}

\begin{cor}\label{arckey*}
Let $x_1\in P_{(1,r_1), (1,r_2)}(x_0,x_2)$  with $r_1,r_2\geq 1$, $q\in \{r_1, r_2\}$ and $q-1=\partial_{\Gamma}(x_2,x_0)$. Suppose $q \geq 4$ and $p^{(1,q-1)}_{(1,q),(1,h)} = 0$ for all $h\geq 1$. Then $(x_0,x_2)\in\Sigma_2$. Moreover, if $y_1\in C(x_0,x_2)\cap\Sigma_2(x_1)$, then $(x_0, y_1, x_2)$ is a path in $\Gamma$ and the following hold:
\begin{enumerate}
\item\label{arckey*1} if $r_1\neq q$, then $y_1\in P_{(1,q),(1,r_1)}(x_0,x_2)$;
\item\label{arckey*2} if $r_2\neq q$, then $y_1\in P_{(1,r_2),(1,q)}(x_0,x_2)$.
\end{enumerate}
\end{cor}
\begin{proof}
Since $p^{(1,q-1)}_{(1,q),(1,h)}=p^{(1,q-1)}_{(1,h),(1,q)}=0$ for all $h\geq1$, we get $(x_0,x_2)\in\Gamma_{2,q-1}$.  The fact $q\geq4$ implies $\partial_{\Gamma}(x_2,x_0)>2$, and so $(x_0,x_2)\in\Sigma_2$. Hence, the first statement holds.

If $r_1=r_2=q$, from Lemma \ref{arckey} \ref{arckey1}, then $(x_0, y_1, x_2)$ is a path in $\Gamma$.  If $r_2\neq q$, then $r_1=q$, which implies $y_1\in P_{(1,r_2),(1,q)}(x_0,x_2)$ from Lemma \ref{arckey} \ref{arckey2}. Thus, \ref{arckey*2} holds.

To prove the second statement, it suffices to show that \ref{arckey*1} holds. Note that $r_2=q$.  Suppose for the contrary that $y_1\notin P_{(1,q),(1,r_1)}(x_0,x_2)$. Since $x_1\in P_{(1,r_1),(1,q)}(x_0,x_2)$, there exists $z\in P_{(1,q),(1,r_1)}(x_0,x_2)$ with $z\neq y_1$. Since $(x_1,y_1)\in\Sigma_2$, from Lemma \ref{sigma2}, one gets $x_1,y_1\in\Sigma_1(z)$. By Lemma \ref{sigma2} and Lemma \ref{arckey}   \ref{arckey2}, there exists $w\in C(x_0,x_2)\cap\Sigma_2(u)$ such that $w\in P_{(1,r_1),(1,q)}(x_0,x_2)$. The fact $c_2=4$ implies $C(x_0,x_2)=\{x_1,y_1,z,w\}$. Since $w,x_1\in P_{(1,r_1),(1,q)}(x_0,x_2)$, we get $z,y_1\in P_{(1,q),(1,r_1)}(x_0,x_2)$, a contradiction. Thus, $y_1\in P_{(1,q),(1,r_1)}(x_0,x_2)$, and so \ref{arckey*1} holds.
\end{proof}

\begin{lemma}\label{circuit}
Let $q\geq3$. Suppose $p^{(1,q-1)}_{(1,q),(1,h)}=0$ for all $h\geq1$. Then exactly one of the following holds:
\begin{enumerate}
\item\label{circuit1} $p^{(1,q)}_{(2,q-1),(q,1)}=0$ and each circuit of length $q+1$ contains at most one arc of type $(1,q)$;

\item\label{circuit2} $p^{(1,q)}_{(2,q-1),(q,1)}\neq0$ and there exists a circuit of length $q+1$ consisting of arcs of type $(1,q)$.
\end{enumerate}
\end{lemma}
\begin{proof}


Suppose $p^{(1,q)}_{(2,q-1),(q,1)}=0$. Let  $(x_0,x_1)\in\Gamma_{1,q}$. If there exists $x_2\in\Gamma_{1,q}(x_1)$ such that $\partial_{\Gamma}(x_2,x_0)=q-1$,  then $(x_0,x_2)\in\Gamma_{2,q-1}$ since $p^{(1,q-1)}_{(1,q),(1,h)}=0$ for all $h$, which implies $x_2\in P_{(2,q-1),(q,1)}(x_0,x_1)$, a contradiction. By the commutativity of $\Gamma$, (i) holds.

Suppose $p^{(1,q)}_{(2,q-1),(q,1)}\neq0$. According to 
Lemma \ref{intersection numners} \ref{intersection numners 2}, we have $p_{(1,q),(1,q)}^{(2,q-1)}\neq0$. By \cite[Lemma 2.6]{MW24}, (ii) holds.
\end{proof}



Now we are ready to prove Lemma \ref{arc}.

\begin{proof}[Proof of Lemma \ref{arc}]
Suppose for the contrary that $q\geq4$.   Let $(x_0,x_1,\ldots,x_{q})$ be a circuit in $\Gamma$ with $(x_0,x_1)\in\Gamma_{1,q}$ and $(x_i,x_{i+1})\in\Gamma_{1,r_i}$ for $r_i\geq1$ with $1\leq i\leq q$ and $x_{q+1}=x_0$. According to Lemma \ref{circuit}, we may assume $r_i=q$ for all $1\leq i\leq q$, or $r_i\neq q$ for all $1\leq i\leq q$.

In the following, we construct a sequence of vertices $(y_1, y_2, \ldots, y_q)$ using Corollary  \ref{arckey*} inductively such that $y_k\in C(y_{k-1},x_{k+1})\cap\Sigma_2(x_k)\cap\Sigma_2(x_{k+2})$ for $k\in\{1,2,\ldots, q\}$, where $y_0=x_0$.


By the first statement of Corollary  \ref{arckey*}, we have 
$(x_0,x_2)\in\Sigma_2$. In view of Lemma \ref{sigma2}, there exists $y_1\in C(x_0,x_2)$ such that $(x_1,y_1)\in\Sigma_2$. 
If $r_i=q$ for all $1\leq i\leq q$, from Corollary  \ref{arckey*}, then $(x_0,y_1,x_2)$ is a path in $\Gamma$; if $r_i\neq q$ for all $1\leq i\leq q$, from Corollary \ref{arckey*} \ref{arckey*2}, then $y_1\in P_{(1,r_1),(1,q)}(x_0,x_2)$. It follows from the statement of Corollary  \ref{arckey*} that
$(y_1,x_3)\in\Sigma_2$. Since $x_2\in C(y_1,x_3)$, from Lemma \ref{sigma2}, there exists $y_2\in C(y_1,x_3)$ such that $(y_2,x_2)\in\Sigma_2$.

If $r_i\neq q$ for all $1\leq i\leq q$, from Corollary \ref{arckey*} \ref{arckey*2}, then $y_2\in P_{(1,r_2),(1,q)}(y_1,x_3)$ since $x_2\in P_{(1,q),(1,r_2)}(y_1,x_3)$.  Now suppose $r_i=q$ for all $1\leq i\leq q$. If $(y_1,x_2)\in\Gamma_{1,q}$, from Corollary \ref{arckey*}, then $(y_1,y_2,x_3)$ is a path in $\Gamma$; if $(y_1,x_2)\in\Gamma_{1,r}$ with $r\neq q$, from Corollary \ref{arckey*} \ref{arckey*1}, then $y_2\in P_{(1,q),(1,r)}(y_1,x_3)$.
We conclude that $(y_1,y_2,x_3)$ is a path, and $(y_2,x_3)\in\Gamma_{1,q}$ or $(x_3,x_4)\in\Gamma_{1,q}$. It follows from the statement of Corollary  \ref{arckey*} that $(y_2,x_4)\in\Sigma_2$. 
By induction, there exist $y_3,y_4,\ldots,y_q$ satisfy the desired property.



Note that $x_0\in{X\choose e}$. Since  $(x_0,x_2)\in\Sigma_2$ and $x_1\in C(x_0,x_2)$, there exist $a_1,a_2\in x_0$ and  $b_1,b_2\in X\setminus x_0$ such that $x_1=x_0(a_1,b_1)$ and $x_2=x_0(\{a_1,a_2\},\{b_1,b_2\})$, which implies $y_1=x_0(a_2,b_2)$. We claim that, for $2\leq k\leq q$, there exist $u_k\subseteq x_0\setminus\{a_1\}$ and $v_k\subseteq X\setminus(x_0\cup \{b_1\})$ with $|u_k|=|v_k|$ such that $$x_k=x_0(\{a_1\}\cup u_k,\{b_1\}\cup v_k) ~\mbox{and}~ y_{k-1}=x_0(u_k,v_k).$$  We prove this claim by induction on $k$. The case $k=2$ is valid. Suppose that $k>2$. By the inductive hypothesis, one obtains $x_k=x_0(\{a_1\}\cup u_k,\{b_1\}\cup v_k)$ and $y_{k-1}=x_0(u_k,v_k)$ for   $u_k\subseteq x_0\setminus\{a_1\}$ and $v_k\subseteq X\setminus(x_0\cup\{b_1\})$ with $|u_k|=|v_k|$. Since $x_k=y_{k-1}(a_1,b_1)$ and $x_{k+1}\in\Sigma_1(x_k)\cap\Sigma_2(y_{k-1})$, we have $x_{k+1}=y_{k-1}(\{a_1,a\},\{b_1,b\})$ for some $a\in y_{k-1}\setminus\{a_1\}$ and $b\in X\setminus(y_{k-1}\cup\{b_1\}).$ Then $a\in (x_0\setminus(u_k\cup\{a_1\}))\cup v_k$ and $b\in((X\setminus x_0)\setminus(v_k\cup\{b_1\}))\cup u_k$. Since $y_k\in C(y_{k-1},x_{k+1})$ and $(y_k,x_k)\in\Sigma_2$, one gets $y_k=y_{k-1}(a,b)$.

If $a\in x_0\setminus(u_k\cup\{a_1\})$ and $b\in (X\setminus x_0)\setminus(v_k\cup\{b_1\})$, then $$x_{k+1}=x_0(\{a_1,a\}\cup u_k,\{b_1,b\}\cup v_k) ~\mbox{and}~ y_k=x_0(\{a\}\cup u_k,\{b\}\cup v_k).$$
If $a\in x_0\setminus(u_k\cup\{a_1\})$ and $b\in u_k$, then $$x_{k+1}=x_0(\{a_1,a\}\cup(u_k\setminus\{b\}),\{b_1\}\cup v_k)~\mbox{and}~y_k=x_0(\{a\}\cup(u_k\setminus\{b\}),v_k).$$
If $a\in v_k$ and $b\in(X\setminus x_0)\setminus(v_k\cup\{b_1\})$, then  $$x_{k+1}=x_0(\{a_1\}\cup u_k,\{b_1,b\}\cup(v_k\setminus\{a\}))~\mbox{and}~y_k=x_0(u_k,\{b\}\cup(v_k\setminus\{a\})).$$
If  $a\in v_k$ and $b\in u_k$, then $$x_{k+1}=x_0(\{a_1\}\cup(u_k\setminus\{b\}),\{b_1\}\cup(v_k\setminus\{a\}))~\mbox{and}~y_k=x_0(u_k\setminus\{b\},v_k\setminus\{a\}).$$  Since all possible cases for the choices of $a$ and $b$ have been considered, the claim holds.

By the claim, we obtain $x_q=x_0(\{a_1\}\cup x,\{b_1\}\cup y)$ for some $x\subseteq x_0$ and $y\subseteq X\setminus x_0$ with $|x|=|y|$. Since $x_q\in\Sigma_1(x_0)$, we get $x_q=x_0(a_1,b_1)=x_1$, a contradiction.
\end{proof}

\section{Proof of Lemma \ref{1211i}}

To prove Lemma \ref{1211i}, we need an auxiliary lemma.

\begin{lemma}\label{tildei}
%
%
%
%
Let $(x,y)\in\Sigma_1$, $\tilde{h}\in\{(1,2),(2,1)\}$ and $\tilde{i},\tilde{j}\in\{(1,1),(1,2),(2,1)\}$. Suppose $P_{\tilde{i},\tilde{j}}(x,y)\neq\emptyset$.  Then exactly one of the following holds:
\begin{enumerate}

\item\label{tildei1} $\tilde{i}=\tilde{j}$ and $P_{(1,1),\tilde{h}}(x,y)\cup P_{\tilde{h},(1,1)}(x,y)\subseteq Y_i(x,y)$ for some $i\in\{1,2\}$  with $\tilde{h}\neq\tilde{i}$;

\item \label{tildei2} $(\tilde{i},\tilde{j})\in\{((1,1),\tilde{h}),(\tilde{h},(1,1))\}$ and $P_{\tilde{i}^*,\tilde{j}^*}(x,y)\subseteq Y_i(x,y)$ for some $i\in\{1,2\}$;

\item \label{tildei3} $\{\tilde{i},\tilde{j}\}=\{(1,2),(2,1)\}$, and $P_{(1,1),\tilde{h}}(x,y)=\emptyset$ or $P_{(1,1),\tilde{h}}(x,y)\cup P_{\tilde{h},(1,1)}(x,y)\nsubseteq\Sigma_2(z)$ for $z\in P_{\tilde{i},\tilde{j}}(x,y)$.

\end{enumerate}
Moreover, in the case \ref{tildei1}, if $P_{(1,1),\tilde{h}}(x,y)\cup P_{\tilde{h},(1,1)}(x,y)\neq\emptyset$, then $P_{\tilde{i},\tilde{i}}\subseteq Y_i(x,y)$;  in the case \ref{tildei2}, if $P_{\tilde{i}^*,\tilde{j}^*}(x,y)\neq\emptyset$, then $P_{\tilde{i},\tilde{j}}\subseteq Y_i(x,y)$.
\end{lemma}
\begin{proof}
Let $z\in P_{\tilde{i},\tilde{j}}(x,y)$. Then $z\in Y_i(x,y)$ for some $i\in\{1,2\}$.

\textbf{Case 1.} $\tilde{i}=\tilde{j}$.

If  $P_{(1,1),\tilde{h}}(x,y)\cup P_{\tilde{h},(1,1)}(x,y)=\emptyset$, then (i) holds. Now we consider the case that $P_{(1,1),\tilde{h}}(x,y)\cup P_{\tilde{h},(1,1)}(x,y)\neq\emptyset$. Assume the contrary, namely,   $P_{(1,1),\tilde{h}}(x,y)\nsubseteq\Sigma_1(z)$ or $P_{\tilde{h},(1,1)}(x,y)\nsubseteq\Sigma_1(z)$ with $\tilde{h}\neq\tilde{i}$. Since the proofs are similar,  we may assume  $P_{(1,1),\tilde{h}}(x,y)\nsubseteq\Sigma_1(z)$. Let $w\in P_{(1,1),\tilde{h}}(x,y)$ with $w\notin\Sigma_1(z)$. Then $x\in P_{(1,1),\tilde{i}}(w,z)$ and $y\in P_{\tilde{i},\tilde{h}^*}(z,w)$.
If $\tilde{i}=(1,1)$, then $(w,z)\in\Gamma_{2,2}$ and $p^{(2,2)}_{(1,1),(1,1)}p^{(2,2)}_{(1,1),(1,2)}\neq0$; if $\tilde{i}=(1,2)$ or $(2,1)$, then $\tilde{h}=(2,1)$ or $(1,2)$, which implies $(w,z)\in\Gamma_{2,2}$ and $p^{(2,2)}_{(1,1),(1,2)}p^{(2,2)}_{(1,2),(1,2)}\neq0$, both contrary to Lemma \ref{22}. Thus, $P_{(1,1),\tilde{h}}(x,y)\cup P_{\tilde{h},(1,1)}(x,y)\subseteq\Sigma_1(z)$. By Lemma \ref{sigma1} \ref{sigma11}, one gets $P_{(1,1),\tilde{h}}(x,y)\cup P_{\tilde{h},(1,1)}(x,y)\subseteq Y_i(x,y)$. Since $z\in P_{\tilde{i},\tilde{i}}(x,y)$ was arbitrary, from Lemma \ref{sigma1} \ref{sigma11},  we have $P_{\tilde{i},\tilde{i}}(x,y)\cup P_{(1,1),\tilde{h}}(x,y)\cup P_{\tilde{h},(1,1)}(x,y)\subseteq Y_i(x,y)$.

\textbf{Case 2.} $(\tilde{i},\tilde{j})\in\{((1,1),\tilde{h}),(\tilde{h},(1,1))\}$.

If $P_{\tilde{i}^*,\tilde{j}^*}(x,y)=\emptyset$, then (ii) holds. Now we consider the case that  $P_{\tilde{i}^*,\tilde{j}^*}(x,y)\neq\emptyset$. Assume the contrary, namely, $P_{\tilde{i}^*,\tilde{j}^*}(x,y)\nsubseteq\Sigma_1(z)$. Pick a vertex $w\in P_{\tilde{i}^*,\tilde{j}^*}(x,y)$ with $w\notin\Sigma_1(z)$. Then $x\in P_{\tilde{i},\tilde{i}}(w,z)$ and $y\in P_{\tilde{j}^*,\tilde{j}^*}(w,z)$, which implies $(w,z)\in\Gamma_{2,2}$ and $p^{(2,2)}_{(1,1),(1,1)}p^{(2,2)}_{(1,2),(1,2)}\neq0$,  contrary to Lemma \ref{22}. Thus, $P_{\tilde{i}^*,\tilde{j}^*}(x,y)\subseteq\Sigma_1(z)$. By Lemma \ref{sigma1} \ref{sigma11}, we have $P_{\tilde{i}^*,\tilde{j}^*}(x,y)\subseteq Y_i(x,y)$. Since $z\in P_{\tilde{i},\tilde{j}}(x,y)$ was arbitrary, from Lemma \ref{sigma1} \ref{sigma11},  we get $P_{\tilde{i},\tilde{j}}(x,y)\cup P_{\tilde{i}^*,\tilde{j}^*}(x,y)\subseteq Y_i(x,y)$.

\textbf{Case 3.} $\{\tilde{i},\tilde{j}\}=\{(1,2),(2,1)\}$.

Since the proofs are similar,  we may assume $(\tilde{i},\tilde{j})=((1,2),(2,1))$.  Suppose for the contrary that $P_{(1,1),\tilde{h}}(x,y)\neq\emptyset$ and $P_{(1,1),\tilde{h}}(x,y)\cup P_{\tilde{h},(1,1)}(x,y)\subseteq\Sigma_2(z)$  for some $\tilde{h}\in\{(1,2),(2,1)\}$.  Since $P_{(1,1),\tilde{h}}(x,y)\neq\emptyset$ and $P_{\tilde{h},(1,1)}(x,y)\neq\emptyset$, there exist $w_1\in P_{(1,1),\tilde{h}}(x,y)$ and $w_2\in P_{\tilde{h},(1,1)}(x,y)$. Then $x\in P_{(1,1),(1,2)}(w_1,z)\cap P_{\tilde{h}^*,(1,2)}(w_2,z)$ and $y\in P_{\tilde{h},(1,2)}(w_1,z)\cap P_{(1,1),(1,2)}(w_2,z)$. Since $w_1,w_2\in\Sigma_2(z)$, from Lemma \ref{22}, one has $w_1,w_2\in\Gamma_{3,2}(z)$ and $p^{(2,3)}_{(1,1),(1,2)}p^{(2,3)}_{(1,2),(1,1)}p^{(2,3)}_{(1,2),(2,1)}p^{(2,3)}_{(2,1),(1,2)}p^{(2,3)}_{(1,2),(1,2)}\neq0$, which implies $c_2\geq5$ from \eqref{a1c2}, contrary to the fact that $c_2=4$.
%
%
%

This completes the proof of the lemma.
\end{proof}

\begin{proof}[Proof of Lemma \ref{1211i}] Note that $\max T=3$. By way of contradiction, we assume   $p^{(1,2)}_{(1,1),\tilde{i}}\neq0$ for some $\tilde{i}\in\{(1,1),(1,2),(2,1)\}$.

\textbf{Case 1.} $\tilde{i}=(1,2)$.

Note that $p^{(1,2)}_{(1,2),(1,1)}=p^{(1,2)}_{(1,1),(1,2)}\neq0$.  Next, we reach a contradiction step by step.


\begin{step}\label{1-1}
{\rm Show that  $P_{(1,1),(1,2)}(x,y)\cup P_{(1,2),(1,1)}(x,y)\subseteq Y_i(x,y)$ for some $i\in\{1,2\}$ when $(x,y)\in\Gamma_{1,2}$.}
\end{step}

If $p^{(1,2)}_{(a,1),(a,1)}\neq0$ for some $a\in\{1,2\}$, by setting $\tilde{i}=(a,1)$ in Lemma \ref{tildei} \ref{tildei1}, then $P_{(1,1),(1,2)}(x,y)\cup P_{(1,2),(1,1)}(x,y)\subseteq Y_i(x,y)$ for some $i\in\{1,2\}$.
Now we consider the case $p^{(1,2)}_{(2,1),(2,1)}=p^{(1,2)}_{(1,1),(1,1)}=0$. Then each shortest path from $y$ to $x$ consists of an arc of type $(1,2)$ and an arc of type $(1,1)$. Let $(y,w,x)$ be a path in $\Gamma$ such that $w\in P_{(1,1),(2,1)}(x,y)$.  Since $p^{(1,2)}_{(1,1),(1,2)}\neq0$, by setting $\tilde{i}=(1,1)$ and $\tilde{j}=(1,2)$ in Lemma \ref{tildei} \ref{tildei2}, we get $P_{(1,1),(2,1)}(x,y)\subseteq Y_i(x,y)$ for some $i\in\{1,2\}$. Since $P_{(1,1),(2,1)}(x,y)\neq\emptyset,$ from the second statement of Lemma \ref{tildei}, one gets $P_{(1,1),(1,2)}(x,y)\subseteq Y_i(x,y)$.

Suppose $P_{(1,2),(1,1)}(x,y)\nsubseteq Y_i(x,y).$ Pick a vertex $z\in P_{(1,2),(1,1)}(x,y)$ such that $z\in Y_j(x,y)$ with $j\in\{1,2\}\setminus\{i\}$.  Lemma \ref{sigma1} \ref{sigma12} implies $(w,z)\in\Sigma_2$. Since $x\in P_{(1,1),(1,2)}(w,z)$ and $y\in P_{(1,1),(1,2)}(z,w)$, one has $(w,z)\in\Gamma_{2,2}$ and  $p^{(2,2)}_{(1,1),(1,2)}\neq0$. Then Lemma \ref{22} \ref{223} holds.

Since $y\in P_{(1,2),(1,2)}(x,w)$ with $(x,w)\in\Gamma_{1,1}$, there exists $u\in P_{(1,2),(1,2)}(w,x)$ with $u\neq y$. It follows that
\begin{align}\tag{B.1}\label{1-1gongshi1}x\in P_{(1,2),(1,2)}(u,y)~\mbox{and}~w\in P_{(1,2),(1,2)}(y,u).\end{align} If $u\in\Sigma_2(y)$, from \eqref{1-1gongshi1}, then $(u,y)\in\Gamma_{2,2}$, which implies $p^{(2,2)}_{(1,2),(1,2)}\neq0$, contrary to  Lemma \ref{22} \ref{223}. Then $u\in\Sigma_1(y)$.  Since $p^{(1,2)}_{(2,1),(2,1)}=0$, from \eqref{1-1gongshi1}, we get $(y,u)\in\Gamma_{1,1}$.

Since $u\in C(x,y)$ with $(u,w)\in\Sigma_1$, from Lemma \ref{sigma1} \ref{sigma11}, we have $u\in Y_i(x,y)$, and so $(u,z)\in\Sigma_2$.
Since  $y\in P_{(1,1),(1,1)}(u,z)$, we get $(u,z)\in\Gamma_{2,2}$ and $p^{(2,2)}_{(1,1),(1,1)}\neq0$, contrary to Lemma \ref{22} \ref{223}. Thus, $P_{(1,2),(1,1)}(x,y)\subseteq Y_i(x,y),$ and so $P_{(1,1),(1,2)}(x,y)\cup P_{(1,2),(1,1)}(x,y)\subseteq Y_i(x,y)$.

\begin{step}\label{1-2*}
{\rm  Show that $P_{(1,1),(2,1)}(x,y)\cup P_{(2,1),(1,1)}(x,y)\cup P_{(1,2),(2,1)}(x,y)\cup P_{(2,1),(1,2)}(x,y)\subseteq Y_i(x,y)$ when $(x,y)\in\Gamma_{1,2}$ and $P_{(1,1),(1,2)}(x,y)\cup P_{(1,2),(1,1)}(x,y)\subseteq Y_i(x,y)$ for some $i\in\{1,2\}$.}
\end{step}

We prove this step by contradiction. Assume the contrary, namely, $P_{\tilde{a},\tilde{b}}(x,y)\nsubseteq Y_i(x,y)$ for some $(\tilde{a},\tilde{b})\in\{((1,1),(2,1)),((2,1),(1,1)),((1,2),(2,1)),((2,1),(1,2))\}$. Then $P_{\tilde{a},\tilde{b}}(x,y)\neq\emptyset$.

Suppose $(\tilde{a},\tilde{b})\in\{((1,1),(2,1)),((2,1),(1,1))\}$. Since $p^{(1,2)}_{(1,1),(1,2)}=p^{(1,2)}_{(1,2),(1,1)}\neq0$, by setting $(\tilde{i},\tilde{j})=(\tilde{a}^*,\tilde{b}^*)$ in Lemma \ref{tildei} \ref{tildei2}, one obtains $P_{\tilde{a},\tilde{b}}(x,y)\subseteq Y_j(x,y)$ for $j\in\{1,2\}\setminus\{i\}$. Since $P_{\tilde{a},\tilde{b}}(x,y)\neq\emptyset$, from the second statement of Lemma \ref{tildei}, we get $P_{\tilde{a}^*,\tilde{b}^*}(x,y)\subseteq Y_j(x,y)$, contrary to the fact that $P_{(1,1),(1,2)}(x,y)\cup P_{(1,2),(1,1)}(x,y)\subseteq Y_i(x,y)$.

Suppose $(\tilde{a},\tilde{b})\in\{((1,2),(2,1)),((2,1),(1,2))\}$. Then there exists $z\in P_{\tilde{a},\tilde{b}}(x,y)\cap Y_j(x,y)$ with $j\in\{1,2\}\setminus\{i\}$. Note that $p^{(1,2)}_{(1,1),(1,2)}\neq0$. Since $P_{(1,1),(1,2)}(x,y)\cup P_{(1,2),(1,1)}(x,y)\subseteq Y_i(x,y)$, by Lemma \ref{sigma1} \ref{sigma12}, we get $P_{(1,1),(1,2)}(x,y)\cup P_{(1,2),(1,1)}(x,y)\subseteq\Sigma_2(z)$, contrary to Lemma \ref{tildei} \ref{tildei3}. 

\begin{step}\label{1-3}
{\rm   Show that $P_{(1,2),(1,2)}(x,y)\subseteq Y_i(x,y)$ when $(x,y)\in\Gamma_{1,2}$ and $P_{(1,1),(1,2)}(x,y)\cup P_{(1,2),(1,1)}(x,y)\subseteq Y_i(x,y)$ for some $i\in\{1,2\}$.}
\end{step}

Let $z\in P_{(1,1),(1,2)}(x,y)$. Suppose for the contrary that $P_{(1,2),(1,2)}(x,y)\nsubseteq Y_i(x,y)$. Pick  a vertex $w\in P_{(1,2),(1,2)}(x,y)$ with $w\notin Y_i(x,y)$. By Lemma \ref{sigma1} \ref{sigma12}, we have $(z,w)\in\Sigma_2$. Since $c_2=4$, $x\in P_{(1,1),(1,2)}(z,w)$ and $y\in P_{(1,2),(2,1)}(z,w)$, 
there exist
\begin{align}z_1\in P_{(1,2),(1,1)}(z,w)~\mbox{and}~z_2\in P_{(2,1),(1,2)}(z,w)\nonumber
 \end{align} such that $C(z,w)=\{x,y,z_1,z_2\}$. By Lemma \ref{sigma2}, we get $(x,z_2)\in\Sigma_2$ and $C(x,z_2)=\{y,z,w,z_1\}$, or $(x,z_1)\in\Sigma_2$ and $C(x,z_1)=\{y,z,w,z_2\}$.

If  $(x,z_2)\in\Sigma_2$ and $C(x,z_2)=\{y,z,w,z_1\}$, by $w\in P_{(1,2),(2,1)}(x,z_2)$ and $z\in P_{(1,1),(2,1)}(x,z_2)$, then $y\in P_{(2,1),(1,2)}(x,z_2)$ or $y\in P_{(2,1),(1,1)}(x,z_2)$, contrary to the fact that $(x,y)\in\Gamma_{1,2}$. Thus, $(x,z_1)\in\Sigma_2$ and $C(x,z_1)=\{y,z,w,z_2\}$. Since $(z,w)\in\Sigma_2$, from Lemma \ref{sigma2}, we get $(y,z_2)\in\Sigma_2$.

If $(x,z_2)\in\Gamma_{1,1}$, then $z_2\in P_{(1,1),(1,2)}(x,w)$ and $y\in P_{(1,2),(2,1)}(x,w)$ with $(x,w)\in\Gamma_{1,2}$ and $(y,z_2)\in\Sigma_2$, contrary to Step \ref{1-2*} and Lemma \ref{sigma1} \ref{sigma11}. If $(x,z_2)\in\Gamma_{2,1}$, then  $z\in P_{(1,2),(1,1)}(z_2,x)$ and $w\in P_{(1,2),(2,1)}(z_2,x)$ with  $(z,w)\in\Sigma_2$, contrary to Step \ref{1-2*} and Lemma \ref{sigma1} \ref{sigma11}. Since $(x,z_2)\in\Sigma_1$ and $\max T=3$, we get  $(x,z_2)\in\Gamma_{1,2}$.

Since $z\in P_{(1,1),(2,1)}(x,z_2)$ and $w\in P_{(1,2),(2,1)}(x,z_2)$, from Step \ref{1-2*}, we have $z,w\in Y_{k}(x,z_2)$ for some $k\in\{1,2\}$, contrary to that fact that $(z,w)\in\Sigma_2$.

\begin{step}
{\rm We reach a contradiction.}
\end{step}

Let $(x,y)\in\Gamma_{1,2}$. In view of Step \ref{1-1}, we obtain $P_{(1,1),(1,2)}(x,y)\cup P_{(1,2),(1,1)}(x,y)\subseteq Y_i(x,y)$ for some $i\in\{1,2\}$.  Let $\tilde{l}\in\{(1,1),(2,1)\}$. If $P_{\tilde{l},\tilde{l}}(x,y)=\emptyset$, then $P_{\tilde{l},\tilde{l}}(x,y)\subseteq Y_i(x,y)$. Observe that $p^{(1,2)}_{(1,1),(1,2)}\neq0$. If $P_{\tilde{l},\tilde{l}}(x,y)\neq\emptyset$, by setting $\tilde{i}=\tilde{l}$ in Lemma \ref{tildei} \ref{tildei1}, one obtains $P_{\tilde{l},\tilde{l}}(x,y)\subseteq Y_i(x,y)$  from the second statement of Lemma \ref{tildei}. Since $\tilde{l}\in\{(1,1),(2,1)\}$ was arbitrary, we get $P_{(1,1),(1,1)}(x,y)\cup P_{(2,1),(2,1)}(x,y)\subseteq Y_i(x,y)$.



By   Steps \ref{1-2*} and \ref{1-3}, we get $P_{\tilde{i},\tilde{j}}(x,y)\subseteq Y_i(x,y)$ for all $\tilde{i},\tilde{j}\in\{(1,1),(1,2),(2,1)\}$. Since $\max T=3$, one has $C(x,y)\subseteq Y_i(x,y)$, which implies $Y_j(x,y)=\emptyset$ with $j\in\{1,2\}\setminus\{i\}$. By Lemma \ref{sigma1}, one gets  $e=1$ or $m-e=1$, contrary to the fact that $m-e\geq e\geq2$.

\textbf{Case 2.} $\tilde{i}=(2,1)$.

Note that $p^{(1,2)}_{(2,1),(1,1)}=p^{(1,2)}_{(1,1),(2,1)}\neq0$. By Case 1, one has $p^{(1,2)}_{(1,2),(1,1)}=p^{(1,2)}_{(1,1),(1,2)}=0$. Next, we reach a contradiction step by step.

\begin{steppp}\label{2-1}
{\rm  Show that $P_{(1,1),(2,1)}(x,y)\cup P_{(2,1),(1,1)}(x,y)\subseteq Y_i(x,y)$ for some $i\in\{1,2\}$ when $(x,y)\in\Gamma_{1,2}$.}
\end{steppp}

If $p^{(1,2)}_{(1,a),(1,a)}\neq0$ for some $a\in\{1,2\}$, by setting $\tilde{i}=(1,a)$ in Lemma \ref{tildei} \ref{tildei1}, then  $P_{(1,1),(2,1)}(x,y)\cup P_{(2,1),(1,1)}(x,y)\subseteq Y_i(x,y)$ for some $i\in\{1,2\}$. Now we consider the case  $p^{(1,2)}_{(1,1),(1,1)}=p^{(1,2)}_{(1,2),(1,2)}=0$. Lemma \ref{intersection numners} \ref{intersection numners 2} implies $p^{(1,2)}_{(1,2),(2,1)}=0$. Then $p^{(1,2)}_{(1,2),\tilde{h}}=0$ for all $\tilde{h}\in\{(1,1),(1,2),(2,1)\}$.

Suppose that $p^{(1,2)}_{(2,1),(2,1)}=0$.   If there exist $z_1,z_2\in P_{(1,1),(2,1)}(x,y)$ or  $w_1,w_2\in P_{(2,1),(1,1)}(x,y)$ such that $(z_1,z_2)\in\Sigma_1$ or $(w_1,w_2)\in\Sigma_1$, then $z_2\in P_{(1,2),\tilde{h}}(y,z_1)$ or $w_2\in P_{\tilde{h},(1,2)}(w_1,x)$ for some $\tilde{h}\in\{(1,1),(1,2),(2,1)\}$, contrary to the fact that $p^{(1,2)}_{(1,2),\tilde{h}}=p^{(1,2)}_{\tilde{h},(1,2)}=0$ for all $\tilde{h}\in\{(1,1),(1,2),(2,1)\}$. By Lemma \ref{sigma1}, we have $|P_{(1,1),(2,1)}(x,y)\cap Y_i(x,y)|\leq1$ and $|P_{(2,1),(1,1)}(x,y)\cap Y_i(x,x_1)|\leq1$ for all $i\in\{1,2\}$. Since $p^{(1,2)}_{(1,1),(1,2)}=p^{(1,2)}_{(1,1),(1,1)}=0$ and $p^{(1,2)}_{(1,2),\tilde{h}}=0$ for all $\tilde{h}\in\{(1,1),(1,2),(2,1)\}$, by Lemma \ref{intersection numners} \ref{intersection numners 2}, we get $p^{(1,2)}_{\tilde{i},\tilde{j}}\neq0$ for  $\tilde{i},\tilde{j}\in\{(1,1),(1,2),(2,1)\}$ if and only if $(\tilde{i},\tilde{j})\in\{((1,1),(2,1)),((2,1),(1,1))\}$. It follows that $|Y_i(x,y)|\leq2$ for all $i\in\{1,2\}$. By Lemma \ref{sigma1}, one has $e\leq3$ and $m-e\leq3$, which imply that  $m\leq6$ and $|V(\Gamma)|\leq20$, contrary to \cite{H}. Therefore, $p^{(1,2)}_{(2,1),(2,1)}\neq0$.

Let $w\in P_{(2,1),(2,1)}(x,y)$. If there exists a vertex $z\in P_{(1,1),(2,1)}(x,y)\cup P_{(2,1),(1,1)}(x,y)$ such that $(z,w)\in\Sigma_1$, then $z\in P_{(1,2),\tilde{i}}(y,w)$ or $z\in P_{\tilde{i},(1,2)}(w,x)$ for some  $\tilde{i}\in\{(1,1),(1,2),(2,1)\}$,  contrary to the fact that $p^{(1,2)}_{(1,2),\tilde{h}}=p^{(1,2)}_{\tilde{h},(1,2)}=0$ for all $\tilde{h}\in\{(1,1),(1,2),(2,1)\}$. Therefore, $P_{(1,1),(2,1)}(x,y)\cup P_{(2,1),(1,1)}(x,y)\subseteq\Sigma_2(w)$. Let $w\in Y_j(x,y)$ for some $j\in\{1,2\}$.  Lemma \ref{sigma1} \ref{sigma12} implies $P_{(1,1),(2,1)}(x,y)\cup P_{(2,1),(1,1)}(x,y)\subseteq Y_i(x,y)$ with $\{i,j\}=\{1,2\}$.

\begin{steppp}\label{2-2*}
{\rm  Show that $P_{(1,2),(2,1)}(x,y)\cup P_{(2,1),(1,2)}(x,y)\subseteq Y_i(x,y)$ when $(x,y)\in\Gamma_{1,2}$ and $P_{(1,1),(2,1)}(x,y)\cup P_{(2,1),(1,1)}(x,y)\subseteq Y_i(x,y)$ for some $i\in\{1,2\}$.}
\end{steppp}

We prove this step by contradiction. Assume the contrary, namely, $P_{\tilde{a},\tilde{b}}(x,y)\nsubseteq Y_i(x,y)$ for some $(\tilde{a},\tilde{b})\in\{((1,2),(2,1)),((2,1),(1,2))\}$.  Then there exists $z\in P_{\tilde{a},\tilde{b}}(x,y)\cap Y_j(x,y)$ with $j\in\{1,2\}\setminus\{i\}$. Note that $p^{(1,2)}_{(1,1),(2,1)}\neq0$. Since $P_{(1,1),(2,1)}(x,y)\cup P_{(2,1),(1,1)}(x,y)\subseteq Y_i(x,y)$, by Lemma \ref{sigma1} \ref{sigma12}, we get $P_{(1,1),(2,1)}(x,y)\cup P_{(2,1),(1,1)}(x,y)\subseteq\Sigma_2(z)$, contrary to Lemma \ref{tildei} \ref{tildei3}.

\begin{steppp}\label{2-3}
{\rm  Show that $P_{(2,1),(2,1)}(x,y)\subseteq Y_i(x,y)$ when $(x,y)\in\Gamma_{1,2}$ and $P_{(1,1),(2,1)}(x,y)\cup P_{(2,1),(1,1)}(x,y)\subseteq Y_i(x,y)$ for some $i\in\{1,2\}$.}
\end{steppp}

Let $z\in P_{(1,1),(2,1)}(x,y)$. Assume the contrary, namely, $P_{(2,1),(2,1)}(x,y)\nsubseteq Y_i(x,y).$ Pick a vertex $w\in P_{(2,1),(2,1)}(x,y)$ such that $w\notin Y_i(x,y)$. Since $z\in Y_i(x,y)$, from Lemma \ref{sigma1} \ref{sigma12}, we have $(w,z)\in\Sigma_2$. Since $c_2=4$, $x\in P_{(1,2),(1,1)}(w,z)$ and $y\in P_{(2,1),(1,2)}(w,z)$, there exist
\begin{align}z_1\in P_{(1,1),(1,2)}(w,z)~\mbox{and}~z_2\in P_{(1,2),(2,1)}(w,z)\nonumber
 \end{align} such that $ C(w,z)=\{x,y,z_1,z_2\}$. By Lemma \ref{sigma2}, we get $(y,z_1)\in\Sigma_2$ and $C(y,z_1)=\{x,z,w,z_2\}$,  or $(y,z_2)\in\Sigma_2$ and $C(y,z_2)=\{x,z,w,z_1\}$.

Suppose $(y,z_1)\in\Sigma_2$ and $C(y,z_1)=\{x,z,w,z_2\}$. By Lemma \ref{sigma2}, we have $(x,z_2)\in\Sigma_2$.
Since $w\in P_{(2,1),(1,2)}(x,z_2)$ and $z\in P_{(1,1),(1,2)}(x,z_2)$, we get $(x,z_1)\in\Gamma_{1,2}$. Then  $w\in P_{(2,1),(1,1)}(x,z_1)$ and $z\in P_{(1,1),(2,1)}(x,z_1)$. By Step \ref{2-1} and Lemma \ref{sigma1} \ref{sigma11}, one gets $(w,z)\in\Sigma_1$, contrary to the fact that $(w,z)\in\Sigma_2$.

Suppose $(y,z_2)\in\Sigma_2$ and $C(y,z_2)=\{x,z,w,z_1\}$. Note that $p_{(1,1),(1,2)}^{(1,2)}=0$. Since $(w,z_2)\in\Gamma_{1,2}$, one has $x\notin P_{(1,2),(1,1)}(w,z_2)$, and so $(x,z_2)\notin\Gamma_{1,1}$. The fact $z\in P_{(1,1),(1,2)}(x,z_2)$ implies $(x,z_2)\notin\Gamma_{1,2}$. Since $\max T=3$ and $x\in C(y,z_2)$, we get $(x,z_2)\in\Gamma_{2,1}$. Then $w\in P_{(2,1),(1,2)}(z_2,x)$ and $z\in P_{(2,1),(1,1)}(z_2,x)$. By Step \ref{2-2*} and Lemma \ref{sigma1} \ref{sigma11}, one has $(w,z)\in\Sigma_1$, contrary to the fact that $(w,z)\in\Sigma_2$. 

\begin{steppp}
{\rm We reach a contradiction.}
\end{steppp}

Let $(x,y)\in\Gamma_{1,2}$. In view of Step \ref{2-1}, we obtain $P_{(1,1),(2,1)}(x,y)\cup P_{(2,1),(1,1)}(x,y)\subseteq Y_i(x,y)$ for some $i\in\{1,2\}$. Let $\tilde{l}\in\{(1,1),(1,2)\}$. If $P_{\tilde{l},\tilde{l}}(x,y)=\emptyset$, then $P_{\tilde{l},\tilde{l}}(x,y)\subseteq Y_i(x,y)$.  Observe that $p^{(1,2)}_{(1,1),(2,1)}\neq0$. If $P_{\tilde{l},\tilde{l}}(x,y)\neq\emptyset$, by setting $\tilde{i}=\tilde{l}$ in Lemma \ref{tildei} \ref{tildei1}, one obtains $P_{\tilde{l},\tilde{l}}(x,y)\subseteq Y_i(x,y)$  from the second statement of Lemma \ref{tildei}. Since $\tilde{l}\in\{(1,1),(1,2)\}$ was arbitrary, one has $P_{(1,1),(1,1)}(x,y)\cup P_{(1,2),(1,2)}(x,y)\subseteq Y_i(x,y)$.


Since $p^{(1,2)}_{(1,1),(1,2)}=p^{(1,2)}_{(1,2),(1,1)}=0$, from   Steps \ref{2-2*} and \ref{2-3}, we obtain $P_{\tilde{i},\tilde{j}}(x,y)\subseteq Y_i(x,y)$ for all $\tilde{i},\tilde{j}\in\{(1,1),(1,2),(2,1)\}$. Since $\max T=3$, one has $C(x,y)\subseteq Y_i(x,y)$, which implies $Y_j(x,y)=\emptyset$ with $j\in\{1,2\}\setminus\{i\}$. By Lemma \ref{sigma1}, one gets  $e=1$ or $m-e=1$, contrary to the fact that $m-e\geq e\geq2$.

\textbf{Case 3.} $\tilde{i}=(1,1)$.

By  Cases 1 and 2, one has $p^{(1,2)}_{\tilde{i},(1,1)}=p^{(1,2)}_{(1,1),\tilde{i}}=0$ for $\tilde{i}\in\{(1,2),(2,1)\}$. Lemma \ref{intersection numners} \ref{intersection numners 2} implies
\begin{align}\tag{B.2}\label{121111-0}
p^{(1,1)}_{(1,2),(1,2)}=p^{(1,1)}_{(2,1),(2,1)}=p^{(1,1)}_{(1,2),(2,1)}=p^{(1,1)}_{(2,1),(1,2)}=0.
\end{align}

Since $p^{(1,2)}_{(1,1),(1,1)}\neq0$, from Lemma \ref{intersection numners} \ref{intersection numners 2}, one has $p^{(1,1)}_{(1,1),(1,2)}=p^{(1,1)}_{(1,2),(1,1)}=p^{(1.1)}_{(1,1),(2,1)}=p^{(1.1)}_{(2,1),(1,1)}\neq0$. Let $(x,y)\in\Gamma_{1,2}$ and $z\in P_{(1,1),(1,1)}(x,y)$.
Next, we reach a contradiction step by step.

\begin{stepppp}\label{1}
{\rm Show that $p^{(1,1)}_{(1,1),(1,1)}=0.$}
\end{stepppp}

Suppose not. By setting $\tilde{i}=(1,1)$ in Lemma \ref{tildei} \ref{tildei1}, there exist $i,j\in\{1,2\}$ such that $P_{(1,1),(1,2)}(x,z)\cup P_{(1,2),(1,1)}(x,z)\subseteq Y_i(x,z)$ and $P_{(1,1),(2,1)}(x,z)\cup P_{(2,1),(1,1)}(x,z)\subseteq Y_j(x,z)$.  Since  $p^{(1,1)}_{(1,1),(1,2)}=p^{(1,1)}_{(1,1),(2,1)}\neq0$, from the second statement of Lemma \ref{tildei}, one has $P_{(1,1),(1,1)}(x,z)\subseteq Y_i(x,z)$ and $P_{(1,1),(1,1)}(x,z)\subseteq Y_j(x,z)$. The fact $p^{(1,1)}_{(1,1),(1,1)}\neq0$ implies $i=j$.    Since $\max T=3$, from \eqref{121111-0}, we get $C(x,y)\subseteq Y_i(x,y)$, which implies $Y_j(x,z)=\emptyset$ with $j\in\{1,2\}\setminus\{i\}$. By Lemma \ref{sigma1}, one gets  $e=1$ or $m-e=1$, contrary to the fact that $m-e\geq e\geq2$.

\begin{stepppp}\label{zwuv}
{\rm Show that $P_{(1,2),(1,1)}(x,z)\cup P_{(2,1),(1,1)}(x,z)\subseteq Y_i(x,z)$ and $P_{(1,1),(2,1)}(x,z)\cup P_{(1,1),(1,2)}(x,z)\subseteq Y_j(x,z)$ with $\{i,j\}=\{1,2\}$. } 
\end{stepppp}

Note that $p^{(1,1)}_{(1,2),(1,1)}=p^{(1,1)}_{(1,1),(1,2)}\neq0$. Pick vertices $w_1\in P_{(1,2),(1,1)}(x,z)$ and $w_2\in P_{(1,1),(1,2)}(x,z)$. By setting $(\tilde{i},\tilde{j})\in\{((1,2),(1,1)),((1,1),(1,2))\}$ in Lemma \ref{tildei} \ref{tildei2}, we have  $P_{(2,1),(1,1)}(x,z)\subseteq Y_i(x,z)$  and $P_{(1,1),(2,1)}(x,z)\subseteq Y_j(x,z)$ with $i,j\in\{1,2\}$. Since $p^{(1,1)}_{(2,1),(1,1)}=p^{(1,1)}_{(1,1),(2,1)}\neq0$, from the second statement of Lemma \ref{tildei}, we get $P_{(1,2),(1,1)}(x,z)\subseteq Y_i(x,z)$  and $P_{(1,1),(1,2)}(x,z)\subseteq Y_j(x,z)$.   Then $P_{(1,2),(1,1)}(x,z)\cup P_{(2,1),(1,1)}(x,z)\subseteq Y_i(x,z)$ and $P_{(1,1),(2,1)}(x,z)\cup P_{(1,1),(1,2)}(x,z)\subseteq Y_j(x,z)$. If $i=j$, by Step \ref{1} and \eqref{121111-0}, then $Y_k(x,z)=\emptyset$ with $k\in\{1,2\}\setminus\{i\}$, a contradiction. Thus, $i\neq j$. 

\begin{stepppp}\label{122121}
{\rm Show that $p^{(2,2)}_{(1,2),(1,2)}=0$ and there exist $w_1,z_1\in P_{(2,1),(2,1)}(x,y)$ with $(w_1,z_1)\in\Sigma_2$.}
\end{stepppp}

Since  $p_{(2,1),(1,1)}^{(1,1)}\neq0$ and $y\in P_{(1,2),(1,1)}(x,z)$, from Step \ref{zwuv} and Lemma \ref{sigma1} \ref{sigma11}, we have $(w,y)\in\Sigma_1$ for all $w\in P_{(2,1),(1,1)}(x,z)$.
Since $\max T=3$ and $z\in P_{(1,1),(1,1)}(w,y)$ for all $w\in P_{(2,1),(1,1)}(x,z)$, from Step \ref{1}, we obtain $(w,y)\in\Gamma_{1,2}$ or $\Gamma_{2,1}$. If $(w,y)\in\Gamma_{1,2}$ for all $w\in P_{(2,1),(1,1)}(x,z)$, then $w\in P_{(1,1),(1,2)}(z,y)$, which implies $P_{(1,1),(1,2)}(z,y)\supseteq\{x\}\cup P_{(1,1),(1,2)}(z,x)$, a contradiction. Then there exists $w_1\in P_{(2,1),(1,1)}(x,z)$ such that $(w_1,y)\in\Gamma_{2,1}$. Thus, $w_1\in P_{(2,1),(2,1)}(x,y)$.

Let $u\in P_{(1,1),(2,1)}(x,z)$. Since $y\in P_{(1,2),(1,1)}(x,z)$, from Step \ref{zwuv}, one gets $y\in Y_i(x,z)$ and $u\in Y_j(x,z)$ with $\{i,j\}=\{1,2\}$. Lemma \ref{sigma1} \ref{sigma12} implies $(u,y)\in\Sigma_2$. Since $x\in P_{(2,1),(1,1)}(y,u)$ and $z\in P_{(1,1),(1,2)}(y,u)$, one gets $(y,u)\in\Gamma_{2,2}$. Then Lemma \ref{22} \ref{223} holds. Since $c_2=4$, from \eqref{a1c2}, we have $p^{(2,2)}_{(1,2),(1,2)}=0$.

Since $c_2=4$, from Lemma \ref{22} \ref{223}, there exists $$z_1\in P_{(1,2),(1,1)}(y,u)~\mbox{and}~z_2\in P_{(1,1),(2,1)}(y,u)$$ such that $C(y,u)=\{x,z,z_1,z_2\}$. Lemma \ref{sigma2} implies $(x,z_1)\in\Sigma_2$ and $C(x,z_1)=\{y,u,z,z_2\}$, or $(x,z_2)\in\Sigma_2$ and $C(x,z_2)=\{y,u,z,z_1\}$.
If $(x,z_1)\in\Sigma_2$ and $C(x,z_1)=\{y,u,z,z_2\}$, by $u\in P_{(1,1),(1,1)}(x,z_1)$, then $(x,z_1)\in\Gamma_{2,2}$, contrary to Lemma \ref{22} \ref{223}. Thus, $(x,z_2)\in\Sigma_2$ and $C(x,z_2)=\{y,u,z,z_1\}$. Since $(u,y)\in\Sigma_2$, from Lemma \ref{sigma2}, we have $(z,z_1)\in\Sigma_2$.

Since $z,w_1\in C(x,y)$ with $(z,w_1)\in\Sigma_1$ and $z_1\in C(x,y)$ with $(z,z_1)\in\Sigma_2$, from Lemma \ref{sigma1}, we have $z,w_1\in Y_{i'}(x,y)$ and $z_1\in Y_{j'}(x,y)$ with $\{i',j'\}=\{1,2\}$. Lemma \ref{sigma1} \ref{sigma12} implies $(w_1,z_1)\in\Sigma_2$.

Note that $(y,z_2)\in\Gamma_{1,1}$. By \eqref{121111-0}, we have $z_1\notin P_{(1,2),\tilde{i}}(y,z_2)$ for $\tilde{i}\in\{(1,2),(2,1)\}$. Since $z_1\in C(x,z_2)$ and $\max T=3$, we get $(z_1,z_2)\in\Gamma_{1,1}$. Since $y\in P_{(1,2),(1,2)}(x,z_1)$, by \eqref{121111-0}, one obtains $(x,z_1)\in\Gamma_{1,2}$ or $\Gamma_{2,1}$.

Suppose $(x,z_1)\in\Gamma_{1,2}$. Since $C(x,z_2)=\{y,u,z,z_1\}$ and $y,z_1\in P_{(1,2),(1,1)}(x,z_2)$, we have $z,u\in P_{(1,1),(1,2)}(x,z_2)$. Lemma \ref{22} \ref{223} implies $(x,z_2)\in\Gamma_{2,3}$. Since $c_2=4$, from \eqref{a1c2}, we get
\begin{align}\tag{B.3}\label{121111gongshi}
p^{(2,3)}_{(1,2),(1,1)}=p^{(2,3)}_{(1,1),(1,2)}=2.
\end{align}
Since $x\in P_{(1,2),(1,2)}(w_1,z_1)$ and $y\in P_{(2,1),(1,2)}(w_1,z_1)$, from Lemma \ref{22} \ref{223}, one gets $(w_1,z_1)\in\Gamma_{2,3}$, and so $p^{(2,3)}_{(1,2),(2,1)}\neq0$. \eqref{a1c2} and \eqref{121111gongshi} imply $c_2\geq5$, a contradiction. Then $(x,z_1)\in\Gamma_{2,1}$, and so $z_1\in P_{(2,1),(2,1)}(x,y)$. Thus, the desired result follows.

\begin{stepppp}
{\rm We reach a contradiction.}
\end{stepppp}

By Step \ref{122121}, there exist $w_1,z_1\in P_{(2,1),(2,1)}(x,y)$ with $(w_1,z_1)\in\Sigma_2$. Let $w_1\in Y_i(x,y)$ and $z_1\in Y_j(x,y)$ with $\{i,j\}=\{1,2\}$. Suppose $p^{(1,2)}_{(1,2),(1,2)}\neq0$. Then there exists $u\in P_{(1,2),(1,2)}(x,y)$. Since $u\in Y_i(x,y)$ or $Y_j(x,y)$, from Lemma \ref{sigma1} \ref{sigma12}, we have $(u,z_1)\in\Sigma_2$ or $(u,w_1)\in\Sigma_2$. Since $(u,y,z_1,x)$ and $(u,y,w_1,x)$ are circuits consisting of arcs of type $(1,2)$, one gets $(u,z_1)$ or $(u,w_1)\in\Gamma_{2,2}$, and so $p_{(1,2),(1,2)}^{(2,2)}\neq0$, contrary to Step \ref{122121}. Thus, $p^{(1,2)}_{(1,2),(1,2)}=0.$ Lemma \ref{intersection numners} \ref{intersection numners 2} implies $p^{(1,2)}_{(1,2),(2,1)}=p^{(2,1)}_{(1,2),(2,1)}=0$.

Suppose that there exist distinct vertices $u_1,u_2\in P_{(2,1),(1,1)}(x,z)$. Step \ref{zwuv} implies $(u_1,u_2)\in\Sigma_1$. Since $z\in P_{(1,1),(1,1)}(u_1,u_2)$, from Step \ref{1}, we have $(u_1,u_2)\in\Gamma_{1,2}$ or $\Gamma_{2,1}$. Since $x\in P_{(1,2),(2,1)}(u_1,u_2)$, we get    $p^{(1,2)}_{(1,2),(2,1)}\neq0$ or $p^{(2,1)}_{(1,2),(2,1)}\neq0$, a contradiction. Then $|P_{(2,1),(1,1)}(x,z)|\leq1$.  Since $p^{(1,2)}_{(1,1),(1,1)}\neq0$, from Lemma \ref{intersection numners} \ref{intersection numners 2}, one has $p^{(1,1)}_{(2,1),(1,1)}=1$.  By Lemma \ref{intersection numners} \ref{intersection numners 2} again, we get $|P_{(2,1),(1,1)}(x,z)|=|P_{(1,2),(1,1)}(x,z)|=|P_{(1,1),(1,2)}(x,z)|=|P_{(1,1),(2,1)}(x,z)|=1$.

In view of  Step \ref{1} and \eqref{121111-0}, we obtain $p^{(1,1)}_{\tilde{i},\tilde{j}}\neq0$ for  $\tilde{i},\tilde{j}\in\{(1,1),(1,2),(2,1)\}$ if and only if $(\tilde{i},\tilde{j})\in\{((1,1),\tilde{h}),(\tilde{h},(1,1))\mid \tilde{h}\in\{(1,2),(2,1)\}\}$.
By  Step \ref{zwuv} , we have $|Y_i(x,z)|=2$ for all $i\in\{1,2\}$. Lemma \ref{sigma1} implies $e=m-e=3$. Then $\Sigma$ is a Johnson graph, and so $m=6$ and  $|V(\Gamma)|=20$, contrary to \cite{H}.

This completes the proof of the lemma.
\end{proof}

\end{document}